%% file: main.tex
\documentclass[twoside,11pt]{article}

\input{packages}
\input{commands}

\jmlrheading{24}{2026}{1-\pageref{LastPage}}{03/26}{-}{}{Florian Heinrichs}

\ShortHeadings{Testing for Stable Intervals in Non-Stationary Time Series}{Heinrichs}
\firstpageno{1}

\begin{document}
	
\title{Testing for Stable Intervals in Non-Stationary Time Series}

\author{\name Florian Heinrichs \email f.heinrichs@fh-aachen.de \\
	\addr FH Aachen\\
	Heinrich-Mußmann-Straße 1\\
	52428 Jülich, Germany}

\editor{-}

\maketitle

\begin{abstract}%
	Many time series are not stable over their full observation horizon, but may contain scientifically meaningful periods during which a signal remains stable up to a prescribed tolerance. We formulate this as an existence test for stable intervals in a non-stationary regression model with dependent, locally stationary errors. For a signal $d$ derived from the mean function, including deviations from reference levels and derivatives, stability over duration $\delta$ is encoded by $d_\infty=\inf_{t\in[0,1-\delta]}\sup_{s\in[t,t+\delta]}|d(s)|$.
	The hypothesis $d_\infty\ge\Delta$ states that no interval of length $\delta$ remains within the tolerance $\Delta$, while rejection provides evidence for the existence of a relevant stable period. We estimate $d$ by local linear regression and construct plug-in tests for $d_\infty$.
	The asymptotic distribution is determined only by near-extremal windows at which the minimax functional is attained. We formalize this localization through extremal sets and derive Gaussian and extreme value approximations for kernel estimators over possibly shrinking index sets with time-varying long-run variance. The resulting tests are consistent and have asymptotic level control. The theory also extends supremum-based relevant-change tests to time-varying long-run variance and derivative-based hypotheses. Simulations and applications to physiological and engineering time series illustrate the method. 
\end{abstract}


%

\maketitle 

\input{sec_intro}

\input{sec_methodology}

\input{sec_empirical_results}

\input{sec_proofs}

\bibliography{bibliography}

\newpage

\appendix

\input{app_empirical_results}

\end{document}

%% file: packages.tex
\usepackage{blindtext}

\usepackage[preprint]{jmlr2e}

\usepackage[utf8]{inputenc}
\usepackage[T1]{fontenc}
\usepackage{lmodern}
\usepackage{amsmath}
\usepackage{color}
\usepackage{bm}
\usepackage{booktabs}
\usepackage{dsfont}
\usepackage{hyperref}
\usepackage{setspace}

%% file: commands.tex
\newcommand{\eps}{{\varepsilon}}
\renewcommand{\phi}{\varphi}
\newcommand{\R}{\mathbb{R}}
\newcommand{\Z}{\mathbb{Z}}
\newcommand{\N}{\mathbb{N}}
\newcommand{\pr}{\mathbb{P}}       
\newcommand{\ex}{\mathbb{E}}       
\newcommand{\var}{\textnormal{Var}} 
\newcommand{\cov}{\textnormal{Cov}}

\newcommand{\Nc}{\mathcal{N}}
\newcommand{\Uc}{\mathcal{U}}
\newcommand{\Fc}{\mathcal{F}}

\newcommand{\Kc}{\mathcal{K}}
\newcommand{\Dc}{\mathcal{D}}
\newcommand{\Ac}{\mathcal{A}}
\newcommand{\Bc}{\mathcal{B}}
\newcommand{\Ec}{\mathcal{E}}
\newcommand{\Oc}{\mathcal{O}}
\newcommand{\Pc}{\mathcal{P}}
\newcommand{\Mc}{\mathcal{M}}
\newcommand{\Tc}{\mathcal{T}}

\newcommand{\diff}{{\,\mathrm{d}}}

\newcommand{\convw}{\rightsquigarrow}                           
\newcommand{\convp}{\stackrel{\pr}{\longrightarrow}} 
   
\newcommand{\id}{\mathds{1}}

\DeclareMathOperator*{\argmin}{argmin} 
 
\DeclareMathOperator*{\sgn}{sgn}
\DeclareMathOperator*{\dist}{dist}

\newtheorem{assumption}{Assumption}

%% file: sec_intro.tex
\section{Introduction} \label{sec:intro}

Many time series are collected from processes that are not expected to be stable throughout. A sensor may drift before reaching a reliable operating regime, a drug concentration may rise before entering a therapeutic range, and the validation loss of a learning algorithm initially decreases before settling near a plateau \citep{cui2019, fan2014, amari1995, prechelt1998}. In such examples, the question is not whether the whole trajectory is constant. The question is whether there exists a period, long enough to be useful, during which the signal is stable up to a scientifically meaningful tolerance.

We study this problem in the model
\begin{equation*}
	X_{i,n} = \mu(i/n) + \eps_{i,n}, \qquad i=1,\dots,n,
\end{equation*}
where $\mu$ is an unknown smooth mean function and the errors may be dependent and locally stationary. The object of interest is a stability signal $d$ derived from $\mu$. For instance, $d=\mu'$ measures local flatness, $d(t)=\mu(t)-g(\mu)$ measures deviation from a fixed or estimated reference level, and $d(t)=\mu(t)-g(\mu,t)$ allows time-dependent or locally chosen reference levels. In the latter cases, $g$ may be known or estimated, and we assume that a suitable estimator is available with sufficiently fast convergence. 

For a minimal duration $\delta \ge 0$ define
\begin{equation*}
	d_\infty = \inf_{t\in[0,1-\delta]} \sup_{s\in[t,t+\delta]} |d(s)|.
\end{equation*}
The inner supremum asks how far the signal $d$ can deviate from zero within a candidate interval $[t,t +\delta]$. The outer infimum then searches over all possible starting points and selects the best such interval. Thus $d_\infty < \Delta$, for some tolerance $\Delta > 0$, means that there exists at least one interval of length $\delta$ on which the signal remains within the prescribed bound $\Delta$. Figure \ref{fig:initial_example} illustrates the construction. For each candidate starting point $t$, the largest deviation of $d$ over the window $[t, t+\delta]$ is recorded, and $d_\infty$ is the smallest of these window-wise deviations. We therefore test
\begin{equation*}
	H_0: d_\infty \ge \Delta \quad\text{vs.}\quad H_1: d_\infty < \Delta,
\end{equation*}
so that rejecting $H_0$ provides evidence for the existence of a relevant stable period. 

\begin{figure}[t]
	\centering
	\includegraphics[width=\textwidth]{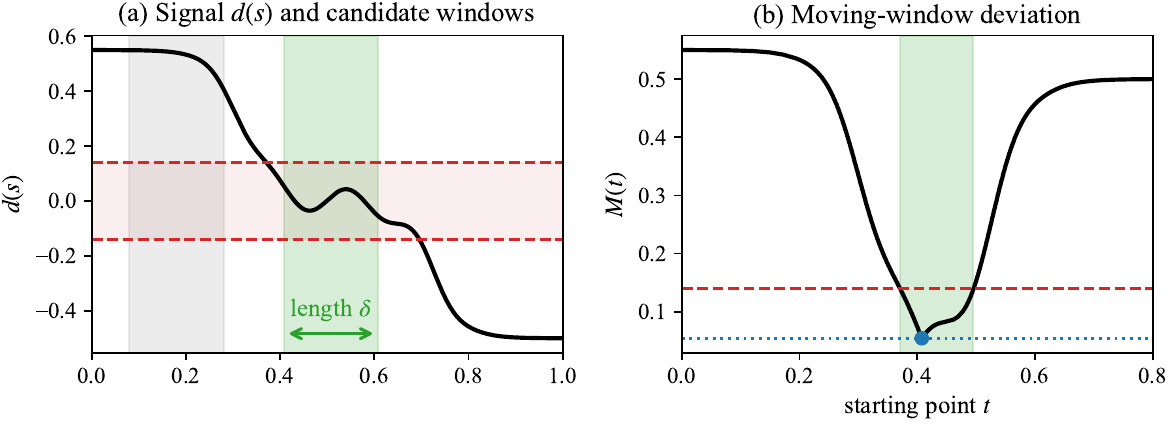}
	\caption{Stable-interval functional. Panel (a) shows a signal $d$ with tolerance band $[-\Delta,\Delta]$. The green window has length $\delta$ and remains inside the tolerance band, whereas the gray window does not. Panel (b) shows the corresponding moving-window deviation $M(t):=\sup_{s\in[t,t+\delta]}|d(s)|$. The stable-interval deviation is $d_\infty=\inf_t M(t)$. Since $d_\infty<\Delta$, at least one interval of length $\delta$ is stable.}
	\label{fig:initial_example}
\end{figure}

This testing problem is different from the usual question of detecting changes in a non-stationary mean. In relevant-change testing, one often asks whether the curve ever deviates from a reference by more than $\Delta$. This naturally leads to a global maximum such as
\begin{equation*}
	\sup_{t\in[0,1]}|\mu(t)-g(\mu)|,
\end{equation*}
because a single large deviation anywhere in time is enough to make the change relevant. Stable-interval detection reverses the logic. We are not looking for the worst time point, but for the best interval. Hence we first measure the worst deviation inside each candidate interval and then minimize this quantity over all intervals. This produces the inf-sup functional above and makes the problem genuinely different from testing for no change or for a relevant global deviation.

Existing methods only partially address this problem. Procedures for steady-state detection target similar scientific questions, but commonly rely on independent, identically distributed errors or do not provide theoretical guarantees \citep{cao1995,kelly2013}. Change point methods allow richer dependence structures, but they usually search for abrupt changes and impose piecewise stationary models \citep{basseville1993,fryzlewicz2014,frick2014,cho2015}. Locally stationary time series methods allow dependence and time-varying behavior, but existing tests mainly address whether changes occur somewhere or whether deviations are relevant \citep{vogt2015, dette2019,heinrichs2021,mies2023}. To the best of our knowledge, there is no general test for the existence of a relevant stable period under dependent and non-stationary errors.

Our contribution is to develop such a test. We estimate $d$ by local linear regression and use the resulting estimator to construct a test statistic based on $\hat d_{n,\infty}$, the plug-in estimator of $d_\infty$. The calibration accounts for serial dependence through a local long-run variance $\sigma^2(t)$, which may vary over time. A key feature of the theory is that the asymptotic behavior is determined only by points near the intervals where the minimum in $d_\infty$ is attained, rather than by the full time domain. We formalize this through an extremal set and a near-extremal set, and derive an extreme value approximation over these sets. The resulting test has asymptotic level $\alpha$ under $H_0: d_\infty \ge \Delta$ and is consistent against alternatives with $d_\infty < \Delta$.
Since extreme value approximations may converge slowly, we also discuss a Gaussian calibration based on the finite-sample Gaussian approximation underlying the limit theory.

The technical contribution is an extreme value theory for kernel-based estimators over possibly shrinking near-extremal sets with time-varying long-run variance. This theory covers both extremal intervals, which arise when a stable period is attained over a non-degenerate region, and isolated extremal points, which arise under tangential contact with the tolerance boundary. The general approximation results are formulated separately, since they may be useful beyond the stable-interval problem. As a by-product, the same arguments yield an extension of the sup-norm relevant-change tests by \cite{bucher2021} to settings with time-varying long-run variance and to hypotheses based on the derivative $\mu'$. 

The remainder of the paper is organized as follows. Section \ref{sec:method} introduces the stable-interval functional and relevant estimators. Section \ref{sec:localized_test} gives the test statistic and main inference results for stable-interval detection. Section \ref{sec:gaussian_evt} develops the Gaussian and extreme value approximations on which the test is based, and discusses the connection to relevant-change testing. Section \ref{sec:empirical} presents simulations and two data examples, and the proofs are collected in Section \ref{sec:proof}.

%% file: sec_methodology.tex
\section{Stable-interval Target and Estimation}\label{sec:method}


We now introduce the stable-interval target and its estimator. Throughout, the observed time series is modeled as
\begin{equation}\label{eq:additive_model}
	X_{i, n} = \mu(i/n) + \eps_{i, n}, \qquad i = 1,\dots, n,
\end{equation}
where $\mu$ is a deterministic mean function and $\eps = (\eps_{i, n})_{1\le i\le n, n\in\N}$ is a triangular array of centered errors. The mean function is sufficiently smooth to be reasonably estimated, as specified in Assumption \ref{assump:mu}, and the error process is locally stationary, as described by Assumption \ref{assump:error}.

\subsection{Stable-interval Functional and Hypotheses}

The target of inference is not necessarily the mean function itself, but a signal $d:[0,1]\to\R$ whose smallness characterizes stability. 
Two leading examples are
\begin{equation*}
	d(t) = \mu(t) - g(\mu,t) \qquad \mathrm{and} \qquad d(t) = \mu'(t).
\end{equation*}
The first choice describes stability around a prescribed, time-varying, or estimated reference level, while the second choice describes level-free plateau detection through local flatness. The map $g$ may be known, for example a deterministic reference curve, or may be a functional of $\mu$ that has to be estimated. The assumptions below only require that the resulting estimator of $d$ admits the stochastic expansion specified in Section~\ref{subsec:localized_sets}.

Fix a minimal duration $\delta\in[0,1]$ and a tolerance $\Delta>0$. 
We say that the signal contains a stable interval of length at least $\delta$ and tolerance $\Delta$ if there exists an interval of length $\delta$ on which $|d|$ remains below $\Delta$. 
This is captured by the \textit{minimax deviation}
\begin{equation}\label{eq:def_d}
	d_\infty = \inf_{t\in[0, 1-\delta]} \sup_{s\in[t, t+\delta]} |d(s)|.
\end{equation}
The inner supremum measures the largest deviation within a candidate interval, whereas the outer infimum selects the best such interval. 
Thus, $d_\infty < \Delta$ is equivalent to the existence of at least one interval of length $\delta$ on which the signal remains inside the tolerance band $[-\Delta,\Delta]$, such that $d$ is ``almost constant''. 
We therefore consider the one-sided testing problem
\begin{equation*}
	H_0: d_\infty \ge  \Delta \quad \mathrm{vs.} \quad H_1: d_\infty < \Delta.
\end{equation*}
The null hypothesis says that no stable interval of length $\delta$ exists at tolerance $\Delta$, while rejection provides evidence for the existence of such an interval.

\subsection{Estimation of the Signal}
\label{subsec:signal_estimation}

Inference on $d_\infty$ is based on a plug-in estimator of the signal $d$. Let $K$ be a kernel and write $K_h(x)=K(x/h)$. For a bandwidth $h_n \to 0$, the local linear estimator of $\mu$ and its first derivative is defined by
\begin{equation*}
	\big(\hat{\mu}_{h_n}(t), \widehat{\mu'}_{h_n}(t)\big) = \argmin_{b_0, b_1\in\R} \sum_{i=1}^n \big[X_{i, n} - b_0 - b_1 \big(\tfrac{i}{n} - t\big)\big]^2 K_{h_n} \big(\tfrac{i}{n} - t\big),
\end{equation*}
see, for example, \cite{fan1996}. 
For the level case, we use the jackknife bias-reduced estimator
\begin{equation*}
	\tilde{\mu}_{h_n}(t) = 2 \hat{\mu}_{h_n/\sqrt{2}}(t) - \hat{\mu}_{h_n}(t),
\end{equation*}
which removes the leading second-order bias term under the smoothness assumptions imposed below \citep{schucany1977}. For the derivative-based plateau problem $d = \mu'$, we use the estimator $\hat{d}_n = \widehat{\mu'}_{h_n}$. For the level-stability case $d = \mu - g(\mu, \cdot)$, we can either plug $\tilde{\mu}_{h_n}$ into $g$ or directly estimate $g(\mu, \cdot)$, which often results in a faster rate of convergence. Since the particular estimator depends on the specific choice of $g(\mu, \cdot)$, we only assume the existence of an estimator $\hat{g}_n$ that converges sufficiently fast. In this case, $\hat{d}_n = \tilde{\mu}_{h_n} - \hat{g}_n$. 
The stable-interval deviation is then estimated by the plug-in statistic $\hat{d}_{n, \infty} = \inf_{t\in[0, 1-\delta]} \sup_{s\in[t, t+\delta]} |\hat{d}_n(s)|$. The estimator is the plug-in analogue of the minimax deviation in \eqref{eq:def_d}.

For the asymptotic theory we exclude boundary-attainment effects. Specifically, we assume that $\delta < 1$ and there exists $\eta > 0$ such that all minimizers of the moving-window deviation $M(t)=\sup_{s\in[t,t+\delta]} |d(s)|$ are contained in $[\eta,1-\delta-\eta]$. The empirical statistic is evaluated on the trimmed set $I_n=[h_n,1-\delta-h_n]$, that is, $\hat d_{n,\infty}^{\mathrm{tr}} = \inf_{t\in I_n} \sup_{s\in[t,t+\delta]} |\hat d_n(s)|$. Since $h_n \to 0$, the trimming does not affect the population minimax deviation for all sufficiently large $n$ under the interior-attainment condition. To simplify notation, the superscript is suppressed in the sequel.

For the asymptotic theory it is useful to write the two cases in a common notation. Let $r_n$ denote the pointwise rate of $\hat d_n$, and let $K^*$ denote the corresponding equivalent kernel. In the level case, $r_n=\sqrt{nh_n}$ and $K^*(x)=2\sqrt 2 K(\sqrt 2 x)-K(x)$, where $K^*$ is the equivalent kernel of the jackknife estimator. In the derivative case, $r_n=\sqrt{nh_n^3}$ and $K^*(x)=xK(x) / \int y^2 K(y) \diff y$. 
All normalizations below are written in terms of $r_n$ and $K^*$. In the level case this reduces to the familiar $\sqrt{nh_n}/\|K^*\|_2$ scaling used in kernel supremum theory. 
We impose the following condition on the kernel.

\begin{assumption} \label{assump:kern}
	The kernel $K:\R\to \R$ is symmetric, supported on $[-1,1]$, twice differentiable and satisfies $\int_{-1}^1 K(x)\diff x = 1$. Moreover, the equivalent kernel $K^*$ satisfies $\|K^*\|_2> 0$ and $\| (K^*)'\|_2^2 / \|K^*\|_2^2 < 32$.
\end{assumption}

\subsection{Locally Stationary Errors and Long-run Variance}
The errors in \eqref{eq:additive_model} may be serially dependent and non-stationary. 
We use the locally stationary framework of \cite{zhou2009}. 
Specifically,
\begin{equation*}
	\eps_{i, n} = H(i/n, \Fc_i), \qquad \Fc_i = (\eta_k)_{k\le i}, 
\end{equation*}
where $(\eta_i)_{i\in\Z}$ is an i.i.d. innovation sequence and $H:[0, 1] \times \R^\infty \to \R$ is a measurable (possibly non-linear) filter that is continuous in its first argument. Let $\Fc_i^*$ denote the coupled version of $\Fc_i$ obtained by replacing $\eta_0$ with an independent copy $\eta_0^*$. 
The \textit{physical dependence measure} of $H$ with $\sup_{t\in[0, 1]} \ex[H^2(t, \Fc_i)] < \infty$ is defined by 
\begin{equation*}
	\delta(H, i) = \sup_{t\in[0, 1]} \ex\big[\big(H(t, \Fc_i) - H(t, \Fc_i^*)\big)^2\big]^{1/2}.
\end{equation*}
The quantity $\delta(H, i)$ measures the strength of serial dependence $H(t, \Fc_i)$ and plays a similar role as mixing coefficients. Indeed, under suitable regularity conditions, one may replace weak physical dependence by $\beta$-mixing, or vice-versa \citep{hill2025,heinrichs2026}.

\begin{assumption} \label{assump:error}
	The triangular array $\{(\eps_{i, n})_{1\le i\le n}\}_{n\in\N}$ in \eqref{eq:additive_model} is centered and locally stationary with map $H$, such that the following conditions are satisfied:
	\begin{enumerate}
		\item $\Theta_m = \sum_{i=m}^\infty \delta(H, i)$ converges to zero as $m \to\infty$.
		\item The map $H$ is Lipschitz continuous with respect to the $L^2$-norm, i.\,e.,
		\begin{equation*}
			\sup_{0 \le s < t \le 1} \ex \big[ \big( H(t, \Fc_i) - H(s, \Fc_i) \big)^2\big]^{1/2} / |t - s| < \infty.
		\end{equation*}		
		and fourth moments are uniformly bounded, so that $\sup_{t\in[0, 1]} \ex[H^4(t, \Fc_0)] < \infty$.
		\item The (local) long-run variance of $H$, defined as
		\begin{equation*}
			\sigma^2(t) = \sum_{i=-\infty}^\infty \cov\big(H(t, \Fc_i), H(t, \Fc_0)\big),
		\end{equation*}
		exists for all $t\in[0, 1]$, is Lipschitz continuous and bounded away from zero.
	\end{enumerate}
\end{assumption}
The assumption is rather standard in the context of local stationarity \citep[see, e.\,g.,][]{zhou2009}.

The calibration of the test requires estimating the time-varying long-run variance. 
We use the difference-based local estimator of \cite{dette2019}. 
Let $\tau_n \to 0$ be a positive bandwidth, $m_n \to\infty$ a positive integer sequence with $m_n =o(n)$, and define partial sums $S_{j, k} = \sum_{i=j}^{k} X_{i, n}$. For $t\in [m_n/n, 1 - m_n/n]$ set
\begin{equation*}
	\hat{\sigma}_n^2(t) = \sum_{i=1}^n \frac{K_{\tau_n}(i/n - t)}{\sum_{j=1}^{n}K_{\tau_n}(j/n - t)} \cdot \frac{(S_{i-m_n+1,i} - S_{i+1, i+m_n})^2}{2m_n},
\end{equation*}
and extend the estimator at the boundaries by $\hat{\sigma}_n^2(t) = \hat{\sigma}_n^2(m_n/n)$, for $t < m_n/n$, and $\hat{\sigma}_n^2(t) = \hat{\sigma}_n^2(1-m_n/n)$, for $t > 1-m_n/n$. By Theorem 6.4 of \cite{bucher2021}, the prediction error is uniformly of order $\Oc_\pr\Big(\tfrac{m_n^{1/4}}{\sqrt{n} \tau_n} + \tfrac{1}{m_n} + \tau_n^2 + \tfrac{m_n^{5/2}}{n}\Big)$. 
This error is minimized whenever $m_n = n^{2/7}$ and $\tau = n^{-1/7}$, in which case
\begin{equation} \label{eq:rate_sigma}
	\sup_{t\in [\gamma_n, 1-\gamma_n]} |\sigma^2(t) - \hat{\sigma}_n^2(t)| 
	= \Oc_\pr(n^{-2/7}),
\end{equation}
for $\gamma_n = \tau_n + m_n / n$.

\section{Extremal Localization and the Stable-interval Test}
\label{sec:localized_test}
The plug-in statistic $\hat{d}_{n, \infty}$ is not governed by the full trajectory of $\hat{d}_n$. Its asymptotic behavior is determined only by points close to the minimax deviation $d_\infty$. This section formalizes this localization and derives the resulting test.

\subsection{Extremal and Near-extremal Sets}

Let $\Ec_1 = \Ec_1^+ \cup \Ec_1^-$, where
\begin{equation*}
	\Ec_1^\pm = \{t\in[0, 1] : \pm d(t) = d_\infty\}.
\end{equation*}
The set $\Ec_1$ records the time points at which the signed signal reaches the level relevant for the minimax deviation. 
To encode the window structure, define
\begin{equation*}
	\Ec_3 = \Big\{ t\in[0, 1-\delta]: \sup_{s\in[t, t+\delta]} |d(s)| = d_\infty \Big\},
\end{equation*}
and $\Ec_2 = \Ec_3 + [0, \delta]$. The extremal set is $\Ec = \Ec_1 \cap \Ec_2$. Thus, $\Ec$ consists of those points which both attain the critical signed level and belong to at least one optimal window. 
With this notation, we can eliminate the inner supremum and absolute value, and write
\begin{equation*}
	d_\infty = \inf_{t\in\Ec} |d(t)| = \min\{ \inf_{t \in \Ec_1^+ \cap \Ec_2} d(t),  \inf_{t \in \Ec_1^- \cap \Ec_2} - d(t) \}.
\end{equation*}
Because $\hat d_n$ is random, the asymptotic behavior of $\hat d_{n,\infty}$ is determined not only by $\Ec$, but also by shrinking neighborhoods of $\Ec$. 
For a positive sequence $\phi_n \to 0$, define
$\Ec(\phi_n) = \Ec_1(\phi_n) \cap \Ec_2(\phi_n)$, for 
\begin{equation*}
	\Ec_1(\phi_n) = \Ec_1^+(\phi_n) \cup \Ec_1^-(\phi_n), \qquad \Ec_1^\pm(\phi_n) = \{t\in[0, 1] : |d_\infty \mp d(t) | \le \phi_n\}
\end{equation*}
and
\begin{equation*}
	\Ec_2(\phi_n) = \Ec_3(\phi_n) + [0, \delta], \quad \Ec_3(\phi_n) = \{ t\in[0, 1-\delta]: \sup_{s\in[t, t+\delta]} |d(s)| - d_\infty  \le \phi_n \}.
\end{equation*}
The following regularity condition controls the geometry of $\Ec(\phi_n)$ as $\phi_n \to 0$. 
For $x\in[0,1]$, $\eta>0$, write $U_\eta^1(x) = (x, x+\eta)$ and $U_\eta^{-1}(x) = (x-\eta, x)$, and let $(\partial \Ec)^{ext}$ denote the exterior boundary sides,
\begin{equation*}
	(\partial \Ec)^{ext} = \{(x, \nu) \in \partial \Ec \times \{-1, 1\}| \exists U_\eta^\nu(x): U_\eta^\nu(x) \cap \Ec = \emptyset\}.
\end{equation*}

\begin{assumption} \label{assump:mu}
	\begin{enumerate}
		\item The function $\mu$ is twice differentiable with Lipschitz continuous second derivative. Moreover, it exists some $\eta > 0$ such that all minimizers of the moving-window deviation $M(t)=\sup_{s\in[t,t+\delta]} |d(s)|$ are contained in $[\eta, 1-\delta-\eta]$.
		\item The extremal set $\Ec$ is a finite union of closed intervals and isolated points with $\lambda(\Ec) > 0$.  Moreover, there exist constants $0 < c < \infty$ and $\gamma > 0$, such that for every $(x, \nu) \in (\partial\Ec)^{ext}$
		\begin{equation} \label{eq:boundary_condition}
			c |t-x|^{\alpha_{x,\nu}} \le \big| d_\infty-|d(t)|\big|,
		\end{equation}
		for some exponent $\alpha_{x,\nu}\in(0,\infty)$ and all $t\in U_\gamma^\nu(x)$.
		\item The extremal set $\Ec$ is a finite union of closed intervals and isolated points. 
		There exist constants $0 < c < C < \infty$ and $\phi_0, \gamma > 0$, such that for every $(x, \nu) \in (\partial\Ec)^{ext}$ and all $0<\phi\le\phi_0$,
		\begin{equation}
			\Ec(\phi)\cap U_\gamma^\nu(x) = U_{r_{x,\nu}(\phi)}^\nu(x), \label{eq:boundary_condition2}
		\end{equation}
		where $r_{x,\nu}$ is continuously differentiable on $(0,\phi_0]$ and satisfies
		\begin{equation*}
			c \phi^{1/\alpha_{x,\nu}} \le r_{x,\nu}(\phi) \le C \phi^{1/\alpha_{x,\nu}},
			\qquad 0 \le r_{x,\nu}'(\phi) \le C \phi^{1/\alpha_{x,\nu} - 1}.
		\end{equation*}
	\end{enumerate}
\end{assumption}
The global smoothness of $\mu$ in Assumption \ref{assump:mu} and the Lipschitz continuity of $\sigma^2$ in Assumption \ref{assump:error} are imposed mainly to keep the notation and proofs transparent. The results should extend to piecewise smooth $\mu$ and piecewise Lipschitz $\sigma^2$, provided the break points are separated from the relevant near-extremal sets, or are treated as additional boundary points in the localization and block arguments. We do not pursue this extension here.

Due to the boundary condition in part 1 of the assumption and Proposition \ref{prop:extremal_set}, $\Ec(\phi_n) \subset [h_n, 1-h_n]$ for all large $n$.

In the derivative case $d = \mu'$, such that $\hat{d}_n$ only depends on the local linear estimator. Contrarily, in the level case, $d= \mu - g(\mu, \cdot)$, such that $\hat{d}_n$ depends on an estimator $\hat{g}_n$.

\begin{assumption} \label{assump:g_estimator}
	The estimator $\hat{g}_n$ of the functional $g(\mu, \cdot)$ satisfies
	\begin{equation*}
		\sup_{t\in[0, 1]} |\hat{g}_n(t) - g(\mu, t)| = o_\pr\bigg(\frac{1}{\sqrt{nh_n |\log(h_n)|}}\bigg), \qquad n\to\infty,
	\end{equation*}
	where $h_n$ denotes the bandwidth of the local linear estimator.
\end{assumption}

Since $\Ec$ is a finite union of intervals and isolated points, it can be expressed as
\begin{equation*}
	\Ec = \bigcup_{i=1}^{L} [a_i, b_i],
\end{equation*}
for $a_i \le b_i$, $i=1,\dots, L$. Crucially, we have two different regimes. If $\Ec$ has a positive length, its estimation substantially simplifies and rates between different sequences become less restrictive. In this case, we only require Assumption \ref{assump:mu} (2). Contrarily, Assumption \ref{assump:mu} (3) is stronger, but allows the extremal set to consist only of isolated points. Under either condition, by Proposition \ref{prop:extremal_set},
\begin{equation*}
	\Ec(\phi_n) = \bigcup_{i=1}^{L} [a_{i, n}, b_{i, n}] 
\end{equation*}
with $U_{\phi_n/L_d}(\Ec) \subset \Ec(\phi_n)$ for all large $n$, where $L_d$ denotes the Lipschitz constant of $d$.

The latter inclusion implies that the shortest component of $\Ec(\phi_n)$ is at least of order $\phi_n$. More precisely, since $\Ec$ has only finitely many components, there is a constant $c_\theta>0$ such that, for all sufficiently large $n$,
\begin{equation} \label{eq:def_theta}
	\theta_n:=\min_i |b_{i,n}-a_{i,n}|\ge c_\theta\phi_n.
\end{equation}
Thus the block condition $\rho_n=o(\theta_n)$, introduced below, is implied by the simpler sufficient condition $\rho_n=o(\phi_n)$. This lower bound uses only Lipschitz continuity of $d$ and is separate from the polynomial boundary exponents in Assumption \ref{assump:mu}.

\subsection{Estimating the Extremal Region} \label{sec:est_ext_set}

The set $\Ec(\phi_n)$ is unknown and is replaced by its plug-in analogue. Define  $\hat{\Ec}(\phi_n) = \hat{\Ec}_1(\phi_n) \cap \hat{\Ec}_2(\phi_n)$, 
where
\begin{equation*} 
	\hat{\Ec}_1(\phi_n) = \hat{\Ec}_1^+(\phi_n) \cup \hat{\Ec}_1^-(\phi_n), \qquad
	\hat{\Ec}_1^\pm(\phi_n) = \{t\in[0, 1] : |\hat{d}_{n,\infty} \mp \hat{d}_n(t) | \le \phi_n\},
\end{equation*}
and
\begin{equation*} 
	\hat{\Ec}_2(\phi_n) = \hat{\Ec}_3(\phi_n) + [0, \delta], \quad \hat{\Ec}_3(\phi_n) = \{ t\in[0, 1-\delta]: \sup_{s\in[t, t+\delta]} |\hat{d}_n(s)| - \hat{d}_{n, \infty} \le \phi_n \}.
\end{equation*}
Recall that $r_n=\sqrt{nh_n}$ in the level case, and $r_n=\sqrt{nh_n^3}$ in the derivative case.

\begin{proposition}[Consistency of the estimated near-extremal set] \label{prop:E_sandwich}
	Let Assumptions \ref{assump:kern}, \ref{assump:error}, \ref{assump:mu} (1) and \ref{assump:g_estimator} be satisfied and $(\phi_n)_{n\in\N}, (e_n)_{n\in\N}$ be positive sequences with $\phi_n, e_n \to 0$ and, $\sqrt{|\log(h_n)|}/(\phi_n r_n) = o(e_n)$. Then, 
	\begin{equation*}
		\Ec\big(\phi_n (1 - e_n)\big) \subset \hat{\Ec}(\phi_n) \subset \Ec\big(\phi_n (1 + e_n)\big),
	\end{equation*}
	with probability converging to $1$.
\end{proposition}

The next result is the basic localization step. It shows that the minimax statistic is controlled by the signed estimation error on the near-extremal set.

\begin{proposition}[Localization of the minimax statistic] \label{prop:localization}
	Let Assumptions \ref{assump:kern}, \ref{assump:error}, \ref{assump:mu} (1) and \ref{assump:g_estimator} hold, $d_\infty > 0$ and $(\phi_n)_{n\in\N}$ be a positive sequence converging to $0$ with 
	$|\log(h_n)|^{1/2} = o(r_n\phi_n)$. 
	Then, with probability converging to one,
	\begin{equation*}
		\hat{d}_{n,\infty} - d_\infty \ge \inf_{t\in\Ec(\phi_n)} \sgn(d(t)) \{\hat{d}_n(t) - d(t)\},
	\end{equation*}
	 where $\sgn(x)$ denotes the sign of $x\in\R$.
\end{proposition}

Proposition \ref{prop:localization} is the point at which the stable-interval problem differs from a standard supremum problem. The statistic is not driven by the largest estimation error over the full time interval, but by the signed error over a shrinking near-extremal region determined by the minimax geometry of $d$.

The Gaussian approximation developed in Section \ref{sec:gaussian_evt} shows that the supremum over $\Ec(\phi_n)$ behaves like the supremum of a locally scaled stationary Gaussian process. Because the scale $\sigma(t)$ varies with time, the relevant number of effective independent blocks depends jointly on the geometry of $\Ec(\phi_n)$ and on the values of $\sigma(t)$ over that set. 

We construct the ``effective sample size'' by partitioning $\Ec(\phi_n)$ into sets of length $\rho_n$. Recall $\theta_n = \min_{i=1}^L |b_{i,n} - a_{i,n}| > 0$ from \eqref{eq:def_theta}, and let $(\rho_n)_{n\in\N}$ be a positive sequence with $\rho_n = o(\theta_n)$ and $h_n = o(\rho_n)$. Partition each interval $[a_{i, n}, b_{i, n}]$ into blocks of length $\rho_n + 2h_n$ with starting points
\begin{equation*}
	\xi_{\nu, i, n} = a_{i, n} + (\nu - 1)(\rho_n + 2h_n), \qquad \nu = 1,\dots, K_{i,n}:= \Big\lfloor \frac{|b_{i, n} - a_{i, n}|}{\rho_n + 2h_n}\Big\rfloor.
\end{equation*} 
Then,
\begin{equation} \label{eq:set_decomposition}
	\Ec(\phi_n) = \bigg(\bigcup_{i=1}^{L} \bigcup_{\nu=1}^{K_{i,n}} [\xi_{\nu, i, n}, \xi_{\nu+1, i, n}]\bigg) \cup R_n,
\end{equation}
where the intervals have length $\rho_n + 2h_n$ and the remainder $R_n$ is of order $\Oc(\rho_n) = o(\theta_n)$. Define  $\Lambda_2 = \| (K^*)' \|_2 / \|K^*\|_2$, $\ell_n = \sqrt{2 \log(\frac{\Lambda_2 \rho_n}{2\pi h_n})}$ and $\sigma_{\max} = \sup_{t\in\Ec(\phi_n)} \sigma(t)$. The effective sample size is 
\begin{equation} \label{eq:effective_sample_size}
	W_n = \sum_{i=1}^{L} \sum_{\nu=1}^{K_{i, n}} \exp\bigg(-\frac{\ell_n^2}{2}\bigg[\frac{\sigma_{\max}^2}{\sigma^2(\xi_{\nu, i, n})} - 1\bigg]\bigg).
\end{equation}
Blocks on which $\sigma(\xi_{\nu,i,n})$ is close to $\sigma_{\max}$ contribute nearly one to $W_n$, whereas blocks with substantially smaller variance are exponentially downweighted. The sample version $\widehat{W}_n$ is obtained by replacing $\sigma$ and $\Ec(\phi_n)$ by  $\hat{\sigma}$ and $\hat{\Ec}(\phi_n)$.
Define $a_n = \ell_n + \log (W_n^{(c)}) / \ell_n$, where $W_n^{(c)}$ is the solution of $\log(W_n^{(c)}) + \log^2(W_n^{(c)}) / (2 \ell_n^2) = \log(W_n)$, and define $\hat{a}_n$ analogously.

The polynomial exponents in Assumption \ref{assump:mu} are used only to control the size of deterministic shells around the extremal set. Let $\alpha_{\max} = \max_{(x, \nu) \in (\partial\Ec)^{ext}} \alpha_{x,\nu}$. By Proposition \ref{prop:extremal_set2}, $\lambda\{\Ec(\phi_n(1+e_n))\setminus\Ec(\phi_n(1-e_n))\} \lesssim \phi_n^{1/\alpha_{\max}}$ under Assumption \ref{assump:mu} (2), whereas the sharper bound $\lambda\{\Ec(\phi_n(1+e_n))\setminus\Ec(\phi_n(1-e_n))\} \lesssim e_n\phi_n^{1/\alpha_{\max}}$ holds under Assumption \ref{assump:mu} (3), with analogous bounds for the Hausdorff distance. Later arguments require the relevant shell width to be	$o(\ell_n^{-2})$. The polynomial bounds in Assumptions \ref{assump:mu} (2) and (3) may therefore be replaced by any regular boundary modulus for which the corresponding shell width satisfies the same $o(\ell_n^{-2})$ condition.
	
The following conditions collect the rate restrictions used in the extreme value approximation.

\begin{assumption} \label{assump:sequences}
	The positive sequences $(h_n)_{n\in\N}$ and $(\rho_n)_{n\in\N}$ tend to zero, as $n\to\infty$, and additionally satisfy
	\begin{equation*}
		nh_n \to\infty, \qquad \rho_n |\log(h_n)| \to 0, \qquad \ell_n r_n h_n^\kappa \to 0,  \qquad  \frac{\ell_n r_n \log^2 n}{n^{3/4}} = \Oc(1),
	\end{equation*}
	and $h_n \le \rho_n^{1+\eta}$, for some $\eta > 0$, where $r_n=\sqrt{nh_n}$ and $\kappa = 3$ in the level case, and $r_n=\sqrt{nh_n^3}$ and $\kappa = 2$ in the derivative case.	
\end{assumption}

The conditions on $h_n$ and $\rho_n$ may appear cumbersome, but simplify the technical proofs. Their explicit use for Theorem \ref{thm:localized_gumbel} is discussed after its proof in Section \ref{sec:auxiliary_results}.

Generally, $W_n$ is not necessarily dominated by those indices $(i, \nu)$, where $\sigma^2(\xi_{\nu, i, n})$ is close to $\sigma_{\max}^2$. The next assumption rules out irregular cases in which the effective sample size is dominated by a delicate competition between long intervals with slightly smaller variance and very short intervals with maximal variance.
It is formulated for an arbitrary finite union of closed intervals and isolated points $\Ac \subset [0, 1]$, and throughout this section it is to be understood with $\Ac = \Ec$.

\begin{assumption} \label{assump:regularity_sigma}
	\begin{enumerate}
		\item The supremum $\sigma_{\max}(\Ac)$ is attained on a non-degenerate interval $I \subset [a_i, b_i], i\in\{1, \dots, L\}$, and the set $\{t \in [0, 1]: \sigma(t) = \sigma_{\max}\}$ is a finite union of intervals and isolated points.
		\item The sequences $h_n, \rho_n$ and $\theta_n$ satisfy $\log|\log h_n| \cdot |\log \rho_n| / | \log h_n| \to 0$ and $\ell_n^{-2} \le \theta_n$.
	\end{enumerate}
\end{assumption}

The role of Assumption \ref{assump:regularity_sigma} is twofold. For deterministic localization sets, it yields the sharp equality in the otherwise conservative bound of Theorem \ref{thm:localized_gumbel}. However, the stable-interval test is feasible only after replacing the unknown near-extremal set $\Ec(\phi_n)$, the maximal long-run variance, and the effective sample size by their plug-in analogues. This replacement is controlled by Proposition \ref{prop:W_plugin} and requires this additional regularity. Consequently, Assumption \ref{assump:regularity_sigma} is imposed in Theorem \ref{thm:main_conv}, even though the underlying extreme value bound has a conservative version under weaker conditions.

The assumption is analogous in spirit to Assumption \ref{assump:mu}. In both cases we distinguish a favorable non-degenerate regime, where the relevant extremal object is attained on a set of positive Lebesgue measure, from a more restrictive regime that also permits isolated maximizers and therefore requires stronger localization control. Here Assumption \ref{assump:mu} governs the geometry of the near-extremal set of the signal $d$, whereas Assumption \ref{assump:regularity_sigma} plays the same role for the variance profile $\sigma$ and the effective sample size entering the extreme value calibration.
 
\subsection{The Stable-interval Test} \label{sec:test}

Let $\hat\sigma_{\max} = \sup_{t\in\hat{\Ec}(\phi_n)}\hat{\sigma}_n(t)$. For $c\in\R$, define
\begin{equation*}
	T_n(c) = \frac{r_n\hat{a}_n}{\|K^*\|_2 \hat{\sigma}_{\max}} \big(c - \hat{d}_{n,\infty}\big) - \hat{a}_n^2,
\end{equation*}
The statistic is large when the estimated minimax deviation $\hat d_{n,\infty}$ is sufficiently below the boundary value $c$. 
At the null boundary $c=d_\infty$, the following theorem gives a conservative Gumbel approximation.

\begin{theorem}[Asymptotic validity of the stable-interval test] \label{thm:main_conv}
	Let the assumptions of Proposition \ref{prop:localization} and Assumption \ref{assump:sequences} hold and $\rho_n / \theta_n \to 0$. Assume in addition either Assumptions \ref{assump:mu} (2) and \ref{assump:regularity_sigma} (1) with $\phi_n^{1/\alpha_{\max}} = o(\ell_n^{-2})$, or Assumptions \ref{assump:mu} (3) and \ref{assump:regularity_sigma} (2) with $e_n\phi_n^{1/\alpha_{\max}} = o(\ell_n^{-2})$, for some sequence $e_n \searrow 0$.
	Then, for any $x\in \R$,
	\begin{equation} \label{eq:main_result}
		\liminf_{n\to\infty} \pr(T_n(d_\infty) \le x) \ge \exp(-\exp(-x)).
	\end{equation}
\end{theorem}
The inequality in Theorem \ref{thm:main_conv} reflects the conservative localization bound in Proposition \ref{prop:localization}. Let $q_{1-\alpha} = - \log(- \log(1-\alpha))$ be the $(1-\alpha)$-quantile of the standard Gumbel distribution. We reject $H_0: d_\infty \ge \Delta$ whenever
\begin{equation} \label{eq:decision_rule}
	T_n(\Delta) > q_{1-\alpha}.
\end{equation}

\begin{corollary} \label{cor:test}
	Let $\Delta > 0$. Under the assumptions of Theorem \ref{thm:main_conv}, the test defined by the decision rule in \eqref{eq:decision_rule} has asymptotic level $\alpha$. Moreover, it is consistent against all fixed alternatives satisfying $0 \le d_\infty < \Delta$.
\end{corollary}
In some cases the inequality in \eqref{eq:main_result} can be sharpened to equality. A sufficient condition is
\begin{equation*}
	\frac{r_n\hat{a}_n}{\|K^*\|_2 \hat{\sigma}_{\max}} \big| \hat{d}_{n,\infty} - d_\infty - \inf_{t\in \Ec(\phi_n)} \sgn(d(t)) \{\hat{d}_n(t) - d(t)\} \big| = o_\pr(1),
\end{equation*}
This condition states that the minimax statistic is asymptotically determined by the same signed extremal process that appears in the localization bound. For example, if $\delta = 0$ and $\inf_{t\in[0, 1]}|d(t)| = d_\infty> 0$, the functional $\Phi(f) = \inf_{t\in[0, 1]} |f(t)|$ is Hadamard-differentiable in $d$ in direction $h\in C([0, 1])$ with derivative 
\begin{equation*}
	\Phi_d'(h) = \inf_{t\in \Ec} \sgn(d(t)) h(t).
\end{equation*}
This gives the usual first-order delta-method expansion. The condition above is stronger, since it is required at the Gumbel scale and with $\Ec$ replaced by $\Ec(\phi_n)$. It therefore has to be verified by a separate localization argument.

\subsection{Estimating the Onset of a Stable Period}
\label{subsec:onset}

The test in Section \ref{sec:test} answers whether a stable interval of length $\delta$ exists at tolerance $\Delta$. In applications one is often also interested in locating such an interval, and in particular in estimating its onset. The moving-window representation of the stable-interval functional provides a natural estimator for this purpose.

Define the population and empirical moving-window deviations $M(t) = \sup_{s\in[t,t+\delta]} |d(s)|$ and $\widehat{M}_n(t) = \sup_{s\in[t,t+\delta]} |\hat d_n(s)|$, for $t \in[0, 1-\delta]$. Then $d_\infty=\inf_{t\in[0,1-\delta]} M(t)$ and $\hat d_{n,\infty}=\inf_{t\in[0,1-\delta]} \widehat{M}_n(t)$. Thus, minimizers of $M$ are the starting points of intervals that are most stable in the minimax sense.
Based on the set of optimal starting points $\Ec_3$, define the onset of the first optimal stable period by $\tau_* = \inf \Ec_3$, where $\inf \empty = \infty$ by convention. The estimated near-optimal start set is $\hat{\Ec}_3(\phi_n)$ and the corresponding onset estimator is $\hat\tau_* = \inf \hat{\Ec}_3(\phi_n)$. The interval $[\hat\tau_*, \hat\tau_*+\delta]$ is therefore the first estimated near-optimal interval. If the test rejects $H_0:d_\infty\ge\Delta$, then this interval is an estimated stable interval. 

The estimator $\hat\tau_*$ targets the first minimax-optimal stable interval. If the scientific target is instead the first interval that satisfies the pre-specified tolerance $\Delta$, define
\begin{equation*}
	\Tc_\Delta = \{t\in[0,1-\delta]:M(t)\le\Delta\}, \qquad \tau_\Delta = \inf\Tc_\Delta,
\end{equation*}
whenever $\Tc_\Delta\neq\emptyset$. A plug-in estimator is
\begin{equation*}
	\hat{\Tc}_\Delta = \{t\in[0,1-\delta]: \widehat{M}_n(t)\le\Delta\}, \qquad \hat\tau_\Delta =
	\inf\hat{\Tc}_\Delta.
\end{equation*}
The following proposition summarizes the basic consistency properties of the onset estimators. It is a direct consequence of the uniform consistency of $\hat d_n$ and the continuity of the moving-window map $f\mapsto \sup_{s\in[t,t+\delta]}|f(s)|$.

\begin{proposition}[Consistency of onset estimators]
	\label{prop:onset_consistency}
	Let $\phi_n\searrow0$ satisfy $\|\hat d_n-d\|_\infty=o_\pr(\phi_n)$. Then, for every sequence $e_n\searrow0$ such that $\|\hat d_n-d\|_\infty=o_\pr(e_n\phi_n)$,
	\begin{equation*}
		\Ec_3\big(\phi_n(1-e_n)\big) \subset \hat{\Ec}_3(\phi_n) \subset \Ec_3\big(\phi_n(1+e_n)\big)
	\end{equation*}
	with probability tending to one. In particular, $\hat\tau_*\convp\tau_*$.
	
	Suppose, moreover, that $\Tc_\Delta\neq\emptyset$, and that stable intervals occur immediately after $\tau_\Delta$, in the sense that, for every $\eps>0$, it exists $t_\eps \in [\tau_\Delta,\tau_\Delta+\eps]\cap[0,1-\delta]$ such that $M(t_\eps) < \Delta$.
	Then, $\hat\tau_\Delta\convp\tau_\Delta$.
\end{proposition}

The distinction between $\tau_*$ and $\tau_\Delta$ is important. The estimator $\hat\tau_*$ locates the first interval that is closest to being stable in the minimax sense, and is therefore tied directly to the extremal localization used in the test. The estimator $\hat\tau_\Delta$ instead locates the first interval satisfying the scientific tolerance $\Delta$. When several intervals are $\Delta$-stable, $\hat\tau_\Delta$ targets the earliest such interval, whereas $\hat\tau_*$ targets the earliest interval among those with minimal window-wise deviation. Both quantities are useful: $\hat\tau_\Delta$ answers when stability first occurs at the chosen tolerance, while $\hat\tau_*$ identifies the first maximally stable regime.

The only additional identifiability condition for the threshold-based onset $\tau_\Delta$ is that the first contact with the tolerance level is not purely tangential. If $M(\tau_\Delta)=\Delta$ but $M(t)>\Delta$ immediately to the right of $\tau_\Delta$, then the set of $\Delta$-stable starts may be hit only at an isolated point, and the plug-in threshold estimator need not be stable under small perturbations.

If the reported interval should be stable with high probability rather than only selected by the plug-in criterion, one may replace the threshold $\Delta$ in	$\hat{\Tc}_\Delta$ by $\Delta-\eta_n$, where $\eta_n\searrow0$ dominates the uniform estimation error. 

In practice, we report the estimated onset together with the corresponding interval and its empirical deviation, $\widehat{M}_n(\hat\tau) = \sup_{s\in[\hat\tau,\hat\tau+\delta]}|\hat d_n(s)|$, where $\hat\tau$ is either $\hat\tau_*$ or $\hat\tau_\Delta$ depending on the inferential target.

\section{Localized Approximation Theory}
\label{sec:gaussian_evt}

The stable-interval statistic is asymptotically governed by the estimation error of $\hat d_n$ on a localized near-extremal set. In this section we isolate the probabilistic approximation behind this phenomenon. The results are stated for deterministic sets $\mathcal A_n$, because the same theory is useful beyond stable-interval detection and also applies to the relevant-change statistics discussed in Section \ref{sec:relevant_change}. In the stable-interval test, $\mathcal A_n$ will be replaced by the near-extremal set $\Ec(\phi_n)$.

\subsection{Localized Gaussian and Extreme-value Approximations}
\label{subsec:localized_sets}

We first specify the class of sets over which the stochastic process is maximized.

\begin{assumption} \label{assump:sets}
	Let $\Ac\subset(0, 1)$ be a finite union of closed intervals and isolated points, i.\,e., $\Ac = \bigcup_{i=1}^{L} [a_i, b_i]$, for $a_i \le b_i$. Further, let $(\Ac_n)_{n\in \N}$ denote a sequence of sets $\Ac_n = \bigcup_{i=1}^{L} [a_{i, n}, b_{i, n}]$, such that $[a_i, b_i] \subset [a_{i, n}, b_{i, n}]$, for all sufficiently large $n\in\N$. Define $\theta_n := \min_{i=1}^L |b_{i,n} - a_{i,n}|$, for $n\in\N$, and assume $\theta_n > 0$.
\end{assumption}
For a finite union of closed intervals and isolated points $\Bc\subset[0, 1]$, define $\sigma_{\max}(\Bc) = \sup_{t\in\Bc}\sigma(t)$. Recall the construction from Section \ref{sec:est_ext_set}, and denote by $W_n(\Bc)$ the effective sample size defined analogously to \eqref{eq:effective_sample_size}. The centering constant is $a_n(\Bc) = \ell_n + \log (W_n^{(c)}(\Bc)) / \ell_n$, where $W_n^{(c)}(\Bc)$ is the solution of $\log(W_n^{(c)}(\Bc)) + \log^2(W_n^{(c)}(\Bc)) / (2 \ell_n^2) = \log(W_n(\Bc))$. The sample analogues $\hat{\sigma}_{\max}(\Bc), \widehat{W}_n(\Bc)$ and $\hat{a}_n(\Bc)$ are obtained by replacing $\sigma$ with $\hat{\sigma}_n$. With this notation, $W_n = W_n(\Ec(\phi_n))$ and $\widehat{W}_n = \widehat{W}_n(\hat{\Ec}(\phi_n))$.

The first step is a Gaussian approximation for the localized supremum of the estimation
error. The statement is formulated for the process $\hat d_n-d$, using the equivalent kernel
$K^*$ introduced in Section \ref{subsec:signal_estimation}. In the level case $r_n=\sqrt{nh_n}$, while in the derivative case $r_n=\sqrt{nh_n^3}$. The normalization $r_n$ is the pointwise rate of the estimator $\hat d_n$.

\begin{theorem}[Gaussian approximation on localized sets]
	\label{thm:gaussian_approx}
	Suppose that Assumptions \ref{assump:kern}, \ref{assump:error}, \ref{assump:mu} (1),  \ref{assump:g_estimator}, \ref{assump:sequences}, and \ref{assump:sets} hold with $\rho_n / \theta_n \to 0$. Then
	\begin{equation} \label{eq:gaussian_discrete}
		 \frac{r_n a_n(\Ac_n)}{\hat\sigma_{\max}(\Ac_n)} \sup_{t\in\Ac_n}\big(\hat d_n(t)-d(t)\big) 
		 \stackrel{\Dc}{=} \frac{a_n(\Ac_n)}{\sqrt{nh_n} \sigma_{\max}(\Ac_n)} \sup_{t\in \Ac_n} \sum_{i=1}^n \sigma\big(\tfrac{i}{n}\big) V_i K_{h_n}^*\big(\tfrac{i}{n} - t\big) + o_\pr(1),
	\end{equation}
	where $V_1,\ldots,V_n$ are independent standard normal random variables. 
	Equivalently,
	\begin{equation}  \label{eq:gaussian_integral}
		\frac{r_n a_n(\Ac_n)}{\hat\sigma_{\max}(\Ac_n)} \sup_{t\in\Ac_n}\big(\hat d_n(t)-d(t)\big) 
		\stackrel{\Dc}{=} \frac{a_n(\Ac_n)}{\sigma_{\max}(\Ac_n)} \sup_{t\in \Ac_n} \int_\R \sigma(t) K^*(x-\tfrac{t}{h_n})\diff B_x +  o_\pr(1),
	\end{equation}
	where $B$ is a standard Brownian motion.
\end{theorem}
The theorem shows that, after the same normalization as in the final extreme value statistic, the localized estimation error behaves as a Gaussian kernel process with time-varying scale $\sigma(t)$. The first representation, \eqref{eq:gaussian_discrete}, is particularly useful for simulation-based calibration, because it only requires drawing independent standard normal variables. The second representation, \eqref{eq:gaussian_integral}, is more convenient for deriving the extreme value limit.

The signed version needed for stable-interval detection follows immediately. If $d$ is bounded away from zero on $\Ac_n$, then
\begin{align}
	& \frac{r_n a_n(\Ac_n)}{\hat{\sigma}_{\max}(\Ac_n)} \sup_{t\in \Ac_n} - \sgn(d(t)) \big(\hat{d}_n(t) - d(t)\big) \notag \\
	& \stackrel{\Dc}{=} \frac{a_n(\Ac_n)}{\sqrt{nh_n} \sigma_{\max}(\Ac_n)} \sup_{t\in \Ac_n} - \sgn(d(t)) \sum_{i=1}^n \sigma\big(\tfrac{i}{n}\big) V_i K_{h_n}^*\big(\tfrac{i}{n} - t\big)  + o_\pr(1). \label{eq:gaussian_signed}
\end{align}
Because the Gaussian process is symmetric on disjoint components, this signed version has the same extreme value behavior as the unsigned supremum.

The Gaussian approximation can be further reduced to an explicit Gumbel law. The reduction is delicate because the scale $\sigma(t)$ varies over time and because the localized set $\Ac_n$ may contain both non-degenerate intervals and shrinking neighborhoods of isolated points.
Assumption \ref{assump:regularity_sigma} rules out pathological cases in which the effective sample size is determined by a delicate competition between long intervals with slightly smaller variance and very short intervals with maximal variance, and yields a sharp form of the following approximation.

\begin{theorem}[Localized extreme value approximation]
	\label{thm:localized_gumbel}
	Under the assumptions of Theorem \ref{thm:gaussian_approx}, 
	
	\begin{equation} \label{eq:main_conv}
		\liminf_{n\to\infty} \pr\bigg(\frac{r_n\hat{a}_n(\Ac_n)}{\|K^*\|_2 \hat{\sigma}_{\max}(\Ac_n)} \sup_{t\in \Ac_n} \big(\hat{d}_n(t) - d(t)\big) - \hat{a}_n^2(\Ac_n) \le x\bigg) \ge \exp(-\exp(-x)),
	\end{equation}
	for every $x\in\R$. If, in addition, Assumption \ref{assump:regularity_sigma} (1) or (2) is satisfied, equality holds in the previous display.
\end{theorem}
The inequality in Theorem \ref{thm:localized_gumbel} is sufficient for level control in the stable-interval test derived from Theorem \ref{thm:main_conv}, since the latter only gives a lower bound. However, to properly estimate $W_n = W_n(\Ec(\phi_n))$ with $\widehat{W}_n = \widehat{W}_n(\hat{\Ec}(\phi_n))$, Assumption \ref{assump:regularity_sigma} (1) or (2) must hold, which gives the sharp limit in \eqref{eq:main_conv}. In particular, the theorem extends kernel supremum approximations from the constant-variance case to locally stationary errors with time-varying long-run variance.

The same conclusion holds for the signed statistic in \eqref{eq:gaussian_signed}. Consequently, whenever $d$ is bounded away from zero on $\Ac_n$,
\begin{align}
	& \liminf_{n\to\infty} \pr\bigg( \frac{r_n\hat a_n(\Ac_n)}{\|K^*\|_2\hat\sigma_{\max}(\Ac_n)} \sup_{t\in\Ac_n} -\sgn\big(d(t)\big)	\big(\hat d_n(t)-d(t)\big) - \hat a_n^2(\Ac_n) \le x \bigg) \notag \\
	& \ge \exp\big(-\exp(-x)\big),
	\label{eq:signed_localized_gumbel}
\end{align}
with equality under Assumption \ref{assump:regularity_sigma} (1) or (2).

\subsection{A Pickands-type Tail Approximation}
\label{subsec:pickands}

The extreme value approximation in Theorem \ref{thm:localized_gumbel} rests on a local tail expansion for the stationary Gaussian process induced by the kernel. We state the result separately because it may be useful in other kernel supremum problems.
Let
\begin{equation*}
	X(t) = \frac{1}{\|K\|_2} \int_{\R}K(x-t) \diff B_x,
\end{equation*}
where $B$ is a standard Brownian motion. Then $X$ is stationary, centered, and has unit variance. Its covariance satisfies
\begin{equation*}
	\cov(X(t), X(0)) = 1 - \frac{\Lambda_2^2}{2} t^2 + o(t^2), 
\end{equation*}
as $t\to 0$, with $\Lambda_2=\|K'\|_2/\|K\|_2$.

\begin{theorem}[Local Pickands approximation for kernel Gaussian processes] \label{thm:evt}
	Let $K$ satisfy Assumption \ref{assump:kern}, $\Lambda_2 = \|K'\|_2/\|K\|_2$ and $C_0 = \Lambda_2^2/2$. 
	Further, let $(h_n)_{n\in\N}$ and $(\rho_n)_{n\in\N}$ be positive sequences converging to $0$ with $h_n = o(\rho_n)$, as $n\to\infty$, and define
	$\ell_n = \sqrt{2 \log(\Lambda_2 \rho_n/(2 \pi h_n))}$. 
	Finally, let $\{c_{k, n}\}_{1\le k \le K_n, n\in\N}$ denote coefficients satisfying $(\inf_{k=1}^{K_n} c_{k, n} - 1)\ell_n^2 \to \infty$, and set $g_{n,k} = c_{k, n}\ell_n$.
	Then,  
	\begin{equation*} \label{eq:pickands_tail}
		\lim_{n\to\infty} \sup_{1 \le k\le K_n}\bigg|\frac{\pr\big(\sup_{t\in[0, \rho_n/h_n]} X(t) > g_{n, k}\big)}{C_0^{1/2} (\rho_n/h_n) g_{n, k} \Psi(g_{n,k})} - H_2\bigg| = 0,
	\end{equation*}
	where $H_2 = \pi^{-1/2}$ denotes Pickand's constant and $\Psi(x) = (2\pi)^{-1/2} \int_x^\infty \exp(-y^2/2)\diff y$.
\end{theorem}
The theorem is a high-level local tail approximation for a kernel-induced Gaussian process. 
Combined with the block decomposition of $\Ac_n$, it yields the effective sample size $W_n(\Ac_n)$ and the centering $a_n(\Ac_n)$ appearing in Theorem \ref{thm:localized_gumbel}.

\begin{remark}[Odd kernels]
	For derivative estimation, the equivalent kernel is typically odd. 
	The preceding theorem remains valid for odd kernels $K(x)=-K(-x)$, supported on $[-1,1]$ and twice differentiable, with the corresponding curvature constant
	\begin{equation*}
		\Lambda_2^2 = \frac{\int_{\R} x^2[K'(x)]^2 \diff x}{\int_{\R} x^2K^2(x) \diff x}.
	\end{equation*}
\end{remark}

\begin{remark}[General index sets]
	Similar to Theorem 7.2 of \cite{piterbarg1996}, the interval $[0,\rho_n/h_n]$ in Theorem \ref{thm:evt} can be replaced by more general regular sets $\Bc_n$. In that case the normalizing logarithm becomes $\ell_n = \sqrt{2 \log(\Lambda_2 \lambda(\Bc_n)/(2 \pi))}$ and the factor $\rho_n/h_n$ in \eqref{eq:pickands_tail} is replaced by $\lambda(\Bc_n)$.
\end{remark}

\subsection{Gaussian Calibration}
\label{subsec:gaussian_calibration}

The Gumbel approximation in Theorem~\ref{thm:localized_gumbel} gives an explicit and analytically convenient critical value. However, as is typical for extreme value approximations of suprema, convergence to the limiting Gumbel distribution may be slow. For finite samples, it is therefore natural to use the Gaussian approximation in Theorem \ref{thm:gaussian_approx} directly.

Let $\hat{\Ac}_n$ denote the estimated localization set used in the statistic. In the stable-interval application, $\hat{\Ac}_n=\hat{\Ec}(\phi_n)$. Generate independent standard normal variables $V_1^{(b)}, \dots, V_n^{(b)}$, for $b=1, \dots, B$ and compute
\begin{equation*}
	Z_n^{(b)}(\hat{\Ac}_n)
	= \sup_{t\in\hat{\Ac}_n} \bigg[	-\sgn\big(\hat d_n(t)\big) \frac{1}{\sqrt{nh_n}} \sum_{i=1}^n \hat\sigma_n\big(\tfrac{i}{n}\big) V_i^{(b)} K_{h_n}^*\big(\tfrac{i}{n} - t\big) \bigg].
\end{equation*}
Let $\hat q_{1-\alpha}^{(G)}$ be the empirical $(1-\alpha)$-quantile of 
\begin{equation*}
	\frac{\hat a_n(\hat{\Ac}_n)}{\|K^*\|_2\hat\sigma_{\max}(\hat{\Ac}_n)} Z_n^{(b)}(\hat{\Ac}_n) - \hat a_n^2(\hat{\Ac}_n), \qquad b=1,\dots,B.
\end{equation*}
The Gaussian-calibrated version of the stable-interval test rejects $H_0:d_\infty\ge\Delta$ whenever
\begin{equation} \label{eq:gaussian_calibrated_rule}
	T_n(\Delta)	> \hat q_{1-\alpha}^{(G)}.
\end{equation}
Equivalently, one may avoid the extreme value normalization altogether and compare the uncentered statistic $r_n(\Delta-\hat d_{n,\infty})$ with the empirical quantile of $Z_n^{(b)}(\hat{\Ac}_n)$, for $b=1, \dots, B$.

Both formulations are based on the same Gaussian approximation. The normalized version in \eqref{eq:gaussian_calibrated_rule} is directly comparable to the Gumbel-calibrated statistic, while the uncentered version is simpler to implement.

The Gaussian calibration keeps the local variance profile, the kernel covariance structure, and the geometry of the estimated extremal set in the critical value. 
It therefore avoids the final extreme value approximation and is expected to provide more accurate finite-sample critical values, especially when the effective sample size $W_n(\hat{\Ac}_n)$ is moderate.

\subsection{Consequences for Relevant-change Testing}
\label{sec:relevant_change}

The localized extreme value theory developed above also yields an extension of the
supremum-type relevant-change tests of \cite{bucher2021}. Their results cover hypotheses based on the sup-norm of a gradually varying mean function, under a calibration in which the long-run variance is constant over time. The present theory removes this restriction and also applies to signals involving the derivative of the mean function.

To explain the connection, consider first a relevant-change hypothesis of the form
\begin{equation*}
	H_0^{\mathrm{rel}}: d_\infty^{\mathrm{rel}}\le \Delta \qquad\text{vs.}\qquad	H_1^{\mathrm{rel}}:	d_\infty^{\mathrm{rel}}> \Delta,
\end{equation*}
where $d_\infty^{\mathrm{rel}} = \sup_{t\in I}|d(t)|$ for some compact interval $I\subset[0,1]$. Typical examples are $d(t)=\mu(t)-g(\mu,t)$ or $d(t)=\mu'(t)$. The first case corresponds to relevant deviations of the mean from a reference curve, while
the second case corresponds to relevant departures from local flatness.

Let $\Mc = \{ t\in I: |d(t)|=d_\infty^{\mathrm{rel}}\}$ be the extremal set of the supremum functional and define its near-extremal enlargement by $\Mc(\phi_n) = \{ t\in I:
d_\infty^{\mathrm{rel}}-|d(t)|\le \phi_n \}$. The same localization principle as in Section \ref{sec:localized_test} implies that the asymptotic distribution of $\hat d_{\infty}^{\mathrm{rel}} = \sup_{t\in I}|\hat d_n(t)|$ is determined by the signed estimation error on $\Mc(\phi_n)$. Consequently, the general Gaussian and extreme value approximations of Section \ref{sec:gaussian_evt} apply with $\Ac_n=\Mc(\phi_n)$.

The resulting calibration differs from the one in \cite{bucher2021} in two ways. First, the local long-run variance $\sigma^2(t)$ is allowed to vary with $t$. The limiting centering and scaling are therefore governed by $\sigma_{\max}\big(\Mc(\phi_n)\big) = \sup_{t\in\Mc(\phi_n)}\sigma(t)$ and by the effective sample size $W_n\big(\Mc(\phi_n)\big)$.

Thus, only blocks whose local variance is close to the maximal variance over the
near-extremal region contribute substantially to the extreme value centering. When $\sigma(t)$ is constant, $W_n\big(\Mc(\phi_n)\big)$ reduces to the usual number of effective blocks and the calibration collapses to the constant-variance form.

Second, the same theory applies to derivative-based hypotheses. For example, one may test whether the slope of the mean function is relevant somewhere,
\begin{equation*}
	H_0^{\mathrm{rel}}:	\sup_{t\in I}|\mu'(t)|\le \Delta \qquad\text{vs.}\qquad	H_1^{\mathrm{rel}}:	\sup_{t\in I}|\mu'(t)|>\Delta.
\end{equation*}

In this case $d=\mu'$, the pointwise rate is $r_n=\sqrt{nh_n^3}$, and the corresponding equivalent kernel is the derivative kernel $K^*(x)=xK(x) / \int y^2 K(y) \diff y$. The Gumbel calibration is obtained from Theorem \ref{thm:localized_gumbel} with these choices of $r_n$ and $K^*$.

More explicitly, let $\hat{d}_\infty^{\mathrm{rel}} = \sup_{t\in I}|\hat d_n(t)|$ and let $\widehat{\Mc}(\phi_n)$ be the plug-in version of $\Mc(\phi_n)$. Define
\begin{equation*}
	T_n^{\mathrm{rel}}(\Delta)
	= \frac{r_n\hat a_n\big(\widehat{\Mc}(\phi_n)\big)}{\|K^*\|_2 \hat\sigma_{\max}\big(\widehat{\Mc}(\phi_n)\big)} (\hat d_\infty^{\mathrm{rel}} - \Delta) -
	\hat a_n^2\{\widehat{\Mc}(\phi_n)\}.
\end{equation*}
Under the analogue of the extremal-set regularity assumptions used above,
\begin{equation*}
	T_n^{\mathrm{rel}}(d_\infty^{\mathrm{rel}}) \convw G,
\end{equation*}
where $G$ is standard Gumbel distributed, provided the equality conditions in Theorem \ref{thm:localized_gumbel} hold. Therefore, the relevant-change test rejects $H_0^{\mathrm{rel}}$ for large values of $T_n^{\mathrm{rel}}(\Delta)$, using either the Gumbel quantile or the Gaussian calibration from Section \ref{subsec:gaussian_calibration}.

This shows that the present extreme value theory strictly extends the calibration of \cite{bucher2021}. It permits time-varying long-run variance, allows extremal regions that may shrink to isolated points, and covers both level-based and derivative-based signals.

%% file: sec_empirical_results.tex
\section{Empirical Results} \label{sec:empirical}

We now study the finite-sample behavior of the proposed stable-interval tests and illustrate their use in two applications. The implementation choices described below are used both in the simulations and in the data examples, unless stated otherwise.\footnote{Python implementations of the methods and experiments are available on GitHub: \url{https://github.com/FlorianHeinrichs/StableIntervalDetection}.}

For local linear estimation, the quartic kernel $K(x) = 15 / 16 (1 - x^2)^2$ was used. The bandwidth $h_n \in [0.03, 0.25]$ is tuned with $k$-fold cross validation, using $k=5$ random splits, to minimize the mean-squared error of the jackknife estimator, in the level case $d = \mu - g(\mu)$, and of the plain local linear estimator, in the derivative case $d = \mu'$.

The tuning parameter $\rho_n$ determines the block length used to approximate the effective number of approximately independent local maxima in the near-extremal set. The asymptotic theory requires the kernel bandwidth $h_n$ to be smaller than the block length, while $\rho_n$ itself still tends to zero, that is $h_n = o(\rho_n)$ and $\rho_n \to 0$. In the simulations we therefore use the default choice $\rho_n=\sqrt{h_n}$, enlarged if necessary to ensure that the logarithmic centering term $\ell_n=\{2\log(\Lambda_2\rho_n/(2\pi h_n))\}^{1/2}$ is well defined. 

The localization parameter $\phi_n$ determines the estimated near-extremal set $\hat\Ec (\phi_n)$, over which the Gaussian and extreme value calibrations are computed. The theory requires $\phi_n$ to dominate the uniform stochastic error of the estimator of the stability signal, because the estimated near-extremal set must contain the relevant population extremal region with high probability. Accordingly, we use a rate-based finite-sample rule proportional to $\widehat{\mathrm{sd}}(\hat d_n)\sqrt{|\log h_n|}/r_n$, where $r_n=\sqrt{nh_n}$ for level-based stability and $r_n=\sqrt{nh_n^3}$ for derivative-based stability. To avoid degenerate estimated extremal sets in small samples, this stochastic rule is combined with a small fraction of the empirical range of $\hat{d}_n$, such that
\begin{equation*}
	\phi_n = \max\Big\{2 \widehat{\mathrm{sd}}(\hat d_n) \frac{\sqrt{|\log h_n|}}{r_n}, 0.05 \Big[\max_{t\in[0, 1]} \hat d_n(t) - \min_{t\in[0, 1]} \hat d_n(t)\Big]\Big\}.
\end{equation*}
These choices are not intended to be optimal tuning rules, but rather stable default choices satisfying the qualitative rate requirements of the asymptotic theory. 

\subsection{Simulation Study} \label{sec:sim_study}

We consider the location scale model $X_{i, n} = \mu(i/n) + \sigma(i/n) \eps_{i, n}$, with different mean functions $\mu$, a smooth standard deviation $\sigma(t) = \tfrac{1}{2}- \tfrac{1}{4}\cos(2 \pi t)$ and locally stationary errors $(\eps_{i, n})$. More specifically, define
\begin{equation*}
	\eps_{i, n} = \sqrt{a(i/n)} \eps_i^{(1)} + \sqrt{1-a(i/n)} \eps_i^{(2)}, \qquad a(t) =  \tfrac{1}{2}[1 - \cos(\tfrac{\pi}{2}[-\cos(\pi t) + 1])],
\end{equation*}
where $\eps_i^{(k)}$ satisfies the AR-equation $\eps_i^{(k)} = \tfrac{\sqrt{3}}{2}(\eta_i^{(k)} + \tfrac{1}{2} \eps_{i-1}^{(k)})$, for $k=1, 2$, and $(\eta_i^{(1)})_{i\in\Z}$, $(\eta_i^{(2)})_{i\in\Z}$ denote independent sequences of i.i.d. random variables with $\eta_i^{(1)} \sim\Nc(0, 1)$ and $\eta_i^{(2)}\sim \Uc([-\sqrt{3}, \sqrt{3}])$.

We use four mean functions with different properties, displayed in Figure \ref{fig:mu}. Let $d_{\infty}(\delta, d_0)$ denote the minimax deviation computed from some signal $d_0$ at interval length $\delta \ge 0$.
For derivative-based stability, $d_0=\mu'$, we consider
\begin{equation*}
	\mu_1(t)=-ct-(1-c)\{1-\exp(-\kappa t)\}/\kappa,
	\qquad
	\mu_2(t)=\beta t+\{\exp(-\kappa t)-1\}/\kappa.
\end{equation*}
Here $\mu_1$ is a monotone curve with decreasing slope, while $\mu_2$ settles transiently and then drifts. We use $c=0.1$, $\beta=0.15$, and $\kappa=6$. 
For level-based stability, $d_0 = \mu$, we consider $\mu_3(t) = t-1/2$, and 
\begin{equation*}
	\mu_4(t)= S\bigg(\frac{t-a}{m-a}\bigg) \bigg\{1-S\bigg(\frac{t-m}{b-m}\bigg)\bigg\}, 
\end{equation*}
based on the smoothstep transition 
\begin{equation*}
	S(x)=(10x^3-15x^4+6x^5) \id_{[0, 1]}(x)+ \id_{(1, \infty)}(x).
\end{equation*}
Hence, $\mu_4$ is a smooth excursion from the fixed baseline $g\equiv0$. We set $a=0.25$, $m=0.5$, and $b=0.75$, so that $d = \mu_4$ has $d_\infty(\delta, d)=0$ for $\delta\le0.25$. For the other functions,  $d_{\infty}(\delta, d_0) \neq 0$, and we define the normalized signal $d(t) = d_0(t) / d_{\infty}(\delta, d_0)$, so that $d_\infty(\delta)=1$. 
Thus $\mu_1,\mu_2,\mu_3$ assess boundary behavior, while $\mu_4$ assesses power under a genuine stable interval.

\begin{figure}[t]
	\centering
	\includegraphics[width=\textwidth]{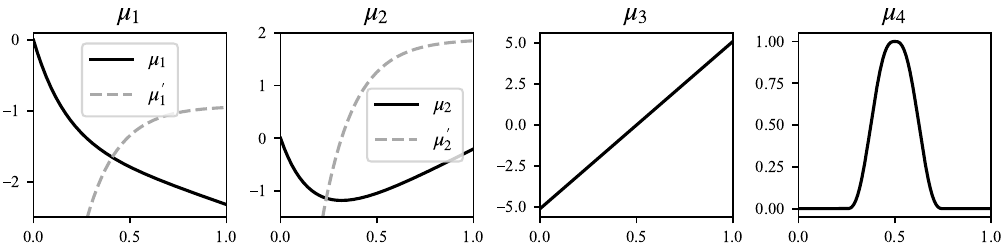}
	\caption{Various mean functions, used to generate time series.}
	\label{fig:mu}
\end{figure}

To obtain critical values for the Gaussian-calibrated test decision \eqref{eq:gaussian_calibrated_rule}, we simulated $1000$ trajectories of the Gaussian approximation. For all settings, we generated $1000$ time series and tested $H_0$ with level $\alpha = 5\%$. Table \ref{tab:simulation_results} contains empirical rejection rates for different choices of $n$ with $\delta =0.2$. Additional results for i.i.d., MA- and AR-errors, as well as different choices of $\sigma$ and $\delta$, are provided in Appendix \ref{app:appendix_empirical_results}.

For $\mu \in \{\mu_1, \mu_2, \mu_3\}$, tolerances $\Delta < 1$ correspond to the interior of the null hypothesis, whereas $\Delta = 1$ is its boundary, and values $\Delta > 1$ correspond to the alternative. For $\mu = \mu_4$, any $\Delta > 0$ corresponds to the alternative. The results suggest that the test based on Gumbel quantiles is conservative, as expected from the construction of the test decision. Moreover, the functions are defined, such that only a single optimal window exists, and the effective sample sizes $W_n$, defined in \eqref{eq:effective_sample_size}, are small. As hypothesized in Section \ref{subsec:gaussian_calibration}, the results are slightly better under the Gaussian calibration. For small $n$, the test exceeds the level of $5\%$ for $\mu = \mu_1$, and estimates it well for $\mu= \mu_2$ at $\Delta = 1$. In the level-stability case, $\mu = \mu_3$, the test is once again conservative, with increasing power as $n$ grows. Interestingly, the power seems to decrease in the derivative-based setting under a Gaussian calibration. Note that the test statistic and its Gaussian approximation are only asymptotically equivalent, and the convergence to the limiting Gumbel distribution is usually slow. Asymptotically, the term $r_n = \sqrt{nh_n^3}$ dominates $a_n = \Oc(\sqrt{|\log h_n|})$, yet, for small values of $n$, $a_n$ can exceed $r_n$. 
The alternative $\mu = \mu_4$ reflects the previous findings. Both decision rules yield conservative results, with increasing power, as $n$ or $\Delta$ grow.

\begin{table}[thb]
	\caption{Empirical rejection rates for $\delta = 0.2$.}
	\label{tab:simulation_results}
	\begin{tabular}{l|rrrrr|rrrrr}
		\toprule
		& \multicolumn{5}{c|}{Gaussian calibration \eqref{eq:gaussian_calibrated_rule}} & \multicolumn{5}{c}{Gumbel calibration \eqref{eq:decision_rule}} \\
		$n \quad \setminus \quad \Delta$ & $0.8$ & $0.9$ & $1.0$ & $1.1$ & $1.2$ & $0.8$ & $0.9$ & $1.0$ & $1.1$ & $1.2$ \\
		\midrule
		\multicolumn{11}{l}{\textit{Panel A: $\mu = \mu_1$}}\\	
		200  & 0.029 & 0.047 & 0.094 & 0.118 & 0.150 & 0.002 & 0.003 & 0.008 & 0.015 & 0.021 \\
		500  & 0.010 & 0.023 & 0.033 & 0.060 & 0.085 & 0.000 & 0.002 & 0.006 & 0.013 & 0.022 \\
		1000 & 0.010 & 0.013 & 0.022 & 0.041 & 0.070 & 0.003 & 0.005 & 0.010 & 0.013 & 0.025 \\
		\midrule
		\multicolumn{11}{l}{\textit{Panel B: $\mu = \mu_2$}}\\			
		200  & 0.026 & 0.039 & 0.051 & 0.071 & 0.101 & 0.001 & 0.004 & 0.007 & 0.015 & 0.026 \\
		500  & 0.014 & 0.024 & 0.037 & 0.053 & 0.078 & 0.002 & 0.004 & 0.011 & 0.017 & 0.028 \\
		1000 & 0.003 & 0.014 & 0.028 & 0.044 & 0.058 & 0.000 & 0.002 & 0.004 & 0.016 & 0.033 \\
		\midrule
		\multicolumn{11}{l}{\textit{Panel C: $\mu = \mu_3$}}\\	
		200  & 0.000 & 0.000 & 0.003 & 0.020 & 0.104 & 0.000 & 0.000 & 0.000 & 0.000 & 0.000 \\
		500  & 0.000 & 0.000 & 0.001 & 0.051 & 0.246 & 0.000 & 0.000 & 0.000 & 0.001 & 0.019 \\
		1000 & 0.000 & 0.000 & 0.000 & 0.080 & 0.376 & 0.000 & 0.000 & 0.000 & 0.001 & 0.065 \\
		\midrule
		\midrule
		$n \quad \setminus \quad \Delta$ & $0.1$ & $0.2$ & $0.3$ & $0.4$ & $0.5$ & $0.1$ & $0.2$ & $0.3$ & $0.4$ & $0.5$ \\
		\midrule
		\multicolumn{11}{l}{\textit{Panel D: $\mu = \mu_4$}}\\			
		200  & 0.000 & 0.000 & 0.000 & 0.040 & 0.243 & 0.000 & 0.000 & 0.000 & 0.000 & 0.009 \\
		500  & 0.000 & 0.000 & 0.006 & 0.154 & 0.464 & 0.000 & 0.000 & 0.000 & 0.003 & 0.086 \\
		1000 & 0.000 & 0.000 & 0.008 & 0.286 & 0.600 & 0.000 & 0.000 & 0.000 & 0.012 & 0.263 \\
		\bottomrule
	\end{tabular}
\end{table}

\subsection{Case Study}

We showcase the methodology with one medical application and one engineering application.

\subsubsection{Mean Arterial Pressure}

We illustrate the method on invasive mean arterial pressure (MAP) recordings from VitalDB\footnote{\url{https://physionet.org/content/vitaldb/1.0.0/}} \citep{lee2022}. MAP is the average arterial pressure over a cardiac cycle and is routinely used by anesthesiologists as a one-dimensional summary of organ perfusion during surgery. We use the track \texttt{Solar8000/ART\_MBP}, measured in mmHg, and analyze the first 250 VitalDB cases for which this variable is available. This corresponds to approximately half of the first 500 recordings. 

The clinical target is motivated by the consensus recommendations of the Anesthesia Patient Safety Foundation, which state that an exact MAP target is imprecise, but that MAP below 65 mmHg increases harm, MAP below 60 mmHg has a high probability of harm, and MAP above 80 mmHg appears to confer no advantage \citep{scott2024}. We therefore ask whether a case contains a 10-minute interval during which the underlying MAP trajectory remains in the range 65-80 mmHg. This is a stable-interval question rather than a global constancy question. Intraoperative MAP is expected to change over time, but the relevant issue is whether a clinically meaningful stable episode occurs. Formally, we set $d(t)=\mu(t)-72.5$ and $\Delta=7.5$, so that $|d(t)|<\Delta$ corresponds to $\mu(t)\in(65,80)$. For each recording we test $H_0: d_\infty\geq 7.5$ against $H_1:d_\infty<7.5$, where $\delta$ corresponds to 10 minutes. Thus rejection means that the data provide evidence for at least one 10-minute interval on which the smoothed MAP curve remains within the target range. This is exactly the infimum-over-windows formulation of the stable-interval functional, where the inner supremum measures the worst deviation in a candidate interval and the outer infimum selects the best interval.

Before applying the test, we used a deterministic artifact-removal procedure. First, changes larger than 20 mmHg between two consecutive observations at 0.5 Hz were marked as missing, corresponding to changes exceeding 10 mmHg/s, since such changes are biologically implausible \citep{khan2022}. Second, leading and trailing missing or implausible values were removed, values outside $[30,220]$ mmHg were treated as biologically implausible, and short feasible boundary runs of length at most 3\% of the recording were ignored when determining the analyzable segment \citep{khan2022}. Third, remaining short gaps were imputed by a moving average over 31 observations, corresponding to approximately one minute at 0.5 Hz. The same preprocessing rule was applied to all 250 recordings and was fixed before computing the test statistics.

For each preprocessed recording we computed the plug-in statistic $\hat d_{n,\infty}$ and rejected for large values of $T_n(7.5)$. The Gumbel-calibrated test, as defined by \eqref{eq:decision_rule}, rejected 51 of the 250 cases, corresponding to an empirical rejection rate of 20.4\%. The Gaussian-calibrated test, defined by \eqref{eq:gaussian_calibrated_rule}, rejected in 70 cases, corresponding to an empirical rejection rate of 28\%.

These rates are descriptive rather than epidemiological, because the cases were selected by availability of the invasive MAP track rather than by a clinical sampling design.

Figure \ref{fig:vitaldb} shows representative cases, including the preprocessed MAP, the target band 65-80 mmHg, and the first selected 10-minute interval when the null is rejected. The example demonstrates the intended interpretation of the method. It does not claim that the complete MAP trajectory is stable, but tests whether at least one scientifically specified stable interval exists within a non-stationary physiological time series. 

\begin{figure}[htb]
	\centering
	\includegraphics[width=\textwidth]{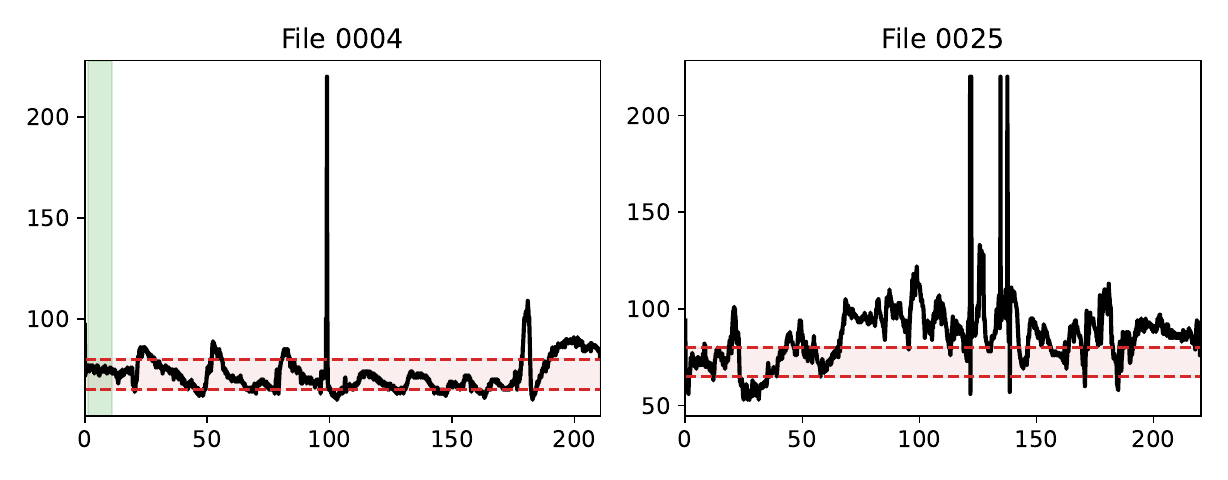}
	\caption{MAP curve over approx. two hours of two recordings. Left: The null hypothesis is rejected and the first stable interval marked in green. Right: The null hypothesis cannot be rejected and no stable interval exists.}
	\label{fig:vitaldb}
\end{figure}

\subsubsection{Battery Voltage}

As a second illustration, we analyze low-current open-circuit-voltage (OCV) discharge curves from the CALCE A123 battery data set.\footnote{\url{https://calce.umd.edu/battery-data\#A123}} The A123 experiment provides low-current OCV measurements at several ambient temperatures, and OCV-based information is a standard ingredient in state-of-charge estimation for battery management systems \citep{xing2014}. We use the discharge recordings at $0^\circ$C, $25^\circ$C, and $50^\circ$C for the two cells labeled \texttt{7} and \texttt{8}. Each recording lasts approximately 21 hours and is sampled at 0.2 Hz. The voltage curves are smooth, and we use local linear estimation with a bandwidth corresponding to 30 minutes.

The scientific question is whether the discharge curve contains a voltage plateau of a prescribed minimal duration. This is naturally formulated with the derivative signal $d(t) = \mu'(t)$, so that stability means local flatness of the underlying voltage curve. All derivatives are reported in physical units, V/h. We set the slope tolerance to $\Delta=0.01$ V/h, that is, 10 mV per hour, and vary the minimal plateau length $\delta$ between 2 and 8 hours. Thus, for each battery and temperature, we test $H_0: \inf_t \sup_{s\in[t,t+\delta]} |\mu'(s)| \geq 0.01$ against $H_1: \inf_t \sup_{s\in[t,t+\delta]} |\mu'(s)| < 0.01$. Rejection means that the voltage curve contains an interval of length at least $\delta$ on which the underlying voltage changes by less than 10 mV/h in absolute local slope. This is precisely the derivative-based plateau version of the stable-interval problem, where the method searches for the best interval rather than testing whether the full trajectory is constant.

The theory is formulated on the unit interval. If $m(u)$ denotes the voltage mean as a function of physical time $u\in[0,T]$, measured in hours, we apply the theory to the rescaled mean $\mu(t)=m(Tt)$, for $t\in[0,1]$. Thus a physical plateau length $\delta_{\rm h}$ corresponds to $\delta=\delta_{\rm h}/T$. For derivative-based stability, $\mu'(t)=T m'(Tt)$, so a physical slope tolerance $\Delta_{\rm h} = 0.01$ in V/h corresponds to the rescaled tolerance $\Delta=T\Delta_{\rm h}$ in the test. Reported derivatives in V/h are obtained by dividing the derivative on the rescaled time scale by $T$.

Figure \ref{fig:battery} shows the voltage curve and estimated derivative for battery \texttt{7}. The green interval marks the first interval satisfying the prescribed tolerance for $\delta=2$ h, while the blue interval marks the minimax-optimal stable interval. The distinction is useful in applications: $\hat\tau_\Delta$ identifies the first scientifically acceptable plateau, whereas $\hat\tau_*$ identifies the first interval with smallest estimated window-wise deviation.

\begin{figure}[t]
	\centering
	\includegraphics[width=\textwidth]{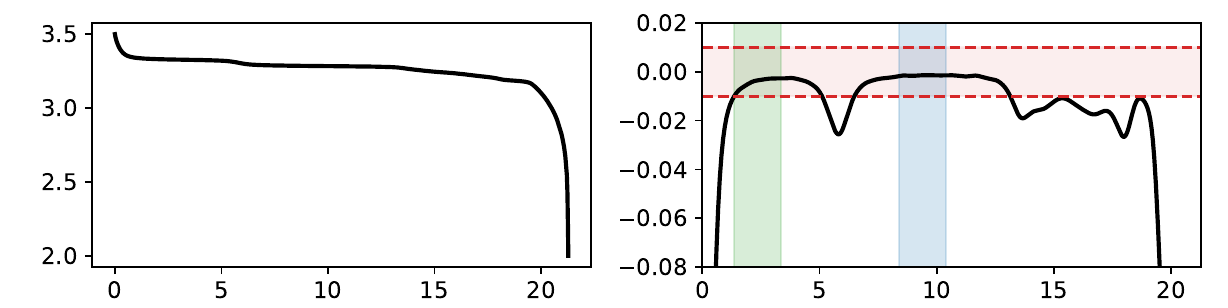}
	\caption{Left: Voltage (in V) of battery \texttt{7} over time (in h). Right: Estimated derivative of voltage (in V/h) over time (in h). The first stable interval for $\delta = 2$ h, defined by $\hat{\tau}_\Delta, \Delta = 0.01$ V/h, is marked in green, whereas the optimal stable interval for $\delta = 2$ h, defined by $\hat{\tau}_*$, is marked in blue.}
	\label{fig:battery}
\end{figure}

Table \ref{tab:battery_results} reports $p$-values for both the Gaussian and Gumbel calibrations. For $\delta=2,5,$ and $6$ hours, all battery-temperature combinations strongly reject the null, giving clear evidence for voltage plateaus of at least six hours. For $\delta=7$ hours, the conclusion becomes calibration- and temperature-dependent. At $0^\circ$C the Gumbel calibration still rejects for both cells, whereas the Gaussian calibration does not, and at $25^\circ$C and $50^\circ$C neither calibration rejects.  For $\delta=8$ hours, no test rejects. Thus the example illustrates the intended quantitative interpretation of the method. Instead of merely confirming visually that the discharge curve has a flat region, the procedure attaches a duration, a physical slope tolerance, and a calibrated uncertainty statement to the existence of a voltage plateau.

\begin{table}[thb]
	\caption{$p$-values for tests based on the voltage curve of batteries \texttt{7} and \texttt{8} for varying $\delta$.}
	\label{tab:battery_results}
	\begin{tabular}{l|rrrrr|rrrrr}
	\toprule
		& \multicolumn{5}{c|}{Gaussian calibration \eqref{eq:gaussian_calibrated_rule}} & \multicolumn{5}{c}{Gumbel calibration \eqref{eq:decision_rule}} \\
		$\delta$ (in h) & 2 & 5 & 6 & 7 & 8 & 2 & 5 & 6 & 7 & 8 \\
		\midrule
		\multicolumn{11}{l}{\textit{Panel A: Battery \texttt{7}}}\\
		0°C & 0.000 & 0.000 & 0.002 & 0.423 & 1.000 & 0.000 & 0.000 & 0.000 & 0.018 & 1.000 \\  
		25°C & 0.000 & 0.000 & 0.000 & 1.000 & 1.000 & 0.000 & 0.000 & 0.000 & 1.000 & 1.000 \\ 
		50°C & 0.000 & 0.000 & 0.000 & 1.000 & 1.000& 0.000 & 0.000 & 0.000 & 1.000 & 1.000 \\
		\midrule
		\multicolumn{11}{l}{\textit{Panel B: Battery \texttt{8}}}\\
		0°C & 0.001 & 0.001 & 0.001 & 0.240 & 1.000 & 0.001 & 0.001 & 0.001 & 0.002 & 1.000 \\ 
		25°C & 0.001 & 0.001 & 0.001 & 1.000 & 1.000 & 0.001 & 0.001 & 0.001 & 1.000 & 1.000 \\ 
		50°C & 0.001 & 0.001 & 0.001 & 1.000 & 1.000 & 0.001 & 0.001 & 0.001 & 1.000 & 1.000 \\
		\bottomrule
	\end{tabular}
\end{table}

%% file: sec_proofs.tex
 \section{Proofs} \label{sec:proof}

\subsection{Proofs of Section \ref{sec:localized_test}}

\begin{proposition} \label{prop:extremal_set}
	Let $(\phi_n)_{n\in\N}$ be a sequence converging to $0$. Under Assumption \ref{assump:mu} (1), it exists a sequence $(\delta_n)_{n\in\N}$ converging to $0$, such that 
	\begin{equation} \label{eq:extremal_set_inclusion1}
		U_{\phi_n/L_d}(\Ec) \subset \Ec(\phi_n) \subset U_{\delta_n}(\Ec),
	\end{equation}
	where $L_d$ denotes the Lipschitz constant of $d$. 	
	If additionally Assumption \ref{assump:mu} (2) or (3) is satisfied, we may select $\delta_n = (\phi_n/c)^{1/\alpha_{\max}}$, where $\alpha_{\max} = \max_{(x, \nu)} \alpha_{x, \nu}$ and $c$ is the constant in \eqref{eq:boundary_condition} or \eqref{eq:boundary_condition2}, respectively.
\end{proposition}

\begin{proof}
	To prove the first inclusion in \eqref{eq:extremal_set_inclusion1}, let $t \in U_{\phi_n/L_d}(\Ec)$. Then, there exists $x_t\in\Ec$ with $|t - x_t| < \phi_n/L_d$. Since $x_t\in\Ec\subset \Ec_1$, we have $|d(x_t)|=d_\infty$. By Lipschitz continuity of $d$, $d_\infty - |d(t)| \le L_d |x_t - t| \le \phi_n$, so that $t \in \Ec_1(\phi_n)$.
	Further, since $x_t\in\Ec\subset \Ec_2 =\Ec_3+[0,\delta]$, there exists $u_t\in\Ec_3$ and $r_t\in[0, \delta]$ such that $x_t = u_t + r_t$. Define $v_t = t - r_t \in[t-\delta, t]$, so that $|v_t - u_t| = | t - x_t| < \phi_n /L_d$. Because $d$ is Lipschitz with constant $L_d$, the function $u\mapsto \sup_{s\in[u,(u+\delta)\wedge 1]}|d(s)|$ is also Lipschitz with the same constant. Since $u_t\in\Ec_3$, $\sup_{s\in[u_t,u_t+\delta]}|d(s)|=d_\infty$. Hence, 
	\begin{equation*}
		\sup_{s\in[v_t,(v_t+\delta)\wedge 1]}|d(s)| - d_\infty 
		= \sup_{s\in[v_t,(v_t+\delta)\wedge 1]}|d(s)| - \sup_{s\in[u_t,(u_t+\delta)\wedge 1]}|d(s)|
		\le L_d |v_t - u_t| < \phi_n.
	\end{equation*}
	Thus, $v_t \in \Ec_3(\phi_n)$, and $t = v_t + r_t \in \Ec_3(\phi_n) + [0, \delta] = \Ec_2(\phi_n)$. Consequently, $t\in \Ec_1(\phi_n)\cap \Ec_2(\phi_n) =\Ec(\phi_n)$, and $U_{\phi_n/L_d}(\Ec)\subset \Ec(\phi_n)$.
	
	For the second inclusion, it suffices to show that $\sup_{t\in \Ec(\phi_n)} \dist(t,\Ec)\to 0$. Suppose this is false. Then there exist $\eps >0 $, a subsequence still indexed by $n$, and points $t_n\in \Ec(\phi_n)$ such that $\dist(t_n,\mathcal E)\ge \eps$ for all $n$. By compactness of $[0,1]$, after passing to a further subsequence, we may assume $t_n\to t^\ast\in[0,1]$.	Since $t_n\in \Ec(\phi_n)\subset \Ec_1(\phi_n)$, we have $\min\{|d_\infty-d(t_n)|, |d_\infty+d(t_n)|\}\le \phi_n$. As $\phi_n\to 0$ and $d$ is continuous, it follows that $t^\ast\in \Ec_1$.
	Moreover, $t_n\in \Ec(\phi_n)\subset \Ec_2(\phi_n)=\Ec_3(\phi_n)+[0,\delta]$, so there exists $u_n\in \Ec_3(\phi_n)$, such that $t_n\in [u_n,u_n+\delta]$. Passing to a further subsequence if necessary, we may assume $u_n\to u^\ast\in[0,1-\delta]$. Since $u_n\in \Ec_3(\phi_n)$, $\sup_{s\in[u_n,u_n+\delta]}|d(s)|\le d_\infty+\phi_n$.
	By continuity of $d$, the map $u\mapsto \sup_{s\in[u,u+\delta]}|d(s)|$ is continuous, hence $\sup_{s\in[u^\ast,u^\ast+\delta]}|d(s)|\le d_\infty$. By definition of $d_\infty$, the reverse inequality always holds, so equality follows and $u^\ast\in \Ec_3$. Since $t_n\in [u_n,u_n+\delta]$ and $(t_n,u_n)\to (t^\ast,u^\ast)$, we obtain $t^\ast\in [u^\ast,u^\ast+\delta]\subset \Ec_2$. Therefore $t^\ast\in \Ec_1\cap \Ec_2=\Ec$. This contradicts $\dist(t_n,\mathcal E)\ge \eps$ for all $n$. Hence $\sup_{t\in \Ec(\phi_n)} \dist(t, \Ec) \to 0$.
	
	For the second part, note that the two-sided bound in \eqref{eq:boundary_condition2} is more restrictive than the one-sided bound in \eqref{eq:boundary_condition}.	For all sufficiently large $n$, $(\phi_n/c)^{1/\alpha_{\max}} < \gamma$. Let $t\in\Ec(\phi_n)\setminus \Ec$, then $t \in \Ec(\phi_n) \cap U_\gamma^\nu(x)$, for some $(x, \nu)\in (\partial \Ec)^{ext}$. By \eqref{eq:boundary_condition} and since $t\in\Ec(\phi_n)$, $c |t - x|^{\alpha_{x,\nu}} \le d_\infty - |d(t)| \le \phi_n$, and therefore, $|t - x|\le(\phi_n/c)^{1/\alpha_{x,\nu}} \le (\phi_n/c)^{1/\alpha_{\max}}$. Hence, $t \in U_{(\phi_n/c)^{1/\alpha_{\max}}}(\Ec)$.
\end{proof}

\begin{proposition} \label{prop:extremal_set2}
	Let $(\phi_n)_{n\in\N}$ and $(e_n)_{n\in\N}$ be positive sequences converging to $0$, Assumptions \ref{assump:mu} (1) and (2) be satisfied and define  $\alpha_{\max} = \max_{(x, \nu)} \alpha_{x, \nu}$. For some constant $C > 0$,
	\begin{equation*}
		\lambda\big(\Ec(\phi_n(1+e_n)) \setminus \Ec(\phi_n(1-e_n))\big) \le C \phi_n^{1/\alpha_{\max}}. 
	\end{equation*}
	If Assumptions \ref{assump:mu} (1) and (3) are satisfied, the stronger bound
	\begin{equation*}
		\lambda\big(\Ec(\phi_n(1+e_n)) \setminus \Ec(\phi_n(1-e_n))\big) \le
		C e_n \phi_n^{1/\alpha_{\max}}
	\end{equation*}
	holds. The same bounds apply to the Hausdorff distance $d_H[\Ec(\phi_n(1+e_n)), \Ec(\phi_n(1-e_n))]$.
\end{proposition}

\begin{proof}
	First consider that Assumptions \ref{assump:mu} (1) and (2) are satisfied. By Proposition \ref{prop:extremal_set}, $\Ec(\phi_n(1+e_n)) \subset U_{(\phi_n(1+e_n)/c)^{1/\alpha_{\max}}}(\Ec)$, hence, 
	\begin{equation*}
		\lambda\big[\Ec(\phi_n(1+e_n)) \setminus \Ec(\phi_n(1-e_n))\big]
		\le \lambda\big[U_{(\phi_n/c)^{1/\alpha_{\max}}}(\Ec) \setminus \Ec\big]
		\le 2 L \bigg(\frac{\phi_n(1+e_n)}{c}\bigg)^{1/\alpha_{\max}},
	\end{equation*}
	which can be bounded from above by $C \phi_n^{1/\alpha_{\max}}$, for a suitable constant $C > 0$. The Hausdorff bound $d_H[\Ec(\phi_n(1+e_n)), \Ec(\phi_n(1-e_n))] \le C \phi_n^{1/\alpha_{\max}}$ follows immediately since $\Ec \subset \Ec(\phi_n(1-e_n))$ and $\Ec(\phi_n(1+e_n)) \setminus \Ec(\phi_n(1-e_n)) \subset U_{(\phi_n/c)^{1/\alpha_{\max}}}(\Ec) \setminus \Ec$.
	
	Now, let Assumptions \ref{assump:mu} (1) and (3) hold. For $n$ sufficiently large, $\phi_n (1+e_n) \le \phi_0$, such that
	\begin{align*}
		\lambda\big(\Ec(\phi_n(1+e_n)) \setminus \Ec(\phi_n(1-e_n))\big)
		& \le \sum_{(x, \nu) \in (\partial \Ec)^{ext}} r_{x, \nu}\big(\phi_n(1+e_n)\big) - r_{x, \nu}\big(\phi_n(1-e_n)\big) \\
		& = \sum_{(x, \nu) \in (\partial \Ec)^{ext}} \int_{\phi_n(1-e_n)}^{\phi_n(1+e_n)}r_{x, \nu}'(z) \diff z \\
		& \le 2 L C_{\max} \int_{\phi_n(1-e_n)}^{\phi_n(1+e_n)} z^{1/\alpha_{\max} - 1} \diff z \\
		& = 2 L C_{\max} \alpha_{\max}   \phi_n^{1/\alpha_{\max}} [(1+e_n)^{1/\alpha_{\max}} - (1-e_n)^{1/\alpha_{\max}}]
	\end{align*}
	where $C_{\max} = \max_{(x, \nu)} C_{x, \nu}$. By the mean value theorem, the right-hand side can be bounded from above by $C e_n \phi_n^{1/\alpha_{\max}}$.
	
	The Hausdorff bound follows similarly. More specifically, $\Ec(\phi_n(1\pm e_n)) \cap U_\gamma^\nu(x) = U_{r_{x,\nu}(\phi_n(1\pm e_n))}^\nu(x)$, by assumption. Therefore, any point $t\in \Ec(\phi_n(1+e_n)) \setminus \Ec(\phi_n(1-e_n))$ lies in one of the one-sided sets $U_{r_{x,\nu}(\phi_n(1+ e_n))}^\nu(x) \setminus U_{r_{x,\nu}(\phi_n(1- e_n))}^\nu(x)$, whose width is
	\begin{align*}
		|r_{x,\nu}(\phi_n(1+ e_n)) - r_{x,\nu}(\phi_n(1- e_n))|
		= \int_{\phi_n(1-e_n)}^{\phi_n(1+e_n)} r_{x,\nu}'(z) \diff z,
	\end{align*}
	which is of order $\Oc(e_n \phi_n^{1/\alpha_{\max}})$ by the same arguments as before. In particular, $d_H[\Ec(\phi_n(1+e_n)), \Ec(\phi_n(1-e_n))] \le C e_n \phi_n^{1/\alpha_{\max}}$.
\end{proof}

\begin{proposition} \label{prop:sigma_max_plugin}
	Let $(\phi_n)_{n\in\N}$ and $(e_n)_{n\in\N}$ be positive sequences converging to $0$ and $\ell_n = o(n^{1/7})$. Further, let Assumptions \ref{assump:mu} (1) and (2) be satisfied with $\phi_n^{1/\alpha_{\max}} = o(\ell_n^{-2})$, or Assumptions \ref{assump:mu} (1) and (3) be satisfied with $e_n \phi_n^{1/\alpha_{\max}} = o(\ell_n^{-2})$. Then,
	\begin{equation*}
		\sup_{t\in \hat{\Ec}(\phi_n)} \hat{\sigma}_n(t) = \sup_{t\in \Ec(\phi_n)} \sigma(t) (1 + o_\pr(\ell_n^{-2})).
	\end{equation*}
	Equivalently, since $a_n / \ell_n \to 1$,
	\begin{equation*}
		\sup_{t\in \hat{\Ec}(\phi_n)} \hat{\sigma}_n(t) = \sup_{t\in \Ec(\phi_n)} \sigma(t) (1 + o_\pr(a_n^{-2})).
	\end{equation*}
\end{proposition}

\begin{proof}
	Let $C_\sigma$ denote the Lipschitz constant of $\sigma$. 
	On the event $\Ec\big(\phi_n (1 - e_n)\big) \subset \hat{\Ec}(\phi_n) \subset 	\Ec\big(\phi_n (1 + e_n)\big)$,
	\begin{equation*}
		\sup_{t\in \Ec(\phi_n (1 - e_n))} \sigma(t) \le \sup_{t\in \hat{\Ec}(\phi_n)} \sigma(t) \le \sup_{t\in \Ec(\phi_n (1 + e_n))} \sigma(t).
	\end{equation*}
	Lipschitz continuity gives
	\begin{equation*}
		\Big|\sup_{t\in \Ec(\phi_n (1 + e_n))} \sigma(t)-\sup_{t\in \Ec(\phi_n (1 - e_n))}\sigma(t)\Big| \le C_\sigma d_H[\Ec(\phi_n(1+e_n)), \Ec(\phi_n(1-e_n))].
	\end{equation*}
	The right-hand side is of order $o(\ell_n^{-2})$ by Proposition \ref{prop:extremal_set2} and the assumptions. In particular,
	\begin{equation*}
		\Big|\sup_{t\in \hat{\Ec}(\phi_n)} \sigma(t) - \sup_{t\in \Ec(\phi_n)} \sigma(t)\Big| = o_\pr(\ell_n^{-2}).
	\end{equation*}
	By \eqref{eq:rate_sigma} and $\ell_n = o(n^{1/7})$, $\sup_{t\in [\gamma_n, 1-\gamma_n]} |\sigma^2(t) - \hat{\sigma}_n^2(t)| = o_\pr(\ell_n^{-2})$, such that
	\begin{equation*}
		\Big|\sup_{t\in \hat{\Ec}(\phi_n)} \hat{\sigma}_n(t) - \sup_{t\in \hat{\Ec}(\phi_n)} \sigma(t)\Big| 
		\le \sup_{t\in [\gamma_n, 1-\gamma_n]} \frac{|\sigma^2(t) - \hat{\sigma}_n^2(t)|}{|\sigma(t) + \hat{\sigma}_n(t)|} = o_\pr(\ell_n^{-2}),
	\end{equation*}
	by the triangle inequality. Therefore,
	\begin{equation*}
		\Big| \sup_{t\in \hat{\Ec}(\phi_n)} \hat{\sigma}_n(t)  - \sup_{t\in \Ec(\phi_n)} \sigma(t) \Big| = o_\pr(\ell_n^{-2})
	\end{equation*}
	Since $\sigma$ is bounded away from zero, 
	\begin{equation*}
		\frac{\sup_{t\in \hat{\Ec}(\phi_n)} \hat{\sigma}_n(t)}{\sup_{t\in \Ec(\phi_n)} \sigma(t)} = 1 + o_\pr(\ell_n^{-2}).
	\end{equation*}
	Finally, the definition of $a_n$ implies $a_n/\ell_n \to 1$, so $\ell_n^{-2}$ may be replaced by $a_n^{-2}$. 
\end{proof}

\begin{proposition}	\label{prop:W_plugin}
	Let $(\phi_n)_{n\in\N}$ and $(e_n)_{n\in\N}$ be positive sequences converging to $0$. Let $(h_n)_{n\in\N}$ and $(\rho_n)_{n\in\N}$ satisfy Assumption \ref{assump:sequences}. Further, let Assumptions \ref{assump:mu} (1) and (2) be satisfied with $\phi_n^{1/\alpha_{\max}} = o(\ell_n^{-2})$, or Assumptions \ref{assump:mu} (1) and (3) be satisfied with $e_n \phi_n^{1/\alpha_{\max}} = o(\ell_n^{-2})$.  Finally, let Assumptions \ref{assump:error} and \ref{assump:regularity_sigma} (1) or (2) be satisfied. 
	Then
	\begin{equation*}
		\log \widehat W_n(\hat \Ec(\phi_n)) = \log W_n(\Ec(\phi_n)) + o_\pr(1),
	\end{equation*}
	and, consequently,
	\begin{equation*}
		\hat a_n(\hat \Ec(\phi_n))-a_n(\Ec(\phi_n)) = o_\pr(\ell_n^{-1}),
		\qquad
		\frac{\hat a_n(\hat \Ec(\phi_n))}{a_n(\Ec(\phi_n))} = 1 + o_\pr(\ell_n^{-2}).
	\end{equation*}
\end{proposition}

\begin{proof} 
	Write
	\begin{align*}
		\log \widehat W_n(\hat \Ec(\phi_n))-\log W_n(\Ec(\phi_n))
		= & \{\log \widehat W_n(\hat \Ec(\phi_n))-\log W_n(\hat \Ec(\phi_n))\} \\
		& + \{\log W_n(\hat \Ec(\phi_n))-\log W_n(\Ec(\phi_n))\}.
	\end{align*}
	By Assumption \ref{assump:mu} (2) or (3) and Proposition \ref{prop:extremal_set}, $\hat{\Ec}(\phi_n)$ satisfies Assumption \ref{assump:sets}, so that 
	\begin{equation} \label{eq:Wn_stochastic_reduction}
		\log \widehat W_n(\hat \Ec(\phi_n)) - \log W_n(\hat \Ec(\phi_n)) = o_\pr(1)
	\end{equation}
	by Lemma \ref{lem:wn_estimation}.
	It remains to compare $W_n(\hat \Ec(\phi_n))$ and $W_n(\Ec(\phi_n))$. 
	On the event $\Ec(\phi_n(1-e_n))\subset \hat \Ec(\phi_n)\subset \Ec(\phi_n(1+e_n))$, the two sets $\Ec(\phi_n)$ and $\hat \Ec(\phi_n)$ differ only inside the deterministic shell $\Ec(\phi_n(1+e_n))\setminus \Ec(\phi_n(1-e_n))$. The number of block representatives that can be added or removed by this replacement is bounded by
	\begin{equation} \label{eq:bound_summands}
		C \left(\frac{\lambda(\Ec(\phi_n(1+e_n))\setminus \Ec(\phi_n(1-e_n)))}{\rho_n} + 1\right),
	\end{equation}
	for some constant $C >0$ depending only on the finite-union complexity of the sets. 
	Each summand in $W_n(\Ec(\phi_n))$ is bounded by one, so the contribution of added or removed blocks is at most the previous display.
	Next, Proposition \ref{prop:sigma_max_plugin} and Lipschitz continuity of $\sigma$ imply
	\begin{equation} \label{eq:rate_sigma2}
		\ell_n^2 | \sup_{t\in \hat{\Ec}(\phi_n)} \sigma^2(t) - \sup_{t\in \Ec(\phi_n)} \sigma^2(t) | = o_\pr(1)
	\end{equation}
	on the same event. 
	Changing $\sup_{t\in \Ec(\phi_n)} \sigma^2(t)$ to $\sup_{t\in \hat{\Ec}(\phi_n)} \sigma^2(t)$ in the exponential weights therefore multiplies $W_n(\Ec(\phi_n))$ by $1 + o(1)$, uniformly over all block representatives.
	Combining \eqref{eq:bound_summands} and \eqref{eq:rate_sigma2}, 
	\begin{equation*}
		|W_n(\hat \Ec(\phi_n))-W_n(\Ec(\phi_n))| \le o(1)W_n(\Ec(\phi_n)) + C \frac{\lambda(\Ec(\phi_n(1+e_n))\setminus \Ec(\phi_n(1-e_n))) + \rho_n}{\rho_n}.
	\end{equation*}
	By Proposition \ref{prop:sequences}, $\ell_n^{-2} \rho_n^{-1} \le W_n(\Ec(\phi_n))$, such that 
	\begin{equation*}
		\frac{\lambda(\Ec(\phi_n(1+e_n))\setminus \Ec(\phi_n(1-e_n)))+\rho_n}{W_n(\Ec(\phi_n))\rho_n}
		\le
		C\{\lambda(\Ec(\phi_n(1+e_n))\setminus \Ec(\phi_n(1-e_n))) \ell_n^2 + \rho_n\ell_n^2\}.
	\end{equation*}	
	The first term converges to zero by Proposition \ref{prop:extremal_set2}, whereas the second term vanishes by Assumption \ref{assump:sequences}. 
	Thus, $W_n(\hat \Ec(\phi_n))/W_n(\Ec(\phi_n)) = 1+o_\pr(1)$, and hence
	\begin{equation*}
		\log W_n(\hat \Ec(\phi_n))-\log W_n(\Ec(\phi_n)) = o_\pr(1).
	\end{equation*}
	Together with \eqref{eq:Wn_stochastic_reduction}, the first part of the proposition follows. Finally, the second part follows from the same algebra as in Lemma \ref{lem:wn_estimation}. If $\log \widehat W_n(\hat \Ec(\phi_n))-\log W_n(\Ec(\phi_n))=o_\pr(1)$, then the corresponding corrected centering constants satisfy
	\begin{equation*}
		\hat a_n(\hat \Ec(\phi_n))-a_n(\Ec(\phi_n))=o_\pr(\ell_n^{-1}),
		\qquad
		\hat a_n(\hat \Ec(\phi_n))/a_n(\Ec(\phi_n))=1+o_\pr(\ell_n^{-2}).
	\end{equation*}
\end{proof}

\noindent \textbf{Proof of Proposition \ref{prop:E_sandwich}.}
We will show separately, for $i=1, 2$, that
\begin{equation} \label{eq:sandwich}
	\pr\Big(\Ec_i\big(\phi_n (1 - e_n)\big) \subset \hat{\Ec}_i(\phi_n) \subset \Ec_i\big(\phi_n (1 + e_n)\big)\Big) \to 1.
\end{equation}

By Theorem A.1 of \cite{bucher2021}, $\| \hat{d}_n - d\|_\infty = \Oc_\pr(\sqrt{|\log(h_n)|}/r_n)$. By assumption, $\sqrt{|\log(h_n)|}/(\phi_n r_n) = o(e_n)$, such that $\| \hat{d}_n - d\|_\infty = o_\pr(e_n \phi_n)$.
Let $t \in \Ec_1\big(\phi_n (1 - e_n)\big)$, then $\big| |d(t)| - d_\infty \big| \le \phi_n(1 - e_n)$. By the triangle inequality,
\begin{align*}
	\big| |\hat{d}_n(t)| - \hat{d}_{n, \infty} \big| 
	& \le \big| |d(t)| - d_\infty \big| + \big| |\hat{d}_n(t)| - |d(t)| \big| + | \hat{d}_{n, \infty} - d_\infty | \\
	& \le \phi_n (1-e_n) + 2 \| \hat{d}_n - d\|_\infty \\
	& = \phi_n - \phi_n e_n (1 + o_\pr(1)),
\end{align*}
which is less than $\phi_n$ with probability converging to $1$. In particular, $t \in \hat{\Ec}_1(\phi_n)$, hence, $\Ec_1\big(\phi_n (1 - e_n)\big) \subset \hat{\Ec}_1(\phi_n)$. Similarly, for $t\in \hat{\Ec}_1(\phi_n)$,
\begin{equation*}
	\big| |d(t)| - d_\infty \big| \le  \phi_n + o_\pr(\phi_n e_n),
\end{equation*}
which is less than $\phi_n(1+e_n)$ with probability converging to $1$, such that $\hat{\Ec}_1(\phi_n) \subset \Ec_1\big(\phi_n (1 + e_n)\big)$, and \eqref{eq:sandwich} follows for $i=1$.

Let $t \in \Ec_3\big(\phi_n (1 - e_n)\big)$, then $\sup_{s\in[t, t+\delta]}|d(s)| - d_\infty \le \phi_n(1 - e_n)$. By the triangle inequality,
\begin{equation*}
	\sup_{s\in[t, t+\delta]}|\hat{d}_n(s)| - \hat{d}_{n,\infty}
	\le \sup_{s\in[t, t+\delta]}|d(s)| - d_\infty  + 2 \| \hat{d}_n - d\|_\infty,
\end{equation*}
where the right-hand side is bounded from above by $\phi_n$, with probability converging to $1$, by same arguments as before. Hence, $\Ec_3\big(\phi_n (1 - e_n)\big) \subset \hat{\Ec}_3(\phi_n)$, and similarly, $\hat{\Ec}_3(\phi_n)\subset \Ec_3\big(\phi_n (1 + e_n)\big)$, where the same follows for $\Ec_2$ by definition. 
\hfill $\blacksquare$

\noindent \textbf{Proof of Proposition \ref{prop:localization}.}
By Theorem A.1 of \cite{bucher2021}, 
\begin{equation*}
	\| \hat{d}_n - d\|_\infty = \Oc_\pr(\sqrt{|\log(h_n)|}/r_n) = o_\pr(\phi_n).
\end{equation*}
With probability converging to $1$, $\| \hat{d}_n - d\|_\infty \le \phi_n/2$. For all $t \notin \Ec_3(\phi_n)$,
\begin{align*}
	\sup_{s \in[t, t+\delta]}|\hat{d}_n(s)| 
	\ge \sup_{s \in[t, t+\delta]}|d(s)| - \| \hat{d}_n - d \|_\infty 
	> d_\infty + \phi_n - \| \hat{d}_n - d \|_\infty 
	> d_\infty + \frac{\phi_n}{2}
\end{align*}
with probability converging to $1$. Moreover,
\begin{equation*}
	\inf_{t\in\Ec(\phi_n)} \sgn(d(t)) \{\hat{d}_n(t) - d(t)\}
	\le \sup_{t\in\Ec(\phi_n)} |\hat{d}_n(t) - d(t)| \le \| \hat{d}_n - d\|_\infty \le \frac{\phi_n}{2},
\end{equation*}
and in particular, 
\begin{equation} \label{eq:outer_set}
	\inf_{t\in\Ec(\phi_n)} \sgn(d(t)) \{\hat{d}_n(t) - d(t)\} \le \sup_{s \in[t, t+\delta]}|\hat{d}_n(s)|  - d_\infty,
\end{equation}
for any $t \notin \Ec_3(\phi_n)$.

Now, let $t \in \Ec_3(\phi_n)$, such that $\sup_{s \in[t, t+\delta]}|d(s)| \le d_\infty + \phi_n$. Choose $s_t \in [t, t+\delta]$ such that $|d(s_t)| = \sup_{s \in[t, t+\delta]}|d(s)|$. Then, $d_\infty \le |d(s_t)| \le d_\infty + \phi_n$, such that $s_t \in \Ec_1(\phi_n)$. Since $\sup_{s \in[t, t+\delta]}|d(s)| \le d_\infty + \phi_n$ and $s_t \in [t, t+\delta]$, $s_t \in \Ec_2(\phi_n)$, hence, $s_t \in \Ec(\phi_n)$. Since $d_\infty > 0$, $d(s_t) \neq 0$, and, whenever $\| \hat{d}_n - d\|_\infty \le \phi_n/2$, $\sgn(\hat{d}_n(s_t)) = \sgn(d(s_t))$, so that 
\begin{equation*}
	| \hat{d}_n(s_t) | = |d(s_t) | + \sgn(d(s_t)) \{\hat{d}_n(s_t) - d(s_t)\}.
\end{equation*} 
In particular,
\begin{align*}
	\sup_{s \in[t, t+\delta]}|\hat{d}_n(s)| - d_\infty
	& \ge |\hat{d}_n(s_t)| - d_\infty \\
	& = |d(s_t) | - d_\infty + \sgn(d(s_t)) \{\hat{d}_n(s_t) - d(s_t)\} \\
	& \ge \sgn(d(s_t)) \{\hat{d}_n(s_t) - d(s_t)\} \\
	& \ge \inf_{t\in\Ec(\phi_n)} \sgn(d(t)) \{\hat{d}_n(t) - d(t)\} ,
\end{align*}
for any $t \in \Ec_3(\phi_n)$. In combination with \eqref{eq:outer_set}, we obtain
\begin{equation*}
	\hat{d}_{n,\infty} - d_\infty \ge \inf_{t\in\Ec(\phi_n)} \sgn(d(t)) \{\hat{d}_n(t) - d(t)\},
\end{equation*}
by taking the infimum over $t\in[0, 1-\delta]$, with probability converging to $1$.
\hfill $\blacksquare$

\noindent \textbf{Proof of Theorem \ref{thm:main_conv}.}
By Assumption \ref{assump:mu} (2) or (3) and Proposition \ref{prop:extremal_set}, $\Ec(\phi_n)$ satisfies Assumption \ref{assump:sets}.
By Proposition \ref{prop:localization}, $d_\infty - \hat{d}_{n,\infty} \le \sup_{t\in\Ec(\phi_n)} - \sgn(d(t)) \{\hat{d}_n(t) - d(t)\}$ with probability tending to $1$. In particular,
\begin{align*}
	T_n(d_\infty) 
	& \le \frac{r_n\hat{a}_n}{\|K^*\|_2 \hat{\sigma}_{\max}} \sup_{t\in\Ec(\phi_n)} - \sgn(d(t)) \{\hat{d}_n(t) - d(t)\} - \hat{a}_n^2 + o_\pr(1) \\
	& \le \frac{r_n a_n}{\|K^*\|_2 \sigma_{\max}} \sup_{t\in\Ec(\phi_n)} - \sgn(d(t)) \{\hat{d}_n(t) - d(t)\} - a_n^2 + o_\pr(1)
\end{align*} 
by Propositions \ref{prop:sigma_max_plugin} and \ref{prop:W_plugin}, such that $\liminf_{n\to\infty} \pr(T_n(d_\infty) \le x) \ge \exp(-\exp(-x))$ by \eqref{eq:signed_localized_gumbel}.
\hfill $\blacksquare$

\noindent \textbf{Proof of Corollary \ref{cor:test}.}
First, write 
\begin{equation*}
	T_n(\Delta) = T_n(d_\infty) + \frac{r_n\hat{a}_n}{\|K^*\|_2 \hat{\sigma}_{\max}} \big(\Delta - d_\infty\big).
\end{equation*}
By Theorem \ref{thm:main_conv}, $\liminf_{n\to\infty} \pr(T_n(d_\infty) \le x) \ge \exp(-\exp(-x))$. In particular, 
\begin{align*}
	\limsup_{n\to\infty} \pr(T_n(\Delta) > q_{1-\alpha}) \begin{cases}
		= 1 & \text{if}~\Delta > d_\infty > 0,\\
		\le \alpha  &  \text{if}~\Delta = d_\infty, \\
		= 0 & \text{if}~\Delta < d_\infty.
	\end{cases}
\end{align*}
It only remains to show consistency against alternatives with $d_\infty = 0$. If $d_\infty = 0$, there exists $t_0\in[0, 1-\delta]$, such that $\sup_{s \in[t_0, t_0+\delta]} |d(s)| = 0$, hence $d(s) = 0$ for all $s\in[t_0, t_0+\delta]$. Therefore,
\begin{equation*}
	|\hat{d}_{n,\infty} |
	\le \sup_{s \in[t_0, t_0+\delta]} |\hat{d}_n(s)|
	= \sup_{s \in[t_0, t_0+\delta]} |\hat{d}_n(s) - d(s)|
	\le \|\hat{d}_n - d\|_\infty.
\end{equation*}
Since $\sigma$ is bounded away from $0$ and finite by assumption, $\| \hat{d}_n - d\|_\infty = \Oc_\pr(\sqrt{|\log(h_n)|}/r_n)$, by Theorem A.1 of \cite{bucher2021}. In particular,
\begin{align*}
	T_n(\Delta) 
	& \ge \frac{r_n\hat{a}_n}{\|K^*\|_2 \hat{\sigma}_{\max}} \Delta - \frac{r_n\hat{a}_n}{\|K^*\|_2 \hat{\sigma}_{\max}} \| \hat{d}_n - d \|_\infty - \hat{a}_n^2 \\
	& = \frac{r_n\hat{a}_n}{\|K^*\|_2 \hat{\sigma}_{\max}} \Delta - \hat{a}_n\big(  \Oc_\pr(|\log(h_n)|^{1/2}) - \hat{a}_n\big),
\end{align*}
which diverges to $\infty$, and consistency follows.
\hfill $\blacksquare$

\noindent \textbf{Proof of Proposition \ref{prop:onset_consistency}.}
The bound
\begin{equation*}
	\sup_t|\widehat{M}_n(t)-M(t)| \le \|\hat d_n-d\|_\infty
\end{equation*}
follows from the triangle inequality, and consistency of $\hat d_{n,\infty}$ follows by
taking infima. For $\hat\tau_*$, continuity of $M$ implies that, for every $\eps>0$, the
function $M-d_\infty$ is bounded away from zero on $[0,(\tau_*-\eps)\vee0]$, unless this interval is empty. Thus no point before $\tau_*-\eps$ belongs to $\hat{\Ec}_3(\phi_n)$ with probability tending to one. Since $\tau_*\in\Ec_3$ and $\phi_n$ dominates the uniform estimation error, $\hat{\Ec}_3(\phi_n)$ intersects every right-neighborhood of $\tau_*$ with probability tending to one. Hence $\hat\tau_*\convp\tau_*$.

For $\hat\tau_\Delta$, the absence of false detections before $\tau_\Delta-\eps$ is automatic by the definition of $\tau_\Delta$, continuity of	$M$, and compactness. The robust-onset condition gives a point $t_\eps\in[\tau_\Delta,\tau_\Delta+\eps]$ with
$M(t_\eps)<\Delta$, which is selected by $\hat{\Tc}_\Delta$ with probability tending to one. Therefore
\begin{equation*}
	\pr(\tau_\Delta-\eps \le \hat\tau_\Delta \le \tau_\Delta+\eps) \to 1,
\end{equation*}
and the result follows. 
\hfill $\blacksquare$

\subsection{Proofs of Section \ref{sec:gaussian_evt}} \label{sec:auxiliary_results}

\noindent \textbf{Proof of Theorem \ref{thm:gaussian_approx}.}
By \eqref{eq:rate_sigma},
\begin{equation*}
	\bigg|\frac{1}{\hat{\sigma}_{\max}(\Ac_n)} - \frac{1}{\sigma_{\max}(\Ac_n)}\bigg|
	= \frac{| \hat{\sigma}_{\max}^2(\Ac_n) - \sigma_{\max}^2(\Ac_n)|}{\hat{\sigma}_{\max}(\Ac_n)\sigma_{\max}(\Ac_n) (\hat{\sigma}_{\max}(\Ac_n) + \sigma_{\max}(\Ac_n))}
	= \Oc_\pr(n^{-2/7}).
\end{equation*}
Since $\| \hat{d}_n - d\|_\infty = \Oc_\pr(\sqrt{|\log(h_n)|}/r_n)$, by Theorem A.1 of \cite{bucher2021}, 
\begin{align*}
	\bigg| \frac{r_n}{\hat{\sigma}_{\max}(\Ac_n)} \sup_{t\in \Ac_n} \big(\hat{d}_n(t) - d(t) \big)  - \frac{r_n}{\sigma_{\max}(\Ac_n)} \sup_{t\in \Ac_n} \big(\hat{d}_n(t) - d(t) \big) \bigg|
	= \Oc_\pr\bigg(\frac{\sqrt{|\log h_n|}}{n^{2/7}}\bigg),
\end{align*}
which is of order $o_\pr(a_n^{-1}(\Ac_n))$.
Analogously to the proof of (A.7) in \cite{bucher2021}, 
\begin{align} \label{eq:gaussian_approx1}
	& \frac{r_n}{\sigma_{\max}(\Ac_n)} \sup_{t\in \Ac_n} \big(\hat{d}_n(t) - d(t)\big) \\
	& = \frac{1}{\sqrt{nh_n} \sigma_{\max}(\Ac_n)} \sup_{t\in \Ac_n} 
	\bigg( \sum_{i=1}^n \sigma\big(\tfrac{i}{n}\big) V_i K_{h_n}^*\big(\tfrac{i}{n} - t\big)\bigg) + \Oc(r_n h_n^\kappa) + o_\pr\big(\tfrac{r_n \log^2 n}{n^{3/4 h_n}}\big), \notag
\end{align}
where $(V_i)_{i\in\N}$ denotes a sequence of i.i.d. standard normal random variables. Let $B = (B(t))_{t\ge 0}$ denote a standard Brownian motion. Then, by self-similarity of $B$,
\begin{align} \label{eq:gaussian_approx2}
	\frac{1}{\sqrt{nh_n}} 
	\sum_{i=1}^n \sigma\big(\tfrac{i}{n}\big) V_i K_{h_n}^*\big(\tfrac{i}{n} - t\big)
	& \stackrel{\Dc}{=} \frac{1}{\sqrt{nh_n}} 
	\sum_{i=1}^n  \sigma\big(\tfrac{i}{n}\big) K_{h_n}^*\big(\tfrac{i}{n} - t\big) \big(B(i) - B(i-1)\big) \\
	& \stackrel{\Dc}{=}
	\sum_{i=1}^n  \sigma\big(\tfrac{i}{n}\big) K_{h_n}^*\big(\tfrac{i}{n} - t\big) \big(B\big(\tfrac{i}{nh_n}\big) - B\big(\tfrac{i-1}{nh_n}\big)\big). \notag
\end{align}
The right-hand side approximates the stochastic integral $\int \sigma(t) K^*(x-t/h_n)\diff B_x$. More specifically,
\begin{align*}
	D_n(t) & := \sum_{i=1}^n  \sigma\big(\tfrac{i}{n}\big) K_{h_n}^*\big(\tfrac{i}{n} - t\big) \big(B\big(\tfrac{i}{nh_n}\big) - B\big(\tfrac{i-1}{nh_n}\big)\big) - \int_0^{1/h_n} \sigma(t) K^*(x-t/h_n)\diff B_x \\
	& = \sum_{i=1}^n \int_{(i-1)/(nh_n)}^{i/(nh_n)} \sigma\big(\tfrac{i}{n}\big) 	K^*\big(\tfrac{i}{nh_n} - t/h_n\big)  - \sigma(t) K^*(x-t/h_n)\diff B_x,
\end{align*}
so that
\begin{align*}
	\sup_{t\in[h_n, 1-h_n]} |D_n(t)| 
	& \le \sum_{i=1}^n \frac{1}{nh_n}  \sup_{s, t\in [\frac{i-1}{nh_n}, \frac{i}{nh_n}]}|B(s) - B(t)| \\
	& \phantom{\le \sum_{i=1}^n}\times \sup_{t\in[h_n, 1-h_n]} \sup_{x\in [\frac{i-1}{nh_n}, \frac{i}{nh_n}]}|\sigma\big(\tfrac{i}{n}\big) K^*\big(\tfrac{i}{nh_n} - t/h_n\big) - \sigma(t) K^*(x-t/h_n)| 
	.
\end{align*}
The kernel $K^*$ is supported on $[-1, 1]$, so that at most $\Oc(nh_n)$ many summands are non-zero. 
By Lipschitz continuity of $\sigma$ and $K^*$, 
\begin{equation*}
	\sup_{x\in [\frac{i-1}{nh_n}, \frac{i}{nh_n}]} | \sigma\big(\tfrac{i}{n}\big) K^*\big(\tfrac{i}{nh_n} - t/h_n\big) - \sigma(t) K^*\big(x - t/h_n\big) | \le C \Big(\frac{1}{nh_n} + h_n\Big),
\end{equation*}
uniformly in $t\in[h_n, 1-h_n]$, for some constant $C \ge 0$. By Theorem 3.2.1 of \cite{khoshnevisan2002}, 
\begin{equation*}
	\sup_{|s - t| \le \frac{1}{nh_n}}|B(s) - B(t)| \le C \sqrt{\frac{\log(nh_n)}{nh_n}},
\end{equation*}
almost surely, for all large $n\in \N$. In particular, 
\begin{equation*}
	\sup_{t\in[h_n, 1-h_n]} |D_n(t)| = \Oc\Big(\big(\tfrac{\log(nh_n)}{nh_n}\big)^{1/2} \big(h_n + \tfrac{1}{nh_n}\big)\Big)
\end{equation*}
almost surely. In view of \eqref{eq:gaussian_approx1} and \eqref{eq:gaussian_approx2}, 
\begin{align*} 
	& \frac{r_n}{\sigma_{\max}(\Ac_n)} \sup_{t\in \Ac_n} \big(\hat{d}_n(t) - d(t)\big) \\
	& \stackrel{\Dc}{=} \frac{1}{\sigma_{\max}(\Ac_n)} \sup_{t\in \Ac_n} 
	\int_0^{1/h_n} \sigma(t) K^*(x-t/h_n)\diff B_x + \Oc(r_nh_n^\kappa) + o_\pr\big(\tfrac{r_n \log^2 n}{n^{3/4} h_n}\big). \notag
\end{align*}
Since $t\in[h_n, 1-h_n]$, the integral might be extended to the whole real line without changing its value. 
\hfill $\blacksquare$

\textbf{Proof of Theorem \ref{thm:localized_gumbel}.}
The theorem follows immediately from Theorem \ref{thm:gaussian_approx} and
Lemmas \ref{lem:wn_estimation} and \ref{lem:gumbel_limit}. \hfill $\blacksquare$

The conditions on $h_n, \rho_n$ and $\theta_n$ may appear cumbersome and are not required for all steps. They are formulated to cover the two distinct cases, specified by Assumption \ref{assump:regularity_sigma}. Part 1 of the assumption allows for polynomial sequences. In the level case $d = \mu - g(\mu, \cdot)$ the conditions are satisfied, if $h_n \sim n^{-1/3-\alpha}$ and $\rho_n = 1/\sqrt{nh_n}$, for some $\alpha \in (0, 1/6)$. In this case, the contribution of intervals $[a_{i,n}, b_{i,n}]$ corresponding to isolated points, is negligible, and we can assume without loss of generality, that $\liminf_{n\to\infty} \theta_n \ge c > 0$. 

In contrast, when $\sigma$ does not assume its maximum on an interval of positive length, equality in \eqref{eq:main_conv} can still hold. In this case, however, $\rho_n$ must decrease substantially slower than $h_n$, as specified in part 2 of Assumption \ref{assump:regularity_sigma}.

The following steps all require $h_n, \rho_n \to 0, nh_n\to\infty$, and $h_n \le \rho_n^{1+\eps}$, for some $\eps > 0$. The last condition implies $\log W_n / \ell_n^2 \le C$, which in turn yields $\rho_n a_n = \Oc(\rho_n \ell_n)$.
Theorem \ref{thm:gaussian_approx} requires $\ell_n r_nh_n^\kappa \to 0$ and $r_n(\log n)^2 / (n^{3/4}h_n) \to 0$. Lemma \ref{lem:wn_estimation} requires $\rho_n \ell_n^2 = \Oc(1)$, which follows from  $\rho_n |\log(h_n)| \to 0$. Lemma \ref{lem:gumbel_limit} requires $\rho_n/\theta_n \to 0$ and $\rho_n \ell_n \to 0$. The latter follows from  $\rho_n |\log(h_n)| \to 0$. Finally, the lemma requires $\rho_n \ell_n^2 = o(1)$ under Assumption \ref{assump:regularity_sigma} (2), which follows from  $\rho_n |\log(h_n)| \to 0$ in this case.

\begin{lemma} \label{lem:wn_estimation}
	Let Assumptions \ref{assump:error}, \ref{assump:sequences} and \ref{assump:sets} be satisfied. Then $\log(\widehat{W}_n^{(c)}(\Ac_n)) - \log(W_n^{(c)}(\Ac_n)) = o_\pr(1)$. In particular, $a_n(\Ac_n) - \hat{a}_n(\Ac_n) = o_\pr(\ell_n^{-1})$ and $\hat{a}_n(\Ac_n) / a_n(\Ac_n) = 1 + o_\pr(\ell_n^{-2})$.
\end{lemma}

\begin{proof}
	By solving the equation for $W_n^{(c)}(\Ac_n)$, we get the explicit representation
	\begin{equation} \label{eq:representation_Wnc}
		\log(W_n^{(c)}(\Ac_n))  = \ell_n \sqrt{2\log(W_n(\Ac_n)) + \ell_n^2} -\ell_n^2.
	\end{equation}
	Using the identity $a - b = (a^2 - b^2)/(a+b)$, we can rewrite
	\begin{equation*}
		\log(\widehat{W}_n^{(c)}(\Ac_n)) - \log(W_n^{(c)}(\Ac_n))
		= 2\ell_n \frac{\log(\widehat{W}_n(\Ac_n)) - \log(W_n(\Ac_n))}{\sqrt{2 \log(W_n(\Ac_n)) + \ell_n^2} + \sqrt{2 \log(\widehat{W}_n(\Ac_n)) + \ell_n^2}}.
	\end{equation*}
	Hence, it suffices to prove $\log(W_n(\Ac_n)) - \log(\widehat{W}_n(\Ac_n)) = o_\pr(1)$.
	By \eqref{eq:rate_sigma} and assumption, $\Delta_n = \sup_{t\in[\gamma_n, 1-\gamma_n]} |\hat{\sigma}(t) - \sigma(t)| = o_\pr(\ell_n^{-2})$, and by the triangle inequality, $|\sigma_{\max}^2 - \hat{\sigma}_{\max}^2| \le \Delta_n$. Since $\sigma$ is bounded away from zero, 
	\begin{align*}
		\sup_{t\in[\gamma_n, 1-\gamma_n]} \frac{\sigma_{\max}^2}{\sigma^2(t)} - \frac{\hat{\sigma}_{\max}^2}{\hat{\sigma}^2(t)}
		& = \sup_{t\in[\gamma_n, 1-\gamma_n]} \frac{\sigma_{\max}^2(\hat{\sigma}^2(t) - \sigma^2(t)) + (\sigma_{\max}^2 - \hat{\sigma}_{\max}^2) \sigma^2(t)}{\sigma^2(t) \hat{\sigma}^2(t)} \\
		& \le 2 \Delta_n (2 \sup_{t\in[\gamma_n, 1-\gamma_n]}\sigma(t) + \Delta_n) \sup_{t\in[\gamma_n, 1-\gamma_n]}\frac{1}{\hat{\sigma}^2(t)} \bigg(1 + \frac{\sigma_{\max}^2}{\sigma^2(t)} \bigg) \\
		& \le 2 \Delta_n (2 \sup_{t\in[\gamma_n, 1-\gamma_n]}\sigma(t) + \Delta_n) \sup_{t\in[\gamma_n, 1-\gamma_n]} \frac{1 + \sigma_{\max}^2/\sigma^2(t)}{\sigma^2(t) - \Delta_n},
	\end{align*}
	which can be bounded by $C_\sigma \Delta_n$, for some suitable constant $C_\sigma\ge 0$, since $\Delta_n = o_\pr(\ell_n^{-2})$.

	By the mean value theorem and monotony of the exponential function, 
	$|e^a - e^b| \le e^{a \vee b} | a - b| \le e^a e^{|a-b|} |a-b|$.
	If we apply this inequality to the summands in $\widehat{W}_n(\Ac_n) - W_n(\Ac_n)$, we can factor out the terms involving differences $e^{|a-b|} |a-b|$, since these differences are bounded by $C_\sigma \Delta_n \ell_n^2/2$, and recover $W_n(\Ac_n)$ from the remaining terms. In particular,
	\begin{equation*}
		|\widehat{W}_n(\Ac_n) - W_n(\Ac_n)|
		\le C_\sigma \Delta_n \frac{\ell_n^2}{2} \exp\Big(C_\sigma \Delta_n \frac{\ell_n^2}{2}\Big) W_n(\Ac_n).
	\end{equation*}
	Note that $\Delta_n \ell_n^2 = o_\pr(1)$, hence, $|\widehat{W}_n(\Ac_n) - W_n(\Ac_n)| / W_n(\Ac_n) = o_\pr(1)$.
	Recall that $\log(1 - x) = -x + \Oc(x^2) = - x (1 + \Oc(x))$, as $x\to 0$, to conclude
	\begin{align*}
		|\log(\widehat{W}_n(\Ac_n)) - \log(W_n(\Ac_n))|
		& = \bigg|\log\bigg(1 - \Big(1 - \frac{\widehat{W}_n(\Ac_n)}{W_n(\Ac_n)}\Big)\bigg)\bigg| \\
		& = \bigg| \frac{W_n(\Ac_n) - \widehat{W}_n(\Ac_n)}{W_n(\Ac_n)} \bigg| \big(1 + o_\pr(1)\big)
		= o_\pr(1).
	\end{align*}
	For the second part of the lemma, note that 
	\begin{equation*}
		a_n(\Ac_n) - \hat{a}_n(\Ac_n) = \ell_n^{-1}\big(\log(\widehat{W}_n^{(c)}(\Ac_n)) - \log(W_n^{(c)}(\Ac_n))\big) = o_\pr(\ell_n^{-1})
	\end{equation*}
	and
	\begin{equation*}
		\frac{\hat{a}_n(\Ac_n)}{a_n(\Ac_n)} 
		= \frac{\ell_n +  \log(\widehat{W}_n^{(c)}(\Ac_n)) / \ell_n}{\ell_n +  \log(W_n^{(c)}(\Ac_n)) / \ell_n }
		= 1 + \frac{\log(\widehat{W}_n^{(c)}(\Ac_n)) - \log(W_n^{(c)}(\Ac_n))}{\ell_n^2 + \log(W_n^{(c)}(\Ac_n))}
		= 1 + o_\pr(\ell_n^{-2}).
	\end{equation*}
\end{proof}

\begin{lemma} \label{lem:gumbel_limit}
	Let Assumptions \ref{assump:kern}, \ref{assump:sequences} and \ref{assump:sets} be satisfied with $\rho_n / \theta_n \to 0$ as $n\to\infty$. Finally, let $\sigma:[0, 1]\to\R$ be Lipschitz continuous and bounded away from $0$, i.\,e., $\inf_{t\in\Ac}\sigma(t) > 0$. Then, 
	\begin{equation*}
		\liminf_{n\to\infty} \pr\bigg(\frac{a_n(\Ac_n)}{\|K^*\|_2  \sigma_{\max}(\Ac_n)} \sup_{t\in \Ac_n} \int_\R \sigma(t) K^*(x-\tfrac{t}{h_n}) \diff B_x - a_n^2(\Ac_n) \le x\bigg) \ge \exp(-\exp(-x)),
	\end{equation*}
	for a standard Brownian motion $B$.	
	If additionally Assumption \ref{assump:regularity_sigma} (1) or (2) is satisfied, equality holds in the previous display.
\end{lemma}

\begin{proof}
	For the sake of readability, denote $a_n = a_n(\Ac_n), \sigma_{\max} = \sigma_{\max}(\Ac_n)$ and $W_n = W_n(\Ac_n)$. First, define 
	\begin{equation*}
		\Pc_0 = \pr\bigg(\frac{a_n}{\|K^*\|_2  \sigma_{\max}} \sup_{t\in \Ac_n} 
		\int_\R \sigma(t) K^*(x-\tfrac{t}{h_n})\diff B_x - a_n^2 \le x\bigg).
	\end{equation*}
	By substituting $t$ with $t h_n$, it holds
	\begin{align*}
		\sup_{t\in \Ac_n} 
		\int_\R \sigma(t) K^*(x-\tfrac{t}{h_n}) \diff B_x 
		= \sup_{t\in \Ac_n / h_n} \sigma(th_n) \int_\R  K^*(x-t)\diff B_x. 
	\end{align*}
	Let $X(t)$ denote the centered Gaussian process $\int_\R K^*(x-t)\diff B_x / \|K^*\|_2$. Note that $\var(X(t)) = 1$ and $\cov(X(t), X(s)) = \int K^*(x) K^*(x - (t-s))\diff x / \|K^*\|_2^2$. Since $K^*$ has support $[-1 ,1]$, $X(t)$ is independent across intervals of distance $2$ or greater. The asymptotic behavior of the supremum is determined by the points in an environment of $\{t \in \Ac: \sigma(t) = \sigma_{\max}\}$. To quantify this dependence, define $M(\delta) = \{t \in \Ac: \sigma^2(t) \ge (1- \delta) \sigma_{\max}^2\}$.
	
	By Assumption \ref{assump:sets}, the intervals $[a_{i, n}, b_{i, n}]$ have lengths greater than $\theta_n$. The intervals $ [\xi_{\nu, i, n}, \xi_{\nu+1, i, n}]$ in the set decomposition \eqref{eq:set_decomposition} have length $\rho_n + 2h_n$. Since $\rho_n = o(\theta_n)$ by assumption, the number of intervals $K_{i, n}$ tends to $\infty$ for each $i \in \{1, \dots, L\}$.	
	For sufficiently large $n$, the intervals $[a_{i, n}, b_{i, n}]$ are disjoint with distances greater than $2h_n$, hence,
	\begin{equation} \label{eq:interval_product}
		\Pc_0 = \prod_{i=1}^{L} \pr\bigg(a_n \frac{\sigma(t h_n)}{\sigma_{\max}} \sup_{t\in [a_{i, n}, b_{i, n}]/h_n} X(t)  - a_n^2 \le x\bigg).
	\end{equation}
	Recall $\xi_{\nu, i, n} = a_{i, n} + (\nu - 1) (\rho_n + 2 h_n)$, and define the big blocks $B_{i, \nu} = [\xi_{\nu, i, n}, \xi_{\nu, i, n} + \rho_n]/h_n$ and small blocks $S_{i, \nu} = (\xi_{\nu, i, n} + \rho_n , \xi_{\nu+1, i, n})/h_n$,
	for $\nu = 1, \dots, K_{i, n}$ and $i = 1, \dots, L$. Further, let $R_i$ denote the remainder, such that
	\begin{equation*}
		\Big[\frac{a_{i, n}}{h_n}, \frac{b_{i, n}}{h_n}\Big] = \bigcup_{\nu=1}^{K_{i, n}} \big(B_{i, \nu} \cup S_{i, \nu}\big) \cup R_i.
	\end{equation*}
	Then $R_i$ has length less than $\rho_n /h_n + 2$. 
	By the same arguments as in the proof of Proposition A.6 of \cite{bucher2021}, more specifically, by the arguments used to treat the right-hand side of (A.27), the small blocks and the remainder are negligible compared to the big blocks. Intuitively, the probability that $X(t)$ assumes its maximum within the (asymptotically) negligible sets $S_{i, \nu}$ or $R_i$, vanishes as $n\to\infty$.
	
	Note that $X(t)$ is independent across the big blocks since they have distances of $2$ or greater. By omitting the small blocks and the remainder, 
	\begin{equation} \label{eq:interval_product_upper}
		\Pc_0 = (1 + o(1)) \prod_{i=1}^{L} \prod_{\nu=1}^{K_{i, n}} \pr\bigg(a_n \bigg[\frac{\sigma(\xi_{\nu, i, n})}{\sigma_{\max}} \sup_{t\in B_{i, \nu}} X(t) - a_n + \Oc(\rho_n)\bigg] \le x\bigg),
	\end{equation}
	where we used Lipschitz continuity of $\sigma$ and independence of $X$ across the blocks. 
	By Theorem \ref{thm:evt}, and since $x \Psi(x) = (1+o(1)) (2\pi)^{-1/2} \exp(-x^2/2)$ as $x\to\infty$,
	\begin{align*}
		p(B_{i, \nu}) & := \pr\bigg(a_n \bigg[\frac{\sigma(\xi_{\nu, i, n})}{\sigma_{\max}} \sup_{t\in B_{i, \nu}} X(t) - a_n + \Oc(\rho_n)\bigg] > x\bigg)\\
		& = (H_2 + o(1)) \sqrt{\frac{C_0}{2\pi}} \frac{\rho_n}{h_n}  \exp\bigg( - \frac{1}{2} \frac{\sigma_{\max}^2}{\sigma^2(\xi_{\nu, i, n})} \big(a_n + x/a_n + \Oc(\rho_n)\big)^2 \bigg)
	\end{align*}
	uniformly for all $\nu \in \{1, \dots, K_{i, n}\}$. By Proposition \ref{prop:sequences}, $\rho_n a_n = \rho_n \ell_n (1 + \log (W_n^{(c)}) / \ell_n^2) = \Oc(\rho_n\ell_n) = o(1)$. Since $H_2 = \pi^{-1/2}$ and $\log(W_n^{(c)}) + \log^2(W_n^{(c)}) / (2 \ell_n^2) = \log(W_n)$,
	\begin{align*}
		p(B_{i, \nu}) & = (H_2 + o(1)) \sqrt{\frac{C_0}{2\pi}}  \frac{\rho_n}{h_n}  \exp\Big( - \frac{\sigma_{\max}^2}{\sigma^2(\xi_{\nu, i, n})} \Big[\frac{a_n^2}{2} + x + \frac{x^2}{2a_n^2} + \Oc(\rho_n a_n)\Big] \Big) \\
		& = (H_2 + o(1)) \frac{2\pi}{\Lambda_2} \sqrt{\frac{C_0}{2\pi}}  \exp\Big(\frac{\ell_n^2}{2} - \frac{\sigma_{\max}^2}{\sigma^2(\xi_{\nu, i, n})} \Big[\frac{a_n^2}{2} + x \Big] \Big) \\
		& = (H_2 + o(1)) \frac{2\pi}{\Lambda_2} \sqrt{\frac{C_0}{2\pi}}  \exp\Big(\frac{\ell_n^2}{2} - \frac{\sigma_{\max}^2}{\sigma^2(\xi_{\nu, i, n})} \Big[\frac{\ell_n^2}{2} + \log W_n^{(c)} + \frac{\log^2 W_n^{(c)}}{2 \ell_n^2} + x \Big] \Big) \\
		& = (e^{-x} + o(1)) \frac{1}{W_n} \exp\Big(- \Big[\frac{\sigma_{\max}^2}{\sigma^2(\xi_{\nu, i, n})} - 1 \Big] \Big[\frac{\ell_n^2}{2} + \log W_n + x \Big] \Big) .
	\end{align*}
	Since $\sigma^2(\xi_{\nu, i, n}) \le \sigma_{\max}^2$ and $\ell_n^2/2 +\log W_n + x \ge 0$, for sufficiently large $n$, $p(B_{i,\nu}) = \Oc(W_n^{-1})$. 
	Recall that $\log(1 - x) = -x + \Oc(x^2) = - x (1 + \Oc(x))$, as $x\to 0$. Combining \eqref{eq:interval_product} and \eqref{eq:interval_product_upper},
	\begin{align} \label{eq:log_prob}
		\log(\Pc_0) 
		& = \sum_{i=1}^{L} \sum_{\nu=1}^{K_{i, n}} \log(1 - p(B_{i, \nu})) \\
		& = - (1 + \Oc(W_n^{-1})) \sum_{i=1}^{L} \sum_{\nu=1}^{K_{i, n}} p(B_{i, \nu}) \notag  \\
		& \ge - (e^{-x} + o(1)) \frac{1}{W_n} \sum_{i=1}^{L} \sum_{\nu=1}^{K_{i, n}} \exp\Big(- \Big[\frac{\sigma_{\max}^2}{\sigma^2(\xi_{\nu, i, n})} - 1 \Big] \frac{\ell_n^2}{2} \Big) \notag \\
		& = - \exp(-x) + o(1), \notag
	\end{align}
	and $\Pc_0 \ge \exp\big(-\exp(-x)\big)(1+ o(1))$. 
	
	Consider the case that Assumption \ref{assump:regularity_sigma} (1) or (2) is satisfied. By Proposition \ref{prop:sequences} (4), 
	\begin{equation*}
		\frac{1}{W_n} \sum_{(i, \nu): \xi_{\nu, i, n} \in M(\eps_n)} \exp\Big(- \Big[\frac{\sigma_{\max}^2}{\sigma^2(\xi_{\nu, i, n})} - 1 \Big] \frac{\ell_n^2}{2}\Big) \to 1,
	\end{equation*}
	for some sequence $\eps_n$ with $\eps_n \log W_n \to 0$. Then, 
	\begin{align*} 
		& \frac{1}{W_n} \sum_{i=1}^{L} \sum_{\nu=1}^{K_{i, n}} 
		\exp\Big(- \Big[\frac{\sigma_{\max}^2}{\sigma^2(\xi_{\nu, i, n})} - 1 \Big] \frac{\ell_n^2}{2}\Big) \exp\Big(- \Big[\frac{\sigma_{\max}^2}{\sigma^2(\xi_{\nu, i, n})} - 1 \Big]  (\log W_n + x) \Big)  \\
		& \ge \frac{1}{W_n} \sum_{(i, \nu): \xi_{\nu, i, n} \in M(\eps_n)} 
		\exp\Big(- \Big[\frac{\sigma_{\max}^2}{\sigma^2(\xi_{\nu, i, n})} - 1 \Big] \frac{\ell_n^2}{2}\Big) \exp\Big(- \Big[\frac{\sigma_{\max}^2}{\sigma^2(\xi_{\nu, i, n})} - 1 \Big]  (\log W_n + x) \Big)  \\
		& \ge \frac{1}{W_n} \sum_{(i, \nu): \xi_{\nu, i, n} \in M(\eps_n)} 
		\exp\Big(- \Big[\frac{\sigma_{\max}^2}{\sigma^2(\xi_{\nu, i, n})} - 1 \Big] \frac{\ell_n^2}{2}\Big) \exp(- \eps_n (\log W_n + x) )  \\
		& = 1 + o(1),
	\end{align*}
	since $\exp(- \eps_n \log W_n ) \to 1$. Analogously to \eqref{eq:log_prob}, we obtain the upper bound $\log(\Pc_0) \le - \exp(-x) + o(1)$, and hence, equality.
\end{proof}

\begin{proposition} \label{prop:sequences}
	Let the assumptions of Theorem \ref{thm:localized_gumbel} be satisfied. 
	\begin{enumerate}
		\item For any $n\in\N$, $\rho_n W_n(\Ac_n) \le 1$.
		\item If additionally Assumption \ref{assump:regularity_sigma} (1) is satisfied, $C \le \rho_n W_n(\Ac_n) \le 1$, for some constant $C > 0, N\in\N$ and any $n \ge N$. In particular, $W_n^{(c)}(\Ac_n) \to \infty$ as $n\to\infty$.  
		\item If additionally Assumption \ref{assump:regularity_sigma} (2) is satisfied, $C \ell_n^{-2} \le \rho_n W_n(\Ac_n) \le 1$, for some constant $C > 0, N\in\N$ and any $n \ge N$. In particular, $W_n^{(c)}(\Ac_n) \to \infty$ as $n\to\infty$.  
		\item If Assumption \ref{assump:regularity_sigma} (1) or (2) is satisfied, a sequence $(\eps_n)_{n\in\N}$ exists with $\eps_n \log W_n(\Ac_n) \to 0$, such that
		\begin{equation*}
			\frac{1}{W_n(\Ac_n)} \sum_{(i, \nu) \in N(\eps_n)} \exp\Big(- \Big[\frac{\sigma_{\max}(\Ac_n)^2}{\sigma^2(\xi_{\nu, i, n})} - 1 \Big] \frac{\ell_n^2}{2}\Big) \to 1,
		\end{equation*}
		where $N(\delta) = \{(i, \nu): \sigma^2(\xi_{\nu, i, n}) \ge (1- \delta) \sigma_{\max}^2(\Ac_n)\}$.
	\end{enumerate}
\end{proposition}

\begin{proof}
	Throughout the proof, denote $\sigma_{\max} = \sigma_{\max}(\Ac_n)$ and $W_n = W_n(\Ac_n)$ for the sake of readability.
	\begin{enumerate}
		\item Each summand in the definition of $W_n(\Ac_n)$ is bounded from above by $1$. The number of summands is bounded by $\lfloor \rho_n^{-1}\rfloor$, hence $W_n(\Ac_n) \rho_n \le 1$. 
		
		\item By Assumption \ref{assump:regularity_sigma} (1), $\sigma(t) = \sigma_{\max}$ for all $t$ in some interval $I \subset [a_i, b_i]$, for $i\in\{1, \dots, L\}, a_i \neq b_i$. The distance between two points $\xi_{\nu_1, i, n}$ and $\xi_{\nu_2, i, n}$ equals $|\nu_1 - \nu_2| (\rho_n + 2 h_n)$, which behaves asymptotically as $|\nu_1 - \nu_2| \rho_n$. Hence, the number of indices $\nu \in \{1, \dots, K_{i, n}\}$, with $\sigma(\xi_{\nu, i, n}) = \sigma_{\max}$ is of order $\lambda(I) / \rho_n$, i.\,e., $|\Kc_n | \ge C \rho_n^{-1}$, for some $C > 0$ and $\Kc_n = \big\{\xi_{\nu, i, n}: \nu \in \{1, \dots, K_{i, n}\} \big\} \cap I$. 
		In particular,
		\begin{align*}
			W_n & \ge \sum_{\nu \in \Kc_n} \exp\bigg(-\frac{\ell_n^2}{2}\bigg[\frac{\sigma_{\max}^2}{\sigma^2(\xi_{\nu, i, n})} - 1\bigg]\bigg) \\
			& \ge C \rho_n^{-1},
		\end{align*}
		such that $W_n \rho_n \ge C$. 
		
		By \eqref{eq:representation_Wnc} and the identity $a - b = (a^2 - b^2)/(a+b)$,
		\begin{equation*}
			\log(W_n^{(c)}) = 2 \log(W_n) \frac{1}{\sqrt{2 \log(W_n)/\ell_n^2 + 1} + 1} \ge C_1 \log(W_n),
		\end{equation*}
		for some $C_1 \ge 0$, which diverges to $\infty$.
		
		\item In the following, let $C_\sigma$ denote the Lipschitz constant of $\sigma^2$ and $U_\eps(x) = (x-\eps, x+\eps)$ the $\eps$-neighborhood of $x\in [0, 1]$. 
		
		By Assumption \ref{assump:regularity_sigma} (2), $\ell_n^{-2} \le \theta_n$. Let $i\in \{1, \dots, L\}$ be an index, such that $t \in [a_i, b_i]$ with $\sigma(t) = \sigma_{\max}$. For any $x \in \Uc_{\ell_n^{-2}}(t)$, it holds $|\sigma^2(x) - \sigma_{\max}^2|\le C_\sigma \ell_n^{-2}$, by Lipschitz continuity of $\sigma^2$, and in particular,
		\begin{equation}\label{eq:asymptotics_wn}
			\Big| \frac{\sigma_{\max}^2}{\sigma^2(x)} - 1 \Big| 
			= \Big| \frac{\sigma_{\max}^2 - \sigma^2(x)}{\sigma^2(x)} \Big| \le \frac{C_\sigma \ell_n^{-2}}{1 - C_\sigma \ell_n^{-2}}.
		\end{equation}
		 As before, the number of indices $\nu \in \{1, \dots, K_{i, n}\}$ within the $\ell_n^{-2}$ neighborhood of $t$ is of order $\rho_n^{-1} \ell_n^{-2}$, i.\,e., $|\Kc_n | \ge C_t \rho_n^{-1} \ell_n^{-2}$, for some $C_t > 0$ and  
		\begin{equation*}
			\Kc_n = \big\{\xi_{\nu, i, n}: \nu \in \{1, \dots, K_{i, n}\} \big\} \cap U_{\ell_n^{-2}}(t).
		\end{equation*}
		Note that $\rho_n \ell_n^2 \to 0$ by assumption. Moreover, by \eqref{eq:asymptotics_wn},
		\begin{align*}
			W_n & \ge \sum_{\nu \in \Kc_n} \exp\bigg(-\frac{\ell_n^2}{2}\bigg[\frac{\sigma_{\max}^2}{\sigma^2(\xi_{\nu, i, n})} - 1\bigg]\bigg) \\
			& \ge C_t \rho_n^{-1} \ell_n^{-2} \exp\bigg(-\frac{C_\sigma}{2(1 - C_\sigma \ell_n^{-2})}\bigg),
		\end{align*}
		such that $W_n \rho_n \ge C \ell_n^{-2}$, for some suitable $C > 0$. By the same arguments as before, $\log(W_n^{(c)})  \to \infty$.
		
		\item If Assumption \ref{assump:regularity_sigma} (1) is satisfied, define $\eps_n = \delta_n / \log W_n$ for some $\delta_n \to 0$. By the arguments of part 2, $|N(\eps_n)| \ge |N(0)| \ge C_N \rho_n^{-1}$, for some constant $C_N > 0$. In particular, 
		\begin{equation*}
			\sum_{(i, \nu) \in N(\eps_n)} \exp\Big(- \Big[\frac{\sigma_{\max}^2}{\sigma^2(\xi_{\nu, i, n})} - 1 \Big] \frac{\ell_n^2}{2}\Big) \ge C_N \rho_n^{-1}.
		\end{equation*}
		Let $d_n = e_n / \ell_n^2$, for some $e_n \to \infty$ with $e_n = o(\ell_n^2)$. For $(i, \nu)\notin N(d_n)$, $\exp(- [\sigma_{\max}^2/\sigma^2(\xi_{\nu, i, n}) - 1] \ell_n^2/2) \le \exp(-e_n/2) = o(1)$. Further, $|N^c(d_n)| \le \rho_n^{-1}$, such that
		\begin{equation*}
			 \sum_{(i, \nu) \notin N(d_n)} \exp\Big(- \Big[\frac{\sigma_{\max}^2}{\sigma^2(\xi_{\nu, i, n})} - 1 \Big] \frac{\ell_n^2}{2}\Big) = o(\rho_n^{-1}),
		\end{equation*}
		By the same arguments as in the proof of Proposition \ref{prop:extremal_set}, $|N(c_n) \setminus N(\eps_n)| = o(\rho_n^{-1})$. Therefore, the sum over all indices $(i, \nu)\notin N(\eps_n)$ is of order $o(\rho_n^{-1})$, hence, negligible compared to the sum over $N(\eps_n)$, and the statement follows.
		
		 If Assumption \ref{assump:regularity_sigma} (2) is satisfied, define $\eps_n = (4 + \eps) \log \ell_n / \ell_n^2$, for some $\eps > 0$. By part 3, 
		 \begin{equation*}
		 	\eps_n \log W_n \le \delta_n \log(\rho_n^{-1}) = (2 + \eps/2) \frac{\log(\ell_n) |\log(\rho_n)|}{\log(\Lambda_2 / (2\pi)) + \log \rho_n - \log h_n},
		 \end{equation*}
		 which converges to $0$, since $\log|\log h_n| \cdot |\log \rho_n|  / |\log h_n| \to 0$ by assumption.
		 Moreover, for $(i, \nu)\notin N(\eps_n)$,  by part 3,
		 \begin{equation*}
		 	\exp\Big(- \Big[\frac{\sigma_{\max}^2}{\sigma^2(\xi_{\nu, i, n})} - 1 \Big] \frac{\ell_n^2}{2} - \log(W_n \rho_n)\Big) 
		 	\le \exp\Big(- \frac{\eps}{2} \log \ell_n - \log(C)\Big) = o(1) 
		 \end{equation*}
		 Since $|N^c(\eps_n)| \le \rho_n^{-1}$, $W_n^{-1} \sum_{(i, \nu) \notin N(\eps_n)} \exp(- [\sigma_{\max}^2/\sigma^2(\xi_{\nu, i, n}) - 1] \ell_n^2/2) = o(1)$.
	\end{enumerate}
	
\end{proof}

\textbf{Proof of Theorem \ref{thm:evt}.}
Let $T > 0$ be arbitrary. We will first verify the conditions of Theorem 2.2 of \cite{debicki2017} in order to derive
\begin{equation} \label{eq:conv_single_sum}
	\lim_{n\to\infty} \sup_{k=1}^{K_n} \bigg| \frac{\pr(\sup_{t\in[0, 1]} X(q(g_{n,k}) T t) > g_{n,k})}{\Psi(g_{n,k})} - H_2([0, T]) \bigg| = 0,
\end{equation}
where $q(g_{n,k}) = C_0^{-1/2} g_{n,k}^{-1}$ and $H_2([0, T]) = \ex[\exp(\sup_{t\in[0, T]} \sqrt{2}Rt - t^2)]$ is the general Pickand constant, for $R\sim\Nc(0, 1)$. Since $\Gamma$ corresponds to the supremum, (F1) and (F2) are trivially satisfied. Note that $g_{n,k} > \ell_n \to \infty$, so that (C0) is satisfied. Moreover $\var(X(t)) = \var(X(0)) = 1$ by stationarity of $X$, which implies (D1) with $h \equiv 0$.
	
To verify (D2) and (D3), first define $r(t) = \cov\big(X(t), X(0)\big)$ with derivatives $r'(t) = - \|K\|_2^{-2} \int_\R K(x) K'(x-t)\diff x$ and $r''(t) = \|K\|_2^{-2} \int_\R K(x) K''(x-t)\diff x$. By a Taylor expansion, $r(t) = 1 - C_0 t^2 + o(t^2)$, as $t \to 0$, or equivalently, $r(t) = 1 - C_0 t^2 (1 + \delta(t))$, for some function $\delta: \R\to\R$ with $\delta(t) = o(1)$, as $t\to 0$.
For the variance of differences,
\begin{align*}
	\var\big( X(q(g_{n,k}) T t) - X(q(g_{n,k}) T s) \big)
	& = 2 \big(1 - \cov[X(q(g_{n,k}) T t), X(q(g_{n,k}) T s)]\big) \\
	& = 2 \big(1 - r[q(g_{n,k})T(t-s)]\big)
\end{align*}
Moreover, $\delta(q(g_{n,k})T(t-s)) = o(1)$, uniformly in $k$ as $n\to\infty$, since $\ell_n < \inf_k g_{n,k}$, and hence,
\begin{align*}
	g_{n,k}^2 \var\big( X(q(g_{n,k})Tt) - X(q(g_{n,k})Ts) \big)
	& = 2 g_{n,k}^2 C_0 q^2(g_{n,k})T^2(t-s)^2 (1 + \delta(q(g_{n,k})T(t-s))) \\
	& = 2 T^2(t-s)^2 (1 + \delta(q(g_{n,k})T(t-s))) \\
	& = 2 T^2(t-s)^2 + o(1),
\end{align*}
uniformly for $s, t\in[0, 1]$ and $k\in \{1, \dots, K_n\}$, implying (D2), where the Gaussian process $\eta$ can be chosen as $\eta(t) = R T t$, for a standard normal random variable $R$. Finally, note that $g_{n,k}^2 \var\big( X(q(g_{n,k})Tt) - X(q(g_{n,k})Ts) \big) = 2 T^2 (t-s)^2 (1 + o(1)) \le C(t-s)^2$, uniformly in $k$, for some $C\ge 2$, implying (D3). Hence the conditions of Theorem 2.2 of \cite{debicki2017} are satisfied and \eqref{eq:conv_single_sum} follows.

With this uniform bound, the theorem's statement follows along the lines of Pickand's double sum method \citep[see Theorem D.2 in][]{piterbarg1996}. More specifically, define the blocks $B_\nu = [0, \rho_n/h_n] \cap [(\nu-1) T q(g_{n,k}), \nu T q(g_{n,k})]$, for $\nu\in\N$, and $b_n = \lfloor (\rho_n/h_n) T^{-1} q^{-1}(g_{n, k})\rfloor$. By the union bound and \eqref{eq:conv_single_sum},
\begin{align} \notag
	\pr\Big(\sup_{t\in[0, \rho_n/h_n]} X(t) > g_{n,k}\Big) 
	& \le \sum_{\nu = 1}^{b_n} \pr\Big(\sup_{t\in B_\nu} X(t) > g_{n,k}\Big) \\
	& = b_n \Psi(g_{n,k}) (H_2([0, T]) + o(1)) \label{eq:upper_bound_single_sum} \\
	& = C_0^{1/2} \frac{\rho_n}{h_n} g_{n, k} \Psi(g_{n,k}) T^{-1} (H_2([0, T]) + o(1)), \notag
\end{align}
uniformly in $k$, as $n\to\infty$. Contrarily, by the Bonferroni inequality,
\begin{align} \label{eq:lower_bound_double_sum}
	\pr\Big(\sup_{t\in[0, \rho_n/h_n]} X(t) > g_{n,k}\Big) 
	& \ge \sum_{\nu = 1}^{b_n} \pr\Big(\sup_{t\in B_\nu} X(t) > g_{n,k}\Big)  \\
	& -\sum_{1 \le \nu_1 < \nu_2 \le b_n} \pr\Big(\sup_{t\in B_{\nu_1}} X(t) > g_{n,k}, \sup_{t\in B_{\nu_2}} X(t) > g_{n,k}\Big). \notag
\end{align}
By stationarity of $X$, we may rewrite the double sum on the right-hand side as
\begin{equation} \label{eq:double_sum_split}
	\sum_{\nu=1}^{b_n} (b_n - \nu) p_\nu = (b_n - 1) p_1 + \sum_{\nu=2}^{u_1} (b_n - \nu) p_\nu + \sum_{\nu=u_1 + 1}^{u_2} (b_n - \nu) p_\nu + \sum_{\nu=u_2+1}^{b_n} (b_n - \nu) p_\nu.
\end{equation}
for $p_\nu = \pr\big(\sup_{t\in B_0} X(t) > g_{n,k}, \sup_{t\in B_\nu} X(t) > g_{n,k}\big)$, $u_1 = \lfloor (\eps/8) C_0^{1/2} g_{n,k} /T \rfloor + 1$ and $u_2 = \lfloor 2 C_0^{1/2} g_{n,k} / T \rfloor + 1$, where $\eps \in (0, \min\{1/2, T\})$ is chosen such that $\eta:= \inf_{s \ge \eps / 8} 1 - r(s) > 0$. 
Similar to the proof of Theorem D.2 by \cite{piterbarg1996}, each sum on the right-hand side must be treated differently. First, note that the distance between the sets $B_0$ and $B_\nu$ equals $(\nu-1) T q(g_{n, k}) = (\nu-1) T C_0^{-1/2} g_{n,k}^{-1}$, which is larger than $2$ whenever $\nu \ge u_2 + 1$. Since $r(t)=0$ for $t > 2$, the two blocks are independent and
\begin{multline*}
	\sum_{\nu=u_2+1}^{b_n} (b_n - \nu) p_\nu
	= \sum_{\nu=u_2+1}^{b_n} (b_n - \nu)  \pr\big(\sup_{t\in B_0} X(t) > g_{n,k}\big)^2 \\
	\le b_n^2 \Psi(g_{n,k})^2 (H_2([0, T]) + o(1))^2 
	\le \bigg(\frac{H_2([0, T])}{T} + o\Big(\frac{1}{T}\Big)\bigg)^2 C_0 \bigg(\frac{\rho_n}{h_n} g_{n, k} \Psi(g_{n,k})\bigg)^2,
\end{multline*}
by \eqref{eq:conv_single_sum}. Let $f(x) \sim g(x)$ denote asymptotic equality, i.\,e., $f(x)/g(x)\to 1$ as $x\to\infty$. Then $\Psi(x) \sim (2\pi)^{-1/2} x^{-1}\exp(-x^2/2)$, such that
\begin{multline*}
	\frac{\rho_n}{h_n} g_{n, k} \Psi(g_{n,k})
	\sim \frac{1}{\sqrt{2\pi}} \frac{\rho_n}{h_n} \exp\bigg(-\frac{g_{n,k}^2}{2}\bigg) \\
	= \frac{\sqrt{2\pi}}{\Lambda_2} \exp\bigg(- \frac{(c_{k,n}^2 - 1) \ell_n^2}{2}\bigg)
	\le \frac{\sqrt{2\pi}}{\Lambda_2} \exp\bigg(- \frac{(\inf_{k=1}^{K_n} c_{k,n}^2 - 1) \ell_n^2}{2}\bigg)
\end{multline*}
which vanishes as $n \to \infty$, since $(\inf_k c_{k,n}^2 - 1) \ell_n^2 \to\infty$, by assumption. Therefore, $\sum_{\nu=u_2+1}^{b_n} (b_n - \nu) p_\nu = o\big((\rho_n/h_n) g_{n, k} \Psi(g_{n,k})\big)$. 

For $\nu \in \{ u_1 + 1, \dots, u_2\}$, the distance between $B_0$ and $B_\nu$ is larger than $(\nu-1) T C_0^{-1/2} g_{n,k}^{-1} \ge \eps / 8$. By the same arguments leading to (D.12) in the proof of Theorem D.2 by \cite{piterbarg1996}, it holds 
\begin{equation*}
	\sum_{\nu=u_1 + 1}^{u_2} (b_n - \nu) p_\nu 
	\le \frac{4 C_0^{1/2}}{T}g_{n,k} b_n \Psi\Big(\frac{g_{n,k} - a/2}{1 - \eta/4}\Big)
	\le \frac{4 C_0}{T^2}g_{n,k}^2 \frac{\rho_n}{h_n}  \Psi\Big(\frac{g_{n,k} - a/2}{1 - \eta/4}\Big),
\end{equation*}
for some $a \in \R$, uniformly in $k$. Note that uniformity of the bound on $p_\nu$ in $\nu$ follows by the arguments leading to (D.12), whereas uniformity in $k$ follows from the application of Borell's theorem \citep[cf. Theorem D.1 in][]{piterbarg1996}. By definition of $\eta$, $\eta_0 := (1-\eta/4)^{-2} - 1 > 0$. In particular,
\begin{align*}
	g_{n,k}  \Psi\Big(\frac{g_{n, k} - a/2}{1 - \eta/4}\Big) / \Psi(g_{n, k})
	& \sim (1 - \eta/4) g_{n, k} \exp\bigg(-\frac{(g_{n, k} - a/2)^2}{2(1 - \eta/4)^2} + \frac{g_{n,k}^2}{2}\bigg) \\
	& \le C_1 \exp\bigg(-\frac{\eta_0}{2}g_{n, k}^2  + \frac{g_{n,k}  a}{2(1 - \eta/4)^2} + 2 \log(g_{n, k}^2)\bigg),
\end{align*}
for some constant $C_1 \ge 0$. The exponent is dominated by the term $\eta_0 g_{n, k}^2 / 2$, so that the right-hand side converges to $0$ and $\sum_{\nu=u_1 + 1}^{u_2} (b_n - \nu) p_\nu  = o\big((\rho_n/h_n) g_{n, k} \Psi(g_{n,k})\big)$.

For $n$ sufficiently large,
\begin{equation*}
	g_{n, k} 
	> \frac{C_0^{1/2}}{4} g_{n, k} + \frac{4 T}{\eps} 
	= 2 \Big(\frac{\eps}{8} T^{-1} C_0^{1/2} g_{n, k} + 2 \Big) \frac{T}{\eps}
	> 2 (u_1 + 1) \frac{T}{\eps}
\end{equation*}
since $C_0 < 16$ by assumption, and in particular $g_{n, k} > 2 (\nu + 1) T / \eps$ for any $\nu \in \{2, \dots, u_1\}$. Hence, we can apply Lemma D.2 of \cite{piterbarg1996}, such that
\begin{align*}
	\frac{\sum_{\nu=2}^{u_1} (b_n - \nu) p_\nu}{(\rho_n/h_n) g_{n, k} \Psi(g_{n,k})}
	& \le \frac{\sum_{\nu=2}^{u_1} (b_n - \nu) h \Psi(g_{n, k}) \exp(-1/8 (\nu - 1)^2 C_0^{-1}T^2)}{(\rho_n/h_n) g_{n, k} \Psi(g_{n,k})} \\
	& \le h C_0^{1/2} T^{-1} \sum_{\nu=2}^{u_1} \exp(-1/8 (\nu - 1)^2 C_0^{-1}T^2) \\
	& \le C_2 T^{-1} \exp(-T^2/(8 C_0))
\end{align*}
uniformly in $k$, for some constant $C_2 \ge 0$. 
Finally, consider the single term $(b_n - 1) p_1$. By the same arguments leading to (D.14) in \cite{piterbarg1996}, it holds 
\begin{equation*}
	\limsup_{n\to\infty} \frac{(b_n-1) p_1}{b_n \Psi(g_{n, k})} \le h \exp( - T / 8 ) + H_2([0, 1])\sqrt{T}.
\end{equation*}
Combining \eqref{eq:upper_bound_single_sum}, \eqref{eq:lower_bound_double_sum} and the bounds on the different terms of \eqref{eq:double_sum_split} yields
\begin{align*}
	\frac{H_2([0, T])}{T}
	& \ge \lim_{n\to\infty} \frac{\pr(\sup_{t\in[0, \rho_n/h_n]} X(t) > g_{n,k})}{C_0^{1/2} \frac{\rho_n}{h_n} g_{n, k} \Psi(g_{n,k})} \\
	& \ge \frac{H_2([0, T])}{T} - \frac{C_2}{T} \exp\bigg(- \frac{T^2}{8C_0}\bigg) - \frac{h}{T} \exp( - T / 8 ) - \frac{H_2([0, 1])}{\sqrt{T}}.
\end{align*}
Since $T > 0$ was arbitrary, letting $T \to \infty$ finishes the proof. \hfill $\blacksquare$

%% file: app_empirical_results.tex
\section{Additional Empirical Results} \label{app:appendix_empirical_results}

In addition to the results from Section \ref{sec:sim_study}, the following tables contain empirical rejection rates and estimated onset locations, averaged over all time series with rejected null hypothesis. All tests use a Gaussian calibration based on $1000$ simulated Gaussian approximations and a level $\alpha = 5\%$.

Complementing the examples from Section \ref{sec:sim_study}, we used
\begin{equation*}
	\mu_5(t) = 1 - S(t/a), \qquad 
	\mu_6(t) = 1 - S(t/b) + \gamma S\bigg(\frac{t-c}{1-c}\bigg),
\end{equation*}
for $t\in [0, 1]$, with $a=0.35$, $b=0.3$, $c=0.55$, and $\gamma=0.2$. These functions are displayed jointly with their derivatives in Figure \ref{fig:mu_appendix}.
\begin{figure}[bh]
	\centering
	\includegraphics[width=\textwidth]{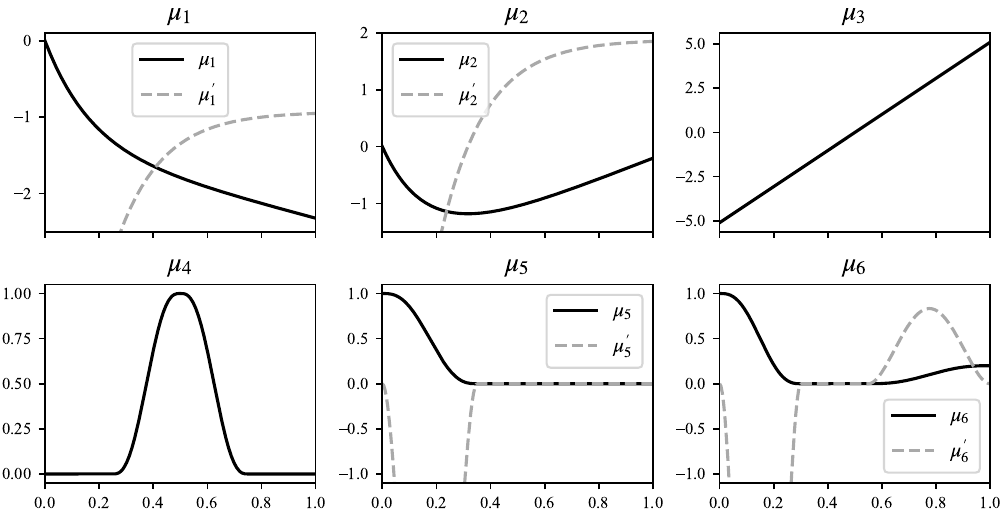}
	\caption{Various mean functions, used to generate time series.}
	\label{fig:mu_appendix}
\end{figure}
In addition to $\sigma_2(x) = 1 / 2 - \cos(2 \pi x) / 4$, we considered 
\begin{equation*}
	\sigma_0(x) = \tfrac{1}{2}, \qquad \sigma_1(x) = \tfrac{1}{4} + \tfrac{x}{2}, \qquad \sigma_3(x) = \tfrac{1}{4} + \tfrac{1}{2} \id(x \ge 1/2).
\end{equation*}
Finally, as error processes, we considered i.i.d. random variables, moving average and autoregressive processes, i.\,e., 
\begin{equation*}
	(\mathrm{iid})~ \eps_i = \eta_i, \qquad (\mathrm{ma})~ \eps_i = \tfrac{2}{\sqrt{5}}(\eta_i + \tfrac{1}{2} \eta_{i-1}), \qquad (\mathrm{ar})~ \eps_i = \tfrac{\sqrt{3}}{2}(\eta_i + \tfrac{1}{2} \eps_{i-1}),
\end{equation*}
for $(\eta_i)_{i\in\Z}$ with $\eta_i\sim\Nc(0, 1)$ i.i.d.

\begin{table}
	\caption{Empirical rejection rates and average estimated onset locations for $\delta = 0$ with $\sigma = \sigma_2$ and locally stationary $\eps$. }
	\begin{tabular}{l|rrrrr|rrrrrr}
		\toprule
		& \multicolumn{5}{c|}{Empirical Rejection Rate} & \multicolumn{6}{c}{Estimated Onset Location} \\
		$n \quad \setminus \quad \Delta$ & $0.8$ & $0.9$ & $1.0$ & $1.1$ & $1.2$ & $0.8$ & $0.9$ & $1.0$ & $1.1$ & $1.2$ & $\hat{\tau}_*$ \\
		\midrule
		\multicolumn{12}{l}{\textit{Panel A: $\mu = \mu_1$}}\\	
		200  & 0.404 & 0.445 & 0.486 & 0.513 & 0.547 & 0.531 & 0.503 & 0.481 & 0.461 & 0.433 & 0.655 \\
		500  & 0.370 & 0.401 & 0.432 & 0.460 & 0.478 & 0.504 & 0.478 & 0.452 & 0.428 & 0.409 & 0.643 \\
		1000 & 0.379 & 0.406 & 0.432 & 0.469 & 0.495 & 0.478 & 0.452 & 0.428 & 0.398 & 0.377 & 0.657 \\
		\midrule
		\multicolumn{12}{l}{\textit{Panel B: $\mu = \mu_2$}}\\		
		200  & 0.000 & 0.000 & 0.000 & 0.000 & 0.000 & -- & -- & -- & -- & -- & 0.317 \\
		500  & 0.000 & 0.000 & 0.000 & 0.000 & 0.000 & -- & -- & -- & -- & -- & 0.317 \\
		1000 & 0.000 & 0.000 & 0.000 & 0.000 & 0.000 & -- & -- & -- & -- & -- & 0.317 \\
		\midrule
		\multicolumn{12}{l}{\textit{Panel C: $\mu = \mu_3$}}\\	
		200  & 0.000 & 0.000 & 0.000 & 0.000 & 0.000 & -- & -- & -- & -- & -- & 0.500  \\
		500  & 0.000 & 0.000 & 0.000 & 0.000 & 0.000 & -- & -- & -- & -- & -- & 0.500  \\
		1000 & 0.000 & 0.000 & 0.000 & 0.000 & 0.000 & -- & -- & -- & -- & -- & 0.500  \\
		\midrule
		\midrule
		$n \quad \setminus \quad \Delta$ & $0.1$ & $0.2$ & $0.3$ & $0.4$ & $0.5$ & $0.1$ & $0.2$ & $0.3$ & $0.4$ & $0.5$ & \\
		\midrule
		\multicolumn{12}{l}{\textit{Panel D: $\mu = \mu_4$}}\\
		200  & 0.000 & 0.005 & 0.418 & 0.641 & 0.762 & -- & 0.112 & 0.028 & 0.003 & 0.000 & 0.534 \\
		500  & 0.000 & 0.306 & 0.596 & 0.735 & 0.828 & -- & 0.171 & 0.016 & 0.002 & 0.000 & 0.525 \\
		1000 & 0.000 & 0.524 & 0.733 & 0.841 & 0.934 & -- & 0.162 & 0.011 & 0.001 & 0.000 & 0.518 \\
		\midrule
		\multicolumn{12}{l}{\textit{Panel E: $\mu = \mu_5$}}\\
		200  & 0.000 & 0.002 & 0.141 & 0.247 & 0.331 & -- & 0.480 & 0.407 & 0.329 & 0.283 & 0.568 \\
		500  & 0.000 & 0.034 & 0.124 & 0.227 & 0.300 & -- & 0.459 & 0.346 & 0.266 & 0.236 & 0.556 \\
		1000 & 0.000 & 0.023 & 0.109 & 0.216 & 0.281 & -- & 0.459 & 0.350 & 0.275 & 0.238 & 0.570 \\
		\midrule
		\multicolumn{12}{l}{\textit{Panel F: $\mu = \mu_6$}}\\
		200  & 0.000 & 0.002 & 0.129 & 0.207 & 0.283 & -- & 0.520 & 0.388 & 0.334 & 0.288 & 0.545 \\
		500  & 0.000 & 0.019 & 0.095 & 0.175 & 0.237 & -- & 0.413 & 0.369 & 0.310 & 0.259 & 0.534 \\
		1000 & 0.000 & 0.016 & 0.086 & 0.171 & 0.249 & -- & 0.406 & 0.345 & 0.260 & 0.220 & 0.555 \\
		\bottomrule
	\end{tabular}
\end{table}

\begin{table}
	\caption{Empirical rejection rates and average estimated onset locations for $\delta = 0.1$ with $\sigma = \sigma_2$ and locally stationary $\eps$.}
	\begin{tabular}{l|rrrrr|rrrrrr}
		\toprule
		& \multicolumn{5}{c|}{Empirical Rejection Rate} & \multicolumn{6}{c}{Estimated Onset Location} \\
		$n \quad \setminus \quad \Delta$ & $0.8$ & $0.9$ & $1.0$ & $1.1$ & $1.2$ & $0.8$ & $0.9$ & $1.0$ & $1.1$ & $1.2$ & $\hat{\tau}_*$ \\\midrule
		\multicolumn{12}{l}{\textit{Panel A: $\mu = \mu_1$}}\\	
		200  & 0.153 & 0.192 & 0.225 & 0.257 & 0.287 & 0.636 & 0.598 & 0.580 & 0.565 & 0.548 & 0.559 \\
		500  & 0.096 & 0.141 & 0.180 & 0.214 & 0.240 & 0.628 & 0.610 & 0.591 & 0.563 & 0.544 & 0.528 \\
		1000 & 0.073 & 0.107 & 0.145 & 0.183 & 0.218 & 0.656 & 0.631 & 0.606 & 0.583 & 0.562 & 0.549 \\
		\midrule
		\multicolumn{12}{l}{\textit{Panel B: $\mu = \mu_2$}}\\	
		200  & 0.005 & 0.013 & 0.020 & 0.038 & 0.058 & 0.286 & 0.304 & 0.297 & 0.298 & 0.294 & 0.386 \\
		500  & 0.009 & 0.017 & 0.020 & 0.034 & 0.051 & 0.277 & 0.271 & 0.268 & 0.267 & 0.270 & 0.339 \\
		1000 & 0.003 & 0.008 & 0.014 & 0.020 & 0.032 & 0.268 & 0.275 & 0.269 & 0.265 & 0.265 & 0.294 \\
		\midrule
		\multicolumn{12}{l}{\textit{Panel C: $\mu = \mu_3$}}\\	
		200  & 0.000 & 0.000 & 0.000 & 0.001 & 0.022 & -- & -- & -- & 0.432 & 0.435 & 0.452 \\
		500  & 0.000 & 0.000 & 0.005 & 0.016 & 0.148 & -- & -- & 0.441 & 0.439 & 0.438 & 0.451 \\
		1000 & 0.000 & 0.001 & 0.002 & 0.040 & 0.298 & -- & 0.435 & 0.433 & 0.441 & 0.439 & 0.450 \\
		\midrule
		\midrule
		$n \quad \setminus \quad \Delta$ & $0.1$ & $0.2$ & $0.3$ & $0.4$ & $0.5$ & $0.1$ & $0.2$ & $0.3$ & $0.4$ & $0.5$ & \\
		\midrule
		\multicolumn{12}{l}{\textit{Panel D: $\mu = \mu_4$}}\\
		200  & 0.000 & 0.000 & 0.013 & 0.157 & 0.464 & -- & -- & 0.034 & 0.003 & 0.000 & 0.475 \\
		500  & 0.000 & 0.001 & 0.043 & 0.348 & 0.600 & -- & 0.747 & 0.045 & 0.002 & 0.000 & 0.439 \\
		1000 & 0.000 & 0.000 & 0.084 & 0.472 & 0.687 & -- & -- & 0.015 & 0.001 & 0.000 & 0.435 \\
		\midrule
		\multicolumn{12}{l}{\textit{Panel E: $\mu = \mu_5$}}\\
		200  & 0.000 & 0.000 & 0.002 & 0.027 & 0.078 & -- & -- & 0.430 & 0.493 & 0.505 & 0.467 \\
		500  & 0.000 & 0.000 & 0.010 & 0.036 & 0.062 & -- & -- & 0.533 & 0.529 & 0.514 & 0.431 \\
		1000 & 0.000 & 0.000 & 0.010 & 0.036 & 0.066 & -- & -- & 0.575 & 0.518 & 0.499 & 0.454 \\
		\midrule
		\multicolumn{12}{l}{\textit{Panel F: $\mu = \mu_6$}}\\
		200  & 0.000 & 0.000 & 0.001 & 0.020 & 0.046 & -- & -- & 0.352 & 0.468 & 0.478 & 0.478 \\
		500  & 0.000 & 0.000 & 0.002 & 0.019 & 0.047 & -- & -- & 0.488 & 0.426 & 0.435 & 0.421 \\
		1000 & 0.000 & 0.000 & 0.004 & 0.011 & 0.038 & -- & -- & 0.539 & 0.440 & 0.430 & 0.461 \\
		\bottomrule
	\end{tabular}
\end{table}

\begin{table}
	\caption{Empirical rejection rates and average estimated onset locations for $\delta = 1$ with $\sigma = \sigma_2$ and locally stationary $\eps$. }
	\begin{tabular}{l|rrrrr|rrrrrr}
		\toprule
		& \multicolumn{5}{c|}{Empirical Rejection Rate} & \multicolumn{6}{c}{Estimated Onset Location} \\
		$n \quad \setminus \quad \Delta$ & $0.8$ & $0.9$ & $1.0$ & $1.1$ & $1.2$ & $0.8$ & $0.9$ & $1.0$ & $1.1$ & $1.2$ & $\hat{\tau}_*$ \\\midrule
		\multicolumn{12}{l}{\textit{Panel A: $\mu = \mu_1$}}\\	
		200  & 0.001 & 0.004 & 0.010 & 0.018 & 0.033 & 0.000 & 0.000 & 0.000 & 0.000 & 0.000 & 0.000 \\
		500  & 0.015 & 0.037 & 0.050 & 0.072 & 0.098 & 0.000 & 0.000 & 0.000 & 0.000 & 0.000 & 0.000 \\
		1000 & 0.040 & 0.056 & 0.072 & 0.094 & 0.110 & 0.000 & 0.000 & 0.000 & 0.000 & 0.000 & 0.000 \\\midrule
		\multicolumn{12}{l}{\textit{Panel B: $\mu = \mu_2$}}\\	
		200  & 0.005 & 0.010 & 0.015 & 0.024 & 0.037 & 0.000 & 0.000 & 0.000 & 0.000 & 0.000 & 0.000 \\
		500  & 0.028 & 0.049 & 0.071 & 0.094 & 0.116 & 0.000 & 0.000 & 0.000 & 0.000 & 0.000 & 0.000 \\
		1000 & 0.050 & 0.068 & 0.089 & 0.106 & 0.120 & 0.000 & 0.000 & 0.000 & 0.000 & 0.000 & 0.000 \\
		\midrule
		\multicolumn{12}{l}{\textit{Panel C: $\mu = \mu_3$}}\\	
		200  & 0.000 & 0.000 & 0.001 & 0.023 & 0.150 & -- & -- & 0.000 & 0.000 & 0.000 & 0.000 \\
		500  & 0.000 & 0.000 & 0.014 & 0.143 & 0.397 & -- & -- & 0.000 & 0.000 & 0.000 & 0.000 \\
		1000 & 0.000 & 0.001 & 0.046 & 0.258 & 0.538 & -- & 0.000 & 0.000 & 0.000 & 0.000 & 0.000 \\
		\midrule
		\midrule
		$n \quad \setminus \quad \Delta$ & $0.1$ & $0.2$ & $0.3$ & $0.4$ & $0.5$ & $0.1$ & $0.2$ & $0.3$ & $0.4$ & $0.5$ & \\
		\midrule
		\multicolumn{12}{l}{\textit{Panel D: $\mu = \mu_4$}}\\
		200  & 0.000 & 0.000 & 0.000 & 0.000 & 0.000 & -- & -- & -- & -- & -- & 0.000 \\
		500  & 0.000 & 0.000 & 0.000 & 0.000 & 0.000 & -- & -- & -- & -- & -- & 0.000 \\
		1000 & 0.000 & 0.000 & 0.000 & 0.000 & 0.000 & -- & -- & -- & -- & -- & 0.000 \\
		\midrule
		\multicolumn{12}{l}{\textit{Panel E: $\mu = \mu_5$}}\\
		200  & 0.000 & 0.000 & 0.000 & 0.000 & 0.000 & -- & -- & -- & -- & -- & 0.000 \\
		500  & 0.000 & 0.000 & 0.000 & 0.000 & 0.000 & -- & -- & -- & -- & -- & 0.000 \\
		1000 & 0.000 & 0.000 & 0.000 & 0.000 & 0.000 & -- & -- & -- & -- & -- & 0.000 \\
		\midrule
		\multicolumn{12}{l}{\textit{Panel F: $\mu = \mu_6$}}\\
		200  & 0.000 & 0.000 & 0.000 & 0.000 & 0.000 & -- & -- & -- & -- & -- & 0.000 \\
		500  & 0.000 & 0.000 & 0.000 & 0.000 & 0.000 & -- & -- & -- & -- & -- & 0.000 \\
		1000 & 0.000 & 0.000 & 0.000 & 0.000 & 0.000 & -- & -- & -- & -- & -- & 0.000 \\
		\bottomrule
	\end{tabular}
\end{table}

\begin{table}
	\caption{Empirical rejection rates and average estimated onset locations for different choices of $\sigma$ with $\delta= 0.2$, $n = 1000$ and locally stationary $\eps$.}
	\begin{tabular}{l|rrrrr|rrrrrr}
		\toprule
		& \multicolumn{5}{c|}{Empirical Rejection Rate} & \multicolumn{6}{c}{Estimated Onset Location} \\
		$\sigma \quad \setminus \quad \Delta$ & $0.8$ & $0.9$ & $1.0$ & $1.1$ & $1.2$ & $0.8$ & $0.9$ & $1.0$ & $1.1$ & $1.2$ & $\hat{\tau}_*$ \\
		\midrule
		\multicolumn{12}{l}{\textit{Panel A: $\mu = \mu_1$}}\\	
		$\sigma_0$ & 0.003 & 0.006 & 0.007 & 0.016 & 0.024 & 0.695 & 0.633 & 0.635 & 0.620 & 0.599 & 0.335 \\
		$\sigma_1$ & 0.002 & 0.004 & 0.006 & 0.006 & 0.010 & 0.696 & 0.699 & 0.702 & 0.693 & 0.666 & 0.216 \\
		$\sigma_2$ & 0.010 & 0.013 & 0.022 & 0.041 & 0.070 & 0.712 & 0.672 & 0.681 & 0.663 & 0.638 & 0.417 \\
		$\sigma_3$ & 0.001 & 0.001 & 0.002 & 0.002 & 0.003 & 0.556 & 0.545 & 0.599 & 0.591 & 0.530 & 0.221 \\
		\midrule
		\multicolumn{12}{l}{\textit{Panel B: $\mu = \mu_2$}}\\		
		$\sigma_0$ & 0.002 & 0.006 & 0.016 & 0.033 & 0.047 & 0.262 & 0.238 & 0.239 & 0.235 & 0.233 & 0.206 \\
		$\sigma_1$ & 0.001 & 0.002 & 0.007 & 0.010 & 0.019 & 0.259 & 0.241 & 0.254 & 0.244 & 0.258 & 0.157 \\
		$\sigma_2$ & 0.003 & 0.014 & 0.028 & 0.044 & 0.058 & 0.251 & 0.251 & 0.242 & 0.248 & 0.242 & 0.209 \\
		$\sigma_3$ & 0.000 & 0.001 & 0.002 & 0.007 & 0.009 & -- & 0.250 & 0.261 & 0.248 & 0.243 & 0.181 \\
		\midrule
		\multicolumn{12}{l}{\textit{Panel C: $\mu = \mu_3$}}\\	
		$\sigma_0$ & 0.000 & 0.000 & 0.001 & 0.129 & 0.516 & -- & -- & 0.399 & 0.388 & 0.381 & 0.400 \\
		$\sigma_1$ & 0.000 & 0.000 & 0.000 & 0.030 & 0.177 & -- & -- & -- & 0.388 & 0.381 & 0.399 \\
		$\sigma_2$ & 0.000 & 0.000 & 0.000 & 0.080 & 0.376 & -- & -- & -- & 0.386 & 0.379 & 0.398 \\
		$\sigma_3$ & 0.000 & 0.000 & 0.000 & 0.008 & 0.080 & -- & -- & -- & 0.389 & 0.380 & 0.396 \\
		\midrule
		\midrule
		$n \quad \setminus \quad \Delta$ & $0.1$ & $0.2$ & $0.3$ & $0.4$ & $0.5$ & $0.1$ & $0.2$ & $0.3$ & $0.4$ & $0.5$ & \\
		\midrule
		\multicolumn{12}{l}{\textit{Panel D: $\mu = \mu_4$}}\\
		$\sigma_0$ & 0.000 & 0.000 & 0.014 & 0.235 & 0.551 & -- & -- & 0.105 & 0.004 & 0.001 & 0.219 \\
		$\sigma_1$ & 0.000 & 0.000 & 0.014 & 0.116 & 0.286 & -- & -- & 0.015 & 0.001 & 0.000 & 0.117 \\
		$\sigma_2$ & 0.000 & 0.000 & 0.008 & 0.286 & 0.600 & -- & -- & 0.163 & 0.001 & 0.000 & 0.221 \\
		$\sigma_3$ & 0.000 & 0.000 & 0.025 & 0.154 & 0.497 & -- & -- & 0.007 & 0.000 & 0.000 & 0.104 \\
		\midrule
		\multicolumn{12}{l}{\textit{Panel E: $\mu = \mu_5$}}\\
		$\sigma_0$ & 0.000 & 0.000 & 0.004 & 0.014 & 0.023 & -- & -- & 0.399 & 0.424 & 0.397 & 0.311 \\
		$\sigma_1$ & 0.000 & 0.000 & 0.000 & 0.002 & 0.004 & -- & -- & -- & 0.420 & 0.474 & 0.195 \\
		$\sigma_2$ & 0.000 & 0.000 & 0.003 & 0.012 & 0.025 & -- & -- & 0.482 & 0.503 & 0.443 & 0.381 \\
		$\sigma_3$ & 0.000 & 0.000 & 0.000 & 0.000 & 0.001 & -- & -- & -- & -- & 0.370 & 0.250 \\
		\midrule
		\multicolumn{12}{l}{\textit{Panel F: $\mu = \mu_6$}}\\
		$\sigma_0$ & 0.000 & 0.000 & 0.000 & 0.004 & 0.010 & -- & -- & -- & 0.338 & 0.348 & 0.285 \\
		$\sigma_1$ & 0.000 & 0.000 & 0.000 & 0.000 & 0.000 & -- & -- & -- & -- & -- & 0.211 \\
		$\sigma_2$ & 0.000 & 0.000 & 0.000 & 0.000 & 0.003 & -- & -- & -- & -- & 0.567 & 0.357 \\
		$\sigma_3$ & 0.000 & 0.000 & 0.000 & 0.000 & 0.001 & -- & -- & -- & -- & 0.321 & 0.231 \\
		\bottomrule
	\end{tabular}
\end{table}

\begin{table}
	\caption{Empirical rejection rates and average estimated onset locations for different choices of $\eps$ with $\delta= 0.2$, $n = 1000$ and $\sigma = \sigma_2$}
	\begin{tabular}{l|rrrrr|rrrrrr}
		\toprule
		& \multicolumn{5}{c|}{Empirical Rejection Rate} & \multicolumn{6}{c}{Estimated Onset Location} \\
		$\eps \quad \setminus \quad \Delta$ & $0.8$ & $0.9$ & $1.0$ & $1.1$ & $1.2$ & $0.8$ & $0.9$ & $1.0$ & $1.1$ & $1.2$ & $\hat{\tau}_*$ \\
		\midrule
		\multicolumn{12}{l}{\textit{Panel A: $\mu = \mu_1$}}\\	
		iid & 0.002 & 0.008 & 0.025 & 0.065 & 0.135 & 0.730 & 0.718 & 0.653 & 0.652 & 0.628 & 0.655 \\
		loc\_stat & 0.010 & 0.013 & 0.022 & 0.041 & 0.070 & 0.712 & 0.672 & 0.681 & 0.663 & 0.638 & 0.417 \\
		ma & 0.003 & 0.004 & 0.016 & 0.026 & 0.035 & 0.658 & 0.661 & 0.606 & 0.587 & 0.570 & 0.626 \\
		ar & 0.000 & 0.001 & 0.001 & 0.001 & 0.001 & -- & 0.728 & 0.723 & 0.718 & 0.714 & 0.542 \\
		\midrule
		\multicolumn{12}{l}{\textit{Panel B: $\mu = \mu_2$}}\\		
		iid & 0.027 & 0.042 & 0.068 & 0.110 & 0.156 & 0.239 & 0.237 & 0.234 & 0.232 & 0.232 & 0.355 \\
		loc\_stat & 0.003 & 0.014 & 0.028 & 0.044 & 0.058 & 0.251 & 0.251 & 0.242 & 0.248 & 0.242 & 0.209 \\
		ma & 0.004 & 0.004 & 0.009 & 0.018 & 0.033 & 0.254 & 0.249 & 0.230 & 0.225 & 0.228 & 0.459 \\
		ar & 0.000 & 0.000 & 0.001 & 0.001 & 0.001 & -- & -- & 0.249 & 0.246 & 0.243 & 0.482 \\
		\midrule
		\multicolumn{12}{l}{\textit{Panel C: $\mu = \mu_3$}}\\	
		iid & 0.000 & 0.000 & 0.000 & 0.199 & 0.842 & -- & -- & -- & 0.386 & 0.380 & 0.400 \\
		loc\_stat & 0.000 & 0.000 & 0.000 & 0.080 & 0.376 & -- & -- & -- & 0.386 & 0.379 & 0.398 \\
		ma & 0.000 & 0.000 & 0.000 & 0.029 & 0.237 & -- & -- & -- & 0.382 & 0.378 & 0.398 \\
		ar & 0.000 & 0.000 & 0.000 & 0.001 & 0.008 & -- & -- & -- & 0.379 & 0.371 & 0.395 \\
		\midrule
		\midrule
		$n \quad \setminus \quad \Delta$ & $0.1$ & $0.2$ & $0.3$ & $0.4$ & $0.5$ & $0.1$ & $0.2$ & $0.3$ & $0.4$ & $0.5$ & \\
		\midrule
		\multicolumn{12}{l}{\textit{Panel D: $\mu = \mu_4$}}\\
		iid & 0.000 & 0.000 & 0.013 & 0.670 & 0.970 & -- & -- & 0.141 & 0.002 & 0.000 & 0.383 \\
		loc\_stat & 0.000 & 0.000 & 0.008 & 0.286 & 0.600 & -- & -- & 0.163 & 0.001 & 0.000 & 0.221 \\
		ma & 0.000 & 0.000 & 0.004 & 0.130 & 0.401 & -- & -- & 0.056 & 0.007 & 0.000 & 0.379 \\
		ar & 0.000 & 0.000 & 0.000 & 0.004 & 0.021 & -- & -- & -- & 0.000 & 0.000 & 0.395 \\
		\midrule
		\multicolumn{12}{l}{\textit{Panel E: $\mu = \mu_5$}}\\
		iid & 0.000 & 0.002 & 0.017 & 0.048 & 0.108 & -- & 0.450 & 0.519 & 0.477 & 0.465 & 0.622 \\
		loc\_stat & 0.000 & 0.000 & 0.003 & 0.012 & 0.025 & -- & -- & 0.482 & 0.503 & 0.443 & 0.381 \\
		ma & 0.000 & 0.000 & 0.000 & 0.001 & 0.009 & -- & -- & -- & 0.411 & 0.479 & 0.624 \\
		ar & 0.000 & 0.000 & 0.000 & 0.000 & 0.000 & -- & -- & -- & -- & -- & 0.546 \\
		\midrule
		\multicolumn{12}{l}{\textit{Panel F: $\mu = \mu_6$}}\\
		iid & 0.000 & 0.000 & 0.000 & 0.008 & 0.032 & -- & -- & -- & 0.431 & 0.371 & 0.545 \\
		loc\_stat & 0.000 & 0.000 & 0.000 & 0.000 & 0.003 & -- & -- & -- & -- & 0.567 & 0.357 \\
		ma & 0.000 & 0.000 & 0.000 & 0.003 & 0.006 & -- & -- & -- & 0.334 & 0.414 & 0.595 \\
		ar & 0.000 & 0.000 & 0.000 & 0.000 & 0.000 & -- & -- & -- & -- & -- & 0.551 \\
		\bottomrule
	\end{tabular}
\end{table}

\begin{table}
	\caption{Empirical rejection rates and average estimated onset locations for different choices of $n$ with $\delta= 0.2$, $\sigma = \sigma_2$ and locally stationary $\eps$.}
	\begin{tabular}{l|rrrrr|rrrrrr}
		\toprule
		& \multicolumn{5}{c|}{Empirical Rejection Rate} & \multicolumn{6}{c}{Estimated Onset Location} \\
		$n \quad \setminus \quad \Delta$ & $0.8$ & $0.9$ & $1.0$ & $1.1$ & $1.2$ & $0.8$ & $0.9$ & $1.0$ & $1.1$ & $1.2$ & $\hat{\tau}_*$ \\
		\midrule
		\multicolumn{12}{l}{\textit{Panel A: $\mu = \mu_1$}}\\	
		100  & 0.031 & 0.057 & 0.092 & 0.136 & 0.181 & 0.645 & 0.623 & 0.608 & 0.587 & 0.565 & 0.499 \\
		200  & 0.029 & 0.047 & 0.094 & 0.118 & 0.150 & 0.678 & 0.637 & 0.619 & 0.605 & 0.594 & 0.443 \\
		500  & 0.010 & 0.023 & 0.033 & 0.060 & 0.085 & 0.712 & 0.676 & 0.660 & 0.616 & 0.602 & 0.393 \\
		1000 & 0.010 & 0.013 & 0.022 & 0.041 & 0.070 & 0.712 & 0.672 & 0.681 & 0.663 & 0.638 & 0.417 \\
		\midrule
		\multicolumn{12}{l}{\textit{Panel B: $\mu = \mu_2$}}\\		
		100  & 0.031 & 0.044 & 0.075 & 0.101 & 0.136 & 0.314 & 0.299 & 0.341 & 0.344 & 0.338 & 0.357 \\
		200  & 0.026 & 0.039 & 0.051 & 0.071 & 0.101 & 0.283 & 0.285 & 0.300 & 0.296 & 0.285 & 0.274 \\
		500  & 0.014 & 0.024 & 0.037 & 0.053 & 0.078 & 0.257 & 0.250 & 0.248 & 0.244 & 0.252 & 0.218 \\
		1000 & 0.003 & 0.014 & 0.028 & 0.044 & 0.058 & 0.251 & 0.251 & 0.242 & 0.248 & 0.242 & 0.209 \\
		\midrule
		\multicolumn{12}{l}{\textit{Panel C: $\mu = \mu_3$}}\\	
		100  & 0.000 & 0.000 & 0.000 & 0.011 & 0.037 & -- & -- & -- & 0.377 & 0.370 & 0.403 \\
		200  & 0.000 & 0.000 & 0.003 & 0.020 & 0.104 & -- & -- & 0.379 & 0.372 & 0.373 & 0.400 \\
		500  & 0.000 & 0.000 & 0.001 & 0.051 & 0.246 & -- & -- & 0.375 & 0.383 & 0.378 & 0.398 \\
		1000 & 0.000 & 0.000 & 0.000 & 0.080 & 0.376 & -- & -- & -- & 0.386 & 0.379 & 0.398 \\
		\midrule
		\midrule
		$n \quad \setminus \quad \Delta$ & $0.1$ & $0.2$ & $0.3$ & $0.4$ & $0.5$ & $0.1$ & $0.2$ & $0.3$ & $0.4$ & $0.5$ & \\
		\midrule
		\multicolumn{12}{l}{\textit{Panel D: $\mu = \mu_4$}}\\
		100  & 0.000 & 0.000 & 0.000 & 0.011 & 0.098 & -- & -- & -- & 0.019 & 0.001 & 0.325 \\
		200  & 0.000 & 0.000 & 0.000 & 0.040 & 0.243 & -- & -- & -- & 0.002 & 0.000 & 0.283 \\
		500  & 0.000 & 0.000 & 0.006 & 0.154 & 0.464 & -- & -- & 0.022 & 0.001 & 0.000 & 0.254 \\
		1000 & 0.000 & 0.000 & 0.008 & 0.286 & 0.600 & -- & -- & 0.163 & 0.001 & 0.000 & 0.221 \\
		\midrule
		\multicolumn{12}{l}{\textit{Panel E: $\mu = \mu_5$}}\\
		100  & 0.000 & 0.000 & 0.000 & 0.000 & 0.005 & -- & -- & -- & -- & 0.568 & 0.472 \\
		200  & 0.000 & 0.000 & 0.000 & 0.008 & 0.019 & -- & -- & -- & 0.562 & 0.484 & 0.421 \\
		500  & 0.000 & 0.000 & 0.000 & 0.013 & 0.032 & -- & -- & -- & 0.476 & 0.452 & 0.383 \\
		1000 & 0.000 & 0.000 & 0.003 & 0.012 & 0.025 & -- & -- & 0.482 & 0.503 & 0.443 & 0.381 \\
		\midrule
		\multicolumn{12}{l}{\textit{Panel F: $\mu = \mu_6$}}\\
		100  & 0.000 & 0.000 & 0.000 & 0.001 & 0.005 & -- & -- & -- & 0.657 & 0.622 & 0.469 \\
		200  & 0.000 & 0.000 & 0.000 & 0.003 & 0.013 & -- & -- & -- & 0.365 & 0.406 & 0.422 \\
		500  & 0.000 & 0.000 & 0.001 & 0.004 & 0.008 & -- & -- & 0.381 & 0.468 & 0.450 & 0.347 \\
		1000 & 0.000 & 0.000 & 0.000 & 0.000 & 0.003 & -- & -- & -- & -- & 0.567 & 0.357 \\
		\bottomrule
	\end{tabular}
\end{table}